\documentclass[16pt]{amsart}
\usepackage{amssymb, latexsym, amsmath, amsfonts}

\newtheorem{thm}{Theorem}[section]
\newtheorem{cor}[thm]{Corollary}
\newtheorem{lem}[thm]{Lemma}
\newtheorem{prop}[thm]{Proposition}
\theoremstyle{definition}

\theoremstyle{remark}

\numberwithin{equation}{section}

\newcommand{\mbf}{\mathbf}
\newcommand{\ra}{\rightarrow}

\newcommand{\pa}{\partial}
\newcommand{\ov}{\overline}
\newcommand{\sm}{\setminus}
\newcommand{\ep}{\epsilon}
\newcommand{\no}{\noindent}
\newcommand{\Om}{\Omega}
\newcommand{\cal}{\mathcal}
\newcommand{\ti}{\tilde}
\newcommand{\la}{\lambda}
\newcommand{\ka}{\kappa}

\begin{document}
\title{A characterization of domains in $\mbf C^2$ with noncompact 
automorphism group}
\thanks{The author was supported by DST (India) Grant No.: 
SR/S4/MS-283/05 and in part by a grant from UGC under DSA-SAP, Phase IV} 
\author{Kaushal Verma}
\address{Department of Mathematics,
Indian Institute of Science, Bangalore 560 012, India}
\email{kverma@math.iisc.ernet.in}

\begin{abstract}
Let $D$ be a bounded domain in $\mbf C^2$ with a non-compact 
group of holomorphic automorphisms. Model domains for $D$ are obtained 
under the hypotheses that at least one orbit accumulates at a boundary 
point near which the boundary is smooth, real analytic and of finite type.
\end{abstract}

\maketitle

\section{Introduction}

\no Let $D$ be a bounded, or more generally a Kobayashi hyperbolic domain 
in $\mbf C^n$. It is known that the group of holomorphic automorphisms of 
$D$, henceforth to be denoted by ${\rm Aut}(D)$, is a real analytic Lie 
group in the compact open topology of dimension at most $n^2 + 2n$; the 
maximal value occuring only when $D$ is biholomorphically equivalent to 
the unit ball $\mbf B^n \subset \mbf C^n$. This paper addresses the 
question of determining those domains $D$ for which ${\rm Aut}(D)$ is 
non-compact. By a theorem of H. Cartan, non-compactness of ${\rm Aut}(D)$ 
is equivalent to the existence of $p \in D$ and a sequence $\{\phi_j\} \in 
{\rm Aut}(D)$ such that $\{\phi_j(p)\}$ clusters only on $\pa D$. In other 
words, there is at least one point in $D$ whose orbit under the natural 
action of ${\rm Aut}(D)$ on $D$ accumulates at the boundary of $D$. Call 
$p_{\infty} \in \pa D$ an orbit accumulation point if it is a limit point 
for $\{\phi_j(p)\}$. In this situation it is known that local data 
regarding $\pa D$ near $p_{\infty}$ provides global information about $D$, 
the protoype example of this being the Wong-Rosay theorem. Indeed, it was 
shown in \cite{Won} that a smoothly bounded strongly pseudoconvex domain 
in $\mbf C^n$ with non-compact automorphism group must be equivalent to 
$\mbf B^n$. The same conclusion was arrived at in \cite{Ros} under the 
weaker hypothesis that the boundary of $D$ is strongly pseudoconvex only 
near $p_{\infty}$. A systematic study of this phenomenon for more general 
pseudoconvex domains was initiated by Greene-Krantz and we refer the 
reader to \cite{BP1}, \cite{BP2}, \cite{BP3}, \cite{BP4}, \cite{Ber1}, 
\cite{Ber2}, \cite{Ef}, \cite{Gau}, \cite{Kim1} and \cite{Kim2} which 
provide a panoramic view of some of the known results in this direction. 
The survey articles \cite{IK} and \cite{KimK} contain an overview of the 
techniques that are used and provide a comprehensive list of relevant 
references as well. 

\medskip

The main result in \cite{Won} was generalised for pseudoconvex domains in 
$\mbf C^2$ (see \cite{BP2}and \cite{BP3} for related results) by 
Bedford-Pinchuk in \cite{BP1}. They showed that a smoothly bounded weakly 
pseudoconvex domain in $\mbf C^2$ with real analytic boundary for which 
${\rm Aut}(D)$ is non-compact must be equivalent to the ellipsoid $E_{2m} 
=\{(z_1, z_2) \in \mbf C^2: \vert z_1 \vert^2 + \vert z_2 \vert^{2m} < 
1\}$ for some integer $m \ge 1$. The corresponding local result (which 
would be the analogue of \cite{Ros}) when the boundary of $D$ is smooth 
weakly pseudoconvex and of finite type only near $p_{\infty}$ was obtained 
in \cite{Ber1}. The model domain for $D$ is then not restricted to be 
$E_{2m}$ alone as above. It turns out that $D$ is equivalent to a domain 
of the form $G = \{(z_1, z_2) \in \mbf C^2 : 2 \Re z_2 + P(z_1, \ov z_1) < 
0\}$ where $P(z_1, \ov z_1)$ is a homogeneous subharmonic polynomial of 
degree $2m$ (here $m \ge 1$ is an integer) without harmonic terms. Apart 
from translations along the imaginary $z_2$-axis, $G$ is invariant 
under the action of a one parameter subgroup of ${\rm Aut}(D)$ given by $s 
\mapsto S_s(z_1, z_2) = (\exp(s/2m) z_1, \exp(s) z_2)$ where $s \in \mbf 
R$. As $s \ra -\infty$, it can be seen that the orbit of any point 
$(z_1, z_2) \in G$ under the action of $(S_s)$ accumulates at the 
origin which lies on 
the boundary of $G$. The situation for domains in $\mbf C^2$ was clarified 
further in \cite{BP4}. It was shown that a smoothly bounded domain in 
$\mbf C^2$ with real analytic boundary must be equivalent to $E_{2m}$. 
Thus pseudoconvexity of the domain that was assumed in \cite{BP1} turned 
out to be a consequence. The purpose of this article is to propose a local 
version for the main result in \cite{BP4}.

\begin{thm}
Let $D$ be a bounded domain in $\mbf C^2$. Suppose that there 
exists a point $p \in D$ and a sequence $\{\phi_j\} \in {\rm Aut}(D)$ such 
that $\phi_j(p)$ converges to $p_{\infty} \in \pa D$. Assume that the 
boundary of $D$ is smooth real analytic and of finite type near 
$p_{\infty}$. Then exactly one of the following alternatives holds:

   \begin{enumerate}

	\item[(i)] If $\dim{\rm Aut}(D) = 2$ then either
	  \begin{itemize}
		\item $D \backsimeq \cal D_1 = \{(z_1, z_2) \in \mbf C^2 : 
	2 \Re z_2 + P_1(\Re z_1) < 0\}$ where $P_1(\Re z_1)$ is a 
	polynomial that depends on $\Re z_1$, or

		\item $D \backsimeq \cal D_2 = \{(z_1, z_2) \in \mbf C^2 : 
	2 \Re z_2 + P_2(\vert z_1 \vert^2) < 0\}$ where $P_2(\vert z_1 
	\vert^2)$ is a polynomial that depends on $\vert z_1 \vert^2$, or

		\item $D \backsimeq \cal D_3 =
	\{(z_1, z_2) \in \mbf C^2 : 2 \Re z_2 + 
	P_{2m}(z_1, \ov z_1) < 0\}$ where $P_{2m}(z_1, \ov z_1)$ is a 
	homogeneous polynomial of degree $2m$ without harmonic terms.
	  \end{itemize}

	\item[(ii)] If $\dim{\rm Aut}(D) = 3$ then $D \backsimeq \cal D_4 
	= \{(z_1, z_2) \in \mbf C^2 : 2 \Re z_2 + (\Re 
	z_1)^{2m} < 0\}$ for some integer $m \ge 2$.

	\item[(iii)] If $\dim{\rm Aut}(D) = 4$ then $D \backsimeq \cal D_5 
	= \{(z_1, z_2) \in \mbf C^2 : 2 \Re z_2 + 
	\vert z_1 \vert^{2m} < 0\}$ for some integer $m \ge 2$. Note that 
	$\Om_3 \backsimeq E_m$.

	\item[(vi)] If $\dim{\rm Aut}(D) = 8$ then $D \backsimeq \cal D_6 
	= \mbf B^2$ the unit ball in $\mbf C^2$.
     
    \end{enumerate}
The dimensions $0, 1, 5, 6, 7$ cannot occur with $D$ as above.
\end{thm}

\noindent While $\cal D_4, \cal D_5, \cal D_6$ are evidently pseudoconvex, 
no such claim is being made about any of the model domains in case $\dim 
{\rm Aut}(D) = 2$. Indeed, pseudoconvexity is not always assured as the 
following example from \cite{FIK} shows. Consider the bounded domain
\[
\Om = \Big\{ (z_1, z_2) \in \mbf C^2 : \vert z_1 \vert^2 + \vert z_2 
\vert^4 + 8 \vert z_1 - 1 \vert^2 \Big( \frac{z_2^2}{z_1 - 1} + \frac{\ov 
z_2^2}{\ov z_1 - 1} - \frac{3}{2} \frac{\vert z_2 \vert^2}{\vert z_1 - 
1\vert} \Big) < 1 \Big\}
\]
and set $a_j = 1 - 1/j$ where $j \ge 1$. The sequence 
\[
\phi_j(z_1, z_2) = \Big( \frac{z_1 - a_j}{1 - a_j z_1}, \frac{(1 - 
a_j^2)^{1/4}z_2}{1 - a_j z_1} \Big) \in {\rm Aut}(\Om)
\]
converges uniformly on compact subsets of $\Om$ to $(-1, 0) \in \pa \Om$. 
Note that the boundary $\pa \Om$ near $(-1, 0)$ is smooth real analytic 
and of finite type. The mapping $f(z_1, z_2) = \Big( (z_1 + 1)/(z_1 - 1), 
\sqrt{2} z_2/ \sqrt{z_1 - 1} \Big)$ biholomorphically transforms $\Om$ 
onto its unbounded realisation given by
\[
\Om' = f(\Om) = \Big\{ (z_1, z_2) \in \mbf C^2 : 2 \Re z_1 + \frac{1}{4} 
\vert z_2 \vert^4 + 2 \Big( z_2^2 + \ov z_2^2 - \frac{3}{2} \vert z_2 
\vert^2 \Big)^2 < 0 \Big\}.
\]
The terms involving $z_2$ are homogeneous of order $4$ and therefore 
$\Om'$ admits automorphisms of the form $S_s(z_1, z_2) = (\exp(s) z_1, 
\exp(s/4) z_2)$ for $s \in \mbf R$. Also, $\Om'$ is evidently invariant 
under translations in the imaginary $z_1$-direction and this shows that 
$\dim {\rm Aut}(\Om') \ge 2$. If $\dim{\rm Aut}(\Om') > 2$, then the above 
theorem shows that $\Om'$ and hence $\Om$ must be pseudoconvex. This 
however does not hold; for instance, the boundary $\pa \Om'$ is not 
pseudoconvex near the point $(-3/4, 1) \in \pa \Om'$ as a calculation of 
the Levi form of $\pa \Om'$ shows.  

\medskip

\no The two principal techniques used in \cite{BP1}, 
\cite{BP4} and \cite{Ber1} are scaling and a careful analysis of a 
holomorphic tangential vector field of parabolic type. Moreover, the 
hypotheses that $\pa D$ is globally smooth real analytic 
(in \cite{BP1} and \cite{BP4}) and that $p_{\infty}$ is a smooth weakly 
pseudoconvex finite type boundary point in \cite{Ber1} are used in an 
important way. These hypotheses are not assumed in the theorem above 
and hence the two techniques have to be supplemented with information 
regarding the type of orbits that are possible when 
$\dim{\rm Aut}(D) = 3, 4$. This leads to a classification of $D$ that 
depends on the dimension of ${\rm Aut}(D)$.

\medskip

As mentioned above it is known that $0 \le \dim{\rm Aut}(D) \le n^2 + 2n 
= 8$ (since $n = 2$) both inclusive. By a result of W. Kaup (see 
\cite{Ka}) ${\rm Aut}(D)$ acts transitively on $D$ if $\dim{\rm Aut}(D) 
\ge 5$. Since the boundary of $D$ is smooth real analytic and of finite 
type near $p_{\infty}$, it can be shown that there are strongly 
pseudoconvex points arbitrarily close to $p_{\infty}$. The Wong-Rosay 
theorem now shows that $D \backsimeq \cal D_6 = \mbf B^2$. Hence the 
remaining possibilities are $0 \le \dim{\rm Aut}(D) \le 4$. An initial 
scaling of $D$, as in \cite{BP4}, with respect to the sequence $p_j := 
\phi_j(p)$ shows that $D$ is equivalent to a domain of the form
\[
G_p = \{(z_1, z_2) \in \mbf C^2 : 2 \Re z_2 + P(z_1, \ov z_1) < 0\}
\]
where $P(z_1, \ov z_1)$ is a polynomial without harmonic terms. Note that 
$G_p$ is invariant under translations $(z_1, z_2) \mapsto (z_1, z_2 + it), 
t \in \mbf R$ and hence $\dim{\rm Aut}(D) \ge 1$. The possibility that it 
equals one is ruled out by arguments similar to those in \cite{Ber1} and 
\cite{BP4}. Thus $\dim{\rm Aut}(D) = 2, 3$ or $4$. When $\dim{\rm Aut}(D) 
= 2$ the arguments used in \cite{O} can be adapted to show that $D$ is 
equivalent to $\cal D_1, \cal D_2$ or $\cal D_3$. Finally when the 
dimension is $3$ or $4$, the 
classification obtained in \cite{I1}, \cite{I2} is used. In case 
$\dim{\rm Aut}(D) = 4$ techniques of analytic continuation of germs of 
holomorphic mappings as in \cite{Sha} are used to show that 
$D \backsimeq \cal D_5$. There are many more possibilities for $D$ in 
case $\dim{\rm Aut}(D) = 3$ as \cite{I1} shows. Several on that list have 
Levi flat orbits foliated by copies of the unit disc in the complex plane. 
However, it is shown that $D$ as in the theorem cannot admit Levi flat 
orbits. A further reduction is obtained by studying the Lie algebra of 
${\rm Aut}(D)$ of some of the examples and showing that $D$ is forced to 
be equivalent to a tube domain. This uses ideas from \cite{O} and 
\cite{KS}. This does not exhaust the list; the remaining possibilities 
are ruled out by using arguments that were developed to study the boundary 
regularity of holomorphic mappings between domains in $\mbf C^n$. The 
conclusion is that $D \backsimeq \cal D_4$.

\medskip

\no The author is indebted to A. V. Isaev for various very helpful  
comments on an earlier incarnation of this article. They have contributed 
in pointing out an error, and have helped in improving and clarifying the 
exposition in several places.

\section{The dimension of ${\rm Aut}(D)$ is at least two}

\no Let $D$ be as in theorem 1.1 and $p \in D$ and $\phi_j \in {\rm 
Aut}(D)$ are such that $\phi_j(p) \ra p_{\infty} \in \pa D$. Let $U$ be an 
open neighbourhood of $p_{\infty}$, fixed henceforth, such that the 
boundary of $D$ is smooth real analytic and of finite type in a 
neighbourhood of $\ov U$ and that $U \cap \pa D$ is defined by 
$\{\rho(z, \ov z) = 0\}$ with $d\rho \not= 0$ on $U \cap \pa D$ for some 
$\rho \in C^{\omega}(U)$. For every $a \in U \cap \pa D$, $T^c_a(\pa D)$ 
the complex tangent space at $a$ is spanned by the non-vanishing vector 
field $X_a = (- \pa \rho/\pa z_2 (a), \pa \rho/\pa z_1 (a))$. The function
\[
\mathcal L_{\rho}(z) = \mathcal L_{\rho}(z, X_z)
\]
where the term on the right is the Levi form associated with the defining 
function $\rho$ evaluated at $z \in U \cap \pa D$ and $X_z$, is real 
analytic on $U \cap \pa D$. This function provides a decomposition of $U 
\cap \pa D$ which will be useful in this context and we recall its salient 
features from \cite{DP}.

\medskip

Let $T$ be the zero locus of $\mathcal L_{\rho}(z)$ in $U \cap \pa D$. 
Then $T$ admits a semi-analytic stratification as $T = T_0 \cup T_1 \cup 
T_2$ where $T_j$ is a locally finite union of smooth real analytic 
submanifolds of $U \cap \pa D$ of dimension $j = 0, 1, 2$ respectively. 
Denote by $(U \cap \pa D)^{\pm}_s$ the set of strongly pseudoconvex (resp. 
strongly pseudoconcave) points on $U \cap \pa D$. Let $(U \cap \pa 
D)^{\pm}$ be the relative interior, taken with respect to the subspace 
topology on $U \cap \pa D$, of the closure of $(U \cap \pa D)^{\pm}_s$ in 
$\ov U \cap \pa D$. Then $(U \cap \pa D)^{\pm}$ is the set of weakly 
pseudoconvex (resp. weakly pseudoconcave) points on $U \cap \pa D$ and the 
border
\[
\cal B = (U \cap \pa D) \sm ((U \cap \pa D)^+ \cup (U \cap \pa D)^-) 
\subset T
\]
separates $(U \cap \pa D)^+$ and $(U \cap \pa D)^-$. The stratification of 
$T$ can be refined in such a way that the two dimensional strata become 
maximally totally real and the order of vanishing of $\cal L_{\rho}(z)$ 
along them is constant. The same notation $T_j$ will be retained to denote 
the various strata after the refinement. It was shown in \cite{DF} that  
if the order of vanishing of $\cal L_{\rho}(z)$ along a two dimensional 
stratum, say $S$ is odd then $S \subset \hat D$, the envelope of 
holomorphy of $D$, while if it is even then $S \subset (U \cap 
\pa D)^+$ or $S \subset (U \cap \pa D)^- \subset \hat D$. This discussion 
holds for a germ of a smooth real analytic, finite type hypersurface in 
$\mbf C^2$ and is independent of any assumptions on ${\rm Aut}(D)$. The 
question now is to identify where $p_{\infty}$ lies in this decomposition 
of $U \cap \pa D$. First observe that $p_{\infty} \in T$ as otherwise it 
is either in $(U \cap \pa D)^+_s$ or $(U \cap \pa D)^-_s$. In the former 
case, the Wong-Rosay theorem shows that $D \backsimeq \mbf B^2$ while the 
latter possibility does not arise; indeed it was observed by Greene-Krantz 
(see \cite{GK}) that no point in $\hat D$ can be a boundary orbit 
accumulation point. If $p_{\infty} \in (U \cap \pa D)^+$ then all 
possible model domains are known by \cite{Ber1} while $p_{\infty} \notin 
(U \cap \pa D)^-$ as all weakly pseudoconcave points are contained in 
$\hat D$. It therefore follows that $p_{\infty} \in \cal B$ which means 
that $\cal L_{\rho}(z)$ must change sign in arbitrarily small 
neighbourhods of $p_{\infty}$. Second, if $p_{\infty}$ belongs to a two 
dimensional stratum of $T$ then the discussion above shows that $p_{\infty}$ 
belongs either to $(U \cap \pa D)^+$ or $\hat D$. As before the former case 
is handled by \cite{Ber1} while the latter does not happen by the 
Greene-Krantz observation. It is thus possible to assume 
without loss of generality that $p_{\infty} \in \cal B \cap (T_0 \cup 
T_1)$. This will be the standing assumption henceforth. In sections $4$ 
and $5$ it will be shown that $\pa D$ is pseudoconvex near $p_{\infty}$.

\begin{lem}
In the situation described above, there exists at least one stratum in 
$T_2$, say $S$ that contains $p_{\infty}$ in its closure and for 
which the order of vanishing of $\cal L_{\rho}(z)$ along it is odd. In 
particular $S \subset \hat D$.
\end{lem}
\begin{proof}
If possible let $V \subset U$ be a neighbourhood of $p_{\infty}$ such that 
$V \cap \pa D$ contains no two dimensional stratum of $T$. This implies 
that $V \cap T$ does not separate $V \cap \pa D$. Choose $a, b \in V \cap 
\pa D$ such that $\cal L_{\rho}(a) > 0$ and $\cal L_{\rho}(b) < 0$ and 
join them by a path $\gamma(t)$ parametrised by $[0, 1]$ with the end 
points corresponding to $a, b$ and which lies entirely in $(V \cap \pa D) 
\sm T$. The function $t \mapsto \cal L_{\rho}(\gamma(t))$ changes sign and 
hence $\cal L_{\rho}(\gamma(t_0)) = 0$ for some $0 < t_0 < 1$ which 
means that $\gamma(t_0) \in T$. This is a contradiction. Hence $T_2$ is 
non-empty near $p_{\infty}$.

\medskip

Now fix a ball $B_{\ep} = B(p_{\infty}, \ep)$ for some $\ep > 0$. Then 
$B_{\ep} \sm T$ has finitely many components each of which contains 
$p_{\infty}$ in its closure and the sign of $\cal L_{\rho}(z)$ does not 
change within each component. Let $S_1, S_2, \ldots, S_k$ ($k \ge 1$) be 
all the two dimensional strata each of which contains $p_{\infty}$ in its 
closure. Note that the union of the $S_j$'s is contained in the union of 
the boundaries of the various components of $B_{\ep} \sm T$. Choose $a_j 
\in S_j$ for all $1 \le j \le k$ and let $\sigma(t) : [0, 1] \ra (B_{\ep} 
\cap \pa D) \sm (T_0 \cup T_1)$ be a closed path that contains all the 
$a_j$'s. Let $0 \le t_j \le 1$ be such that $\sigma(t_j) = a_j$ for all $1 
\le j \le k$. The function $t \mapsto \cal L_{\rho}(\sigma(t))$ then has 
zeros at least at all the $t_j$'s. If $S_j \subset (U \cap \pa D)^+$ for 
all $1 \le j \le k$, then $\cal L_{\rho}(\sigma(t))$ is non-negative on 
$[0, 1]$ and hence $p_{\infty}$ is a weakly pseudoconvex point. This is 
not possible. There is therefore at least one value, say $j_0$ for which 
$S_{j_0} \subset (U \cap \pa D)^-$. This implies that $\cal 
L_{\rho}(\sigma(t))$ changes sign on $[0, 1]$ which in turn shows the 
existence of a two dimensional stratum from the collection $S_1, S_2, 
\ldots, S_k$, say $S_{i_o}$ with the property that $\cal L_{\rho}(z)$ 
changes sign near each point on it. Thus the order of vanishing of $\cal 
L_{\rho}(z)$ along $S_{i_0}$ must be odd. The same argument works when 
some of the $S_j$'s (though not all) are contained in $(U \cap \pa D)^-$ 
and the remaining in $(U \cap \pa D)^+$. The case when all the $S_j$'s 
belong to $(U \cap \pa D)^-$ does not arise because then $p_{\infty} \in 
(U \cap \pa D)^-$ which cannot hold by the Greene-Krantz observation 
above.
\end{proof}

\begin{lem}
The sequence $\{\phi_j\} \in {\rm Aut}(D)$ converges uniformly on compact 
subsets of $D$ to the constant map $\phi(z) \equiv p_{\infty}$ for all $z 
\in D$.
\end{lem}
\begin{proof}
The family $\{\phi_j\}$ is normal and hence admits a 
subsequence that converges in the compact open topology on $D$ to $\phi : 
D \ra \ov D$ with $\phi(p) = p_{\infty}$. It then follows from a 
theorem of H. Cartan (cf. \cite{N}) that $\phi(D) \subset \pa D$. Choose 
$r > 0$ small enough so that $\phi : B(p, r) \ra U$ is well defined. Let 
$k > 0$ be the maximal rank of $\phi$ which is attained on the complement 
of an analytic set $A \subset D$. Two cases arise now; first if $p \in D 
\sm A$, choose a small ball $B(p, \ep)$ which does not intersect $A$. The 
image $\phi(B(p, \ep))$ is then a germ of a complex manifold of 
dimension $k$ that is contained in $U \cap \pa D$. This cannot happen 
unless $k = 0$. Second, if $p \in A$ choose $q \in B(p, r) \sm A$. The 
rank of $\phi$ is constant near $q$ and hence the image of a small enough 
neighbourhood of $q$ under $\phi$ is a germ of a complex manifold of 
dimension $k$ that is contained in $U \cap \pa D$. Again this is not 
possible. Since $A$ does not separate $B(p, r)$ it follows that $\phi$ is 
constant on $B(p, r) \sm A$, therefore on $B(p, r)$ and hence everywhere 
on $D$.
\end{proof}

\noindent {\it Remark:} It is now possible to conclude (see for 
example \cite{Kr}) that $D$ is simply 
connected. Indeed if $\gamma$ is a loop in $D$ then for $j$ large enough 
$\phi_j(\gamma)$ is a loop in $U \cap D$ by the above lemma. But $U \cap 
D$ is simply connected if $U$ is small enough and so $\phi_j(\gamma)$ and 
hence $\gamma$ (since $\phi_j \in {\rm Aut}(D)$) are both trivial loops. 
This will be useful later.

\medskip

The domain $D$ can now be scaled using the base point $p$ and the sequence 
$\{\phi_j\} \in {\rm Aut}(D)$. The transformations used in this process 
are the ones in \cite{BP4} and are briefly described as follows. First 
note that for $j$ large there exists a unique point $\ti p_j \in U \cap \pa D$ 
such that ${\rm dist}(\phi_j(p), U \cap \pa D) = \vert \ti p_j - \phi_j(p) 
\vert$. Next translate $p_{\infty}$ to the origin and rotate axes so that 
the defining function $\rho(z)$ takes the form
\begin{equation}
\rho(z) = 2 \Re z_2 + \sum_{k, l} c_{kl}(y_2) z_1^k \ov z_1^l
\end{equation}
where $c_{00}(y_2) = O(y_2^2)$ and $c_{10}(y_2) = \ov c_{01}(y_2) = 
O(y_2)$. Let $m < \infty$ be the $1$-type of $\pa D$ at the origin. It 
follows that there exist $k, l$ both at least one and $k + l = m$ for 
which $c_{kl}(0) \not= 0$ and $c_{kl}(0) = 0$ for all $k + l < m$. The 
pure terms in (2.1) up to order $m$ can be removed by a polynomial 
automorphism of the form
\begin{equation}
(z_1, z_2) \mapsto (z_1, z_2 + \frac{1}{2} \sum_{k \le m}c_{k0}(0)z_1^k).
\end{equation}
Let $\psi^j_{p, 1}(z) = z - \ti p_j$ so that $\psi^j_{p, 1}(\ti p_j) = 0$. 
Next let $\psi^j_{p, 2}(z)$ be a unitary transformation that rotates the 
outer real normal to $\psi^j_{p, 1}(U \cap \pa D)$ at the origin and makes 
it the real $z_2$-axis. The defining function for $\psi^j_{p, 2} \circ 
\psi^j_{p, 1}(U \cap D)$ near the origin is then of the form
\[
\rho_j(z) = 2 \Re z_2 + \sum_{k, l \ge 0}c^j_{kl}(y_2) z_1^k \ov z_1^l
\]
with the same normalisations on $c^j_{00}(y_2)$ and $c^j_{10}(y_2)$ as in 
(2.1). Since $\ti p_j \ra 0$ it follows that both $\psi^j_{p, 1}$ and 
$\psi^j_{p, 2}$ converge to the identity mapping uniformly on compact 
subsets of $\mbf C^2$. The $1$-type of $\psi^j_{p, 2} \circ \psi^j_{p, 
1}(U \cap \pa D)$ is at most $m$ for all large $j$ and an automorphism of 
the form (2.2) will remove all the pure terms up to order $m$ from 
$\rho_j(z)$. Call this $\psi^j_{p, 3}$. Finally $\phi_j(p)$ is on the 
inner real normal to $U \cap \pa D$ at $\ti p_j$ and it follows that 
$\psi^j_{p, 2} \circ \psi^j_{p, 1}(\phi_j(p)) = (0, - \delta_j)$ for some 
$\delta_j > 0$. Furthermore the specific form of (2.2) shows that this is 
unchanged by $\psi^j_{p, 3}$. Let $\psi^j_{p, 4}(z_1, z_2) = (z_1/\ep_j, 
z_2/\delta_j)$ where $\ep_j > 0$ is chosen in the next step. The defining 
function for $\psi^j_p(U \cap D)$ near the origin, where $\psi^j_p = 
\psi^j_{p, 4} \circ \psi^j_{p, 3} \circ \psi^j_{p, 2} \circ \psi^j_{p, 
1}$, is given by
\[
\rho_{j, p}(z) := \frac{1}{\delta_j} \rho_j(\ep_j z_1, \delta_j z_2) = 2 
\Re z_2 + \sum_{k, l} \ep_j^{k + l} \delta_j^{-1} c^j_{kl}(\delta_j y_2) 
z_1^k \ov z_1^l.
\]
Note that $\psi^j_p \circ \phi_j(p) = (0, -1)$ for all $j$. The choice of 
$\ep_j$ is determined by enforcing 
\[
\max \{ \vert \ep_j^{k + l} \delta_j^{-1} c^j_{kl}(0) \vert ; 
\;k + l \le m \} = 1 
\]
for all $j$. In 
particular $\{ \ep_j^m \delta_j^{-1} \}$ is bounded and by passing to a 
subsequence it follows that $\rho_{j, p}(z)$ converges to $2 \Re z_2 + 
P(z_1, \ov z_1)$. Therefore the domains $\psi^j_p(U \cap D)$ converge to 
\begin{equation}
G_p = \{(z_1, z_2) \in \mbf C^2 : 2 \Re z_2 + P(z_1, \ov z_1) < 0\}
\end{equation}
in the sense that every compact $K \subset G_p$ is eventually compactly 
contained in $\psi^j_p(U \cap D)$ and conversely every compact $K \subset 
\psi^j_p(U \cap D)$ for all large $j$ is compact in $G_p$. The polynomial 
$P(z_1, \ov z_1)$ is of degree at most $m$ and does not have any 
harmonic terms. The family $g_{j, p} = (\psi^j_p \circ \phi_j)^{-1} : 
\psi^j_p(U \cap D) \ra D$ is normal and the arguments in \cite{BP4} 
can be applied in this local setting as well to show that a subsequence 
converges to a biholomorphic mapping $g_p : G_p \ra D$. $G_p$ is 
invariant under the one parameter group of translations $T_t(z_1, z_2) = 
(z_1, z_2 + it), \;t \in \mbf R$ and hence the dimension of ${\rm Aut}(D)$ 
is at least one. For brevity we shall write $g, G$ in place of $g_p, G_p$ 
respectively.

\medskip

The holomorphic vector field corresponding to the action of $T_t$ is $i\; 
\pa /\pa z_2 = i/2\; \pa /\pa x_2 + 1/2 \;\pa /\pa y_2$. Then $\cal X = 
g_*(i \;\pa / \pa z_2)$ is a holomorphic vector field on $D$ whose real 
part $\Re \cal X = (\cal X + \ov {\cal X})/2$ generates the one parameter 
group $L_t = g \circ T_t \circ g^{-1} = \exp(t \;\Re \cal X) \in {\rm 
Aut}(D)$. In general the Lie algebra $\mathfrak{g}(M)$ of ${\rm Aut}(M)$, 
where $M$ is a Kobayashi hyperbolic complex manifold of dimension $n$, 
consists of real vector fields, i.e., those of the form 
\[
\sum_{j \le n} \Big( a_j(z) \pa / \pa z_j + \ov {a_j(z)} \pa / \pa \ov z_j 
\Big)
\]
where the $z_j$ are local coordinates and the coefficients $a_j(z)$ are 
holomorphic. Such fields are clearly determined by their $(1, 0)$ 
components and therefore it suffices to indicate only these components. 
This convention shall be followed everywhere in the sequel.

\begin{prop}
The group $(L_t)$ induces a local one parameter group of holomorphic 
automorphisms of a neighbourhood of $p_{\infty}$ in $\mbf C^2$. In 
particular $\cal X$ extends as a holomorphic vector field near 
$p_{\infty}$.
\end{prop}

It is possible to choose sufficiently small neighbourhoods 
$p_{\infty} \in U_1 \subset U_2$ with $U_2$ relatively compact in $U$ 
such that for each $w \in U_1$, the associated Segre variety 
$Q_w = \{z \in U_2: \rho(z, \ov w) = 0\}$ is a closed complex hypersurface 
in $U_2$. Let $\cal S (U_1, U_2)$ be the aggregate of all Segre varieties 
and $\la : U_1 \ra \cal S (U_1, U_2)$ given by $\la(w) = Q_w$. It is 
known that (see \cite{DW}, \cite{DF1}) that $\cal S (U_1, U_2)$ admits 
the structure of a finite dimensional complex analytic set and that 
$\la$ is locally an anti-holomorphic finite-to-one branched covering. 
First assume that $p_{\infty} \in T_1$. Therefore by shrinking $U$ it 
follows that $p_{\infty}$ lies on an embedded real analytic arc, which 
will still be denoted by $T_1$ and $(U \cap \pa D) \sm T_1$ consists 
either of weakly pseudoconvex, finite type points or those that belong 
to $\hat D$. Furthermore $T_1$ admits a complexification, denoted by 
$T^{\mbf C}_1$, which is a closed, smooth one dimensional analytic set in 
$U$ (shrink $U$ further if needed) and since $\la$ is a finite map near 
$p_{\infty}$, there are only finitely many points in $U$ whose Segre 
varieties coincide with $T^{\mbf C}_1$. To start with we will then suppose 
that $Q_{p_{\infty}} \not= T^{\mbf C}_1$. The remaining cases when 
$Q_{p_{\infty}} = T^{\mbf C}_1$ or when $p_{\infty} \in T_0$ are similar 
and will need the knowledge of the conclusions obtainable in case 
$Q_{p_{\infty}} \not= T^{\mbf C}_1$. Now choose coordinates centered at 
$p_{\infty}$ so that $p_{\infty} = 0$ and the defining function $\rho(z)$ 
near the origin takes the form
\[
\rho(z) = 2 \Re z_2 + o(\vert z \vert)
\]
for $z \in U$. For $w = (w_1, w_2) \in U$, it reflection in $U \cap \pa 
D$, denoted by ${}^{\ka}w$, is the unique point whose first coordinate is 
$w_1$ and which lies on $Q_w$, i. e., ${}^{\ka}w = (w_1, {}^{\ka}w_2)$ and 
$\rho({}^{\ka}w, \ov w) = 0$. The map $\ka(w) = {}^{\ka}w$ is a real 
analytic diffeomorphism near the origin that depends on the choice of a 
coordinate system. Finally a word about notation: for a neighbourhood 
$\Omega$ of $p_{\infty} = 0$, let $\Omega^{\pm} = \{z \in \Omega: \pm 
\rho(z) > 0\}$ and ${}_{{}^{\ka}w}Q_w$ will denote the germ of $Q_w$ at 
${}^{\ka}w \in Q_w$.

\medskip

For $\eta > 0$ small and $\vert t \vert < \eta$, and a sufficiently small 
pair of neighbourhoods $U_1 \subset U_2$ that contain the origin define
\[
V^+_t = \{(w, \ti w) \in U^+_1 \times U^+_1 : L_t(Q_w \cap D) \supset 
{}_{{}^{\ka} \ti w}Q_{\ti w}\} \; {\rm and} \; V^-_t = \{(w, \ti w) \in 
U^-_1 \times U^-_1 : Q_{L_t(w)} = Q_{\ti w}\}.
\]
For each fixed $t \in (-\eta, \eta)$, $V^+_t$ is non-empty as there are  
boundary points arbitrarily close to the origin that belong to $\hat D$ 
and hence $L_t$ extends holomorphically across such points and the 
invariance property of Segre varieties (which is the defining condition 
for $V^+_t$) holds. Moreover $V^+_t$ is locally complex analytic near each 
of its points being the graph of the correspondence $w \mapsto \la^{-1} 
\circ L_t \circ \la(w)$, contains the graph of the extension of $L_t$ near 
points in $\hat D$ and is pure two dimensional after removing all 
components of lower dimension if any. On the other hand $V^-_t$ is the 
graph of the correspondence $w \mapsto \la^{-1} \circ \la \circ L_t(w)$ 
and hence it contains the graph of $L_t$ over a non-empty open set in 
$U^-_1$. However, note that the projection from $V^{\pm}_t$ to the first 
factor $U^{\pm}_1$ is not known to be globally proper. It is locally 
proper near each point in $V^{\pm}_t$ since $\la$ is a finite map. 
Evidently $V^-_t$ is also two dimensional and by their construction it 
follows that for each fixed $t \in (-\eta, \eta)$ the locally analytic 
sets $V^+_t, V^-_t$ can be glued together near all points on $U_1 \cap \pa D$ 
across which $L_t$ holomorphically extends. Let $V_t$ denote the locally 
complex analytic set that is obtained from $V^{\pm}_t$ in this manner.

\begin{lem}
Suppose that $Q_0 \not= T^{\mbf C}_1$. Then it is possible to choose $\eta 
> 0$ and $U_1 \subset U_2$ small enough so that $V^+_t \subset U^+_1 
\times U^+_1$ is a closed complex analytic set for all $\vert t \vert < 
\eta$.
\end{lem}

\begin{proof}
There exists a neighbourhood basis of pairs $U_1 \subset U_2$ at the 
origin such that $\la : U_1 \ra \cal S(U_1, U_2)$ is proper and 
$\ka(U_1)$ is compactly contained in $U_2$. For such a pair and $w \in 
U^+_1$ define
\[
R'_t(w) = \{\ti w \in U^+_1 : L_t(Q_w \cap D) \supset {}_{{}^{\ka} \ti 
w}Q_{\ti w} \}
\]
and
\[
R_t(w) = \{q \in Q_w \cap \ov U_2 \cap \ov D : L_t(q) = {}^{\ka} \ti w \; 
{\rm for \; some} \; \ti w \in R'_t(w)\}.
\]
Note that $L_0$ is the identity and hence $R'_0(w) = \la^{-1} \circ 
\la(w)$ while $R_0(w) = \ka \circ \la^{-1} \circ \la(w)$. By the choice of 
$U_1, U_2$ it follows that $R_0(w)$ is at a positive distance from $\pa 
U_2$ uniformly for all $w \in U^+_1$. Now let $U_{1, j}$ shrink to the 
origin and $t_j \ra 0$ and suppose that there are points $w_j \in U^+_{1, 
j}$ for which $R_{t_j}(w_j)$ has points that cluster at $\pa U_2$. So let 
$q_j \in R_{t_j}(w_j)$ be such that $L_{t_j}(q_j) = {}^{\ka} \ti w_j$ and 
$q_j \ra q_0 \in \pa U_2$. Then $w_j, {}^{\ka} \ti w_j \ra 0$ and since 
$Q_{w_j} \ra Q_0$ it follows that $q_0 \in Q_0 \cap \pa U_2 \cap \ov D$. 
Two cases arise; first, if $q_0 \in Q_0 \cap \pa U_2 \cap D$, then since 
$L_t$ converges to $L_0$ uniformly on compact subsets of $D$, it follows 
that ${}^{\ka} \ti w_j = L_{t_j}(q_j) \ra L_0(q_0) = q_0 \not= 0$ and this 
contradicts the fact that ${}^{\ka} \ti w_j \ra 0$. The remaining 
possibility is that $q_0 \in Q_0 \cap \pa U_2 \cap \pa D$. Note that $Q_0 
\cap \pa U_2 \cap \pa D$ consists either of weakly pseudoconvex, finite 
type points or those that belong to $\hat D$. In case $q_0 \in \hat D$
each $L_t$, $t \in \mbf R$ extends to a uniform neighbourhood of $q_0$, 
the extensions being equicontinuous there, and the same argument as above 
shows that this leads to a contradiction. If $q_0$ is a weakly 
pseudoconvex, finite type point then there are local plurisubharmonic peak 
functions near $q_0$ and this can be used (see \cite{Ber1} or \cite{V}) to 
show that for small $\vert t \vert$, each $L_t$ extends to a uniform 
neighbourhood of $q_0$. Again this leads to the same contradiction as 
above. This reasoning shows that $V^+_t$ is closed for small $\vert t 
\vert$. Indeed for $t$ fixed, the main obstacle that possibly prevents 
$V^+_t$ from being closed is that its defining property may cease to hold 
in the limit. This happens exactly when points in $R_t(w)$ get arbitrarily 
close to $\pa U_2$. But it has been shown that this does not happen for a 
uniform choice of $U^+_1$ and $\vert t \vert < \eta$ for small enough 
$\eta > 0$.
\end{proof}

\no {\it Proof of Proposition 2.3.} This will be divided into three parts. 
First, it is shown that $(L_t)$ induces a local one parameter group of 
holomorphic automorphisms near $p_{\infty}$ under the assumption that 
$Q_{p_{\infty}} \not= T^{\mbf C}_1$. This is then used to show that lemma 
2.4 holds even when $Q_{p_{\infty}} = T^{\mbf C}_1$ or $p_{\infty} \in 
T_0$. Finally this in turn is used to show that $(L_t)$ induces a local 
one parameter group of holomorphic automorphisms near $p_{\infty}$ even 
when $Q_{p_{\infty}} = T^{\mbf C}_1$ or $p_{\infty} \in T_0$.

\medskip

 The real analyticity of $T_1$ implies that there exists a non-negative, 
strongly plurisubharmonic function $\tau(z)$ near $p_{\infty} = 0$ whose 
zero lous is exactly $U_1 \cap T_1$. Indeed $T_1$ can be locally 
straightened and so there are coordinates centered at the origin in which 
$U_1 \cap T_1$ coincides with the imaginary $z_1$-axis. The function 
$\tau(z) = (\Re z_1)^2 + \vert z_2 \vert^2$ is then the desired candidate. 
The sub-level sets $\Om_r = \{z \in U_1: \tau(z) < r\}$ are strongly 
pseudoconvex tubular neighbourhoods of $U_1 \cap T_1$. Fix $r > 0$ so 
small that $U_1 \sm \ov \Om_r$ is connected and such that the set of 
points $(c_1, c_2) \in U^+_1$ for which the slice $\{z_1 = c_1\} \cap U_1$ 
does not intersect $\ov \Om_r$ is non-empty and open. On the other hand, 
note that the slice $\{z_1 = c_1\} \cap \ov \Om_r$, if non-empty, is 
compactly contained in $\{z_1 = c_1\} \cap U_1$. By lemma 2.4, $V^+_t$ is 
a closed complex analytic set in $U^+_1 \times U^+_1$ for $\vert t \vert 
< \eta$. Define
\[
\ti V^+_t = \{(w, \ti w) \in U^+_1 \times U^+_1 : L_t({}_{{}^{\ka}w}Q_w) 
\subset Q_{\ti w} \cap D\}.
\]
Since $(L_t)^{-1} = L_{-t}$ it follows that $\ti V^+_t = V^+_{-t}$ for all 
$t$. Moreover $L_t$ extends holomorphically across points $w$ that are 
close to those in $\hat D$ and satisfies
\[
L_t({}_{{}^{\ka}w}Q_w) \subset Q_{L_t(w)} \cap D.
\]
This shows that both $V^+_t$ and $V^+_{-t}$ contain the graph of $L_t$ 
over an open set in $U^+_1$ and it follows that they coincide. Thus we may 
write $V^+_{\vert t \vert}$ in place of $V^+_t$ for $\vert t \vert < \eta$ 
and take the defining condition to be that for $\ti V^+_t$.

\medskip

Let $w_0 \in U^+_1$ have the following property: it is possible to choose 
a ball $B(w_0, \delta)$ (for some positive $\delta$) compactly contained 
in $U^+_1$ and such that $V^+_{\vert t \vert} \cap ((B(w_0, \delta) \times 
U^+_1) \not= \emptyset$ for all $\vert t \vert < \eta$. Then for perhaps a 
smaller $\eta > 0$ it follows that $\pa V^+_{\vert t \vert} \cap (B(w_0, 
\delta) \times \pa U^+_1) = \emptyset$ for all $\vert t \vert < \eta$. 
Indeed suppose that there exist $(w_j, \ti w_j) \in V^+_{\vert t_j \vert} 
\cap ((B(w_0, \delta) \times U^+_1)$,  $\vert t_j \vert \ra 0$ such 
that $\ti w_j$ clusters at $\pa U^+_1$. Then
\[
L_{t_j}({}_{{}^{\ka}w_j}Q_{w_j}) \subset Q_{\ti w_j} \cap D.
\]
for all $j$. Let $w_j \ra w_{\infty} \in \ov B(w_0, \delta)$ and $\ti w_j 
\ra \ti w_{\infty} \in \pa U^+_1$. Since $w_j \in B(w_0, \delta)$ which is 
compactly contained in $U^+_1$, the germs ${}_{{}^{\ka}w_j}Q_{w_j}$ move 
within a compact subset of $U^-_1$ and hence 
\[
L_{t_j}({}_{{}^{\ka}w_j}Q_{w_j}) \ra 
L_0({}_{{}^{\ka}w_{\infty}}Q_{w_{\infty}}) = 
{}_{{}^{\ka}w_{\infty}}Q_{w_{\infty}} \subset Q_{\ti w_{\infty}} \cap D.
\]
This shows that $\la(w_{\infty}) = \la(\ti w_{\infty})$ which contradicts 
the fact that $\la : U_1 \ra \cal S(U_1, U_2)$ is proper. The set of 
points $w_0$ for which this reasoning applies is non-empty and open. An 
example would be all points in $U^+_1$ that lie close to boundary points 
in $\hat D$. Each $L_t$ would be well defined near such points and will 
satisfy the invariance property of Segre varieties. This shows that for 
any compact $K^+ \subset U^+_1$, there exists $\eta > 0$ such that the 
projection $\pi^+_{\vert t \vert} : V^+_{\vert t \vert} \cap (K^+ \times 
U^+_1) \ra K^+$ is proper for all $\vert t \vert < \eta$.

\medskip

Now fix a compact $K^- \subset U^-_1$. Since $V^-_0$ is the graph of the 
proper correspondence $w \mapsto \la^{-1} \circ \la(w)$ and $L_t$ 
converges uniformly on $K^-$ to $L_0(z) \equiv z$, it follows that the 
projection $\pi^-_t : V^-_t \cap (K^- \times U^-_1) \ra K^-$ is proper for 
all $\vert t \vert < \eta$ for a smaller $\eta > 0$ perhaps. Finally note that 
$(U_1 \cap \pa D) \sm \Om_r$ consists entirely of weakly pseudoconvex, finite 
type points or those that belong to $\hat D$. In both cases, each $L_t$ 
extends to a uniform neighbourhood (see \cite{Ber1} or \cite{V}) of those 
points for some $\vert t \vert < \eta$. Thus after shrinking $U_1$ and 
$\eta > 0$ suitably it follows that $V_t \subset (U_1 \sm \ov \Om_r) 
\times U_1$ is a closed complex analytic set and the projection 
\[
\pi_t : V_t \ra U_1 \sm \ov \Om_r
\]
is proper. Associated with each $\vert t \vert < \eta$ are positive 
integers $m_t, k_t$ and functions $a_{\mu}(w, t), b_{\nu}(w, t)$ for 
$1 \le \mu \le m_t, 1 \le \nu \le k_t$ all of which are holomorphic in 
$U_1 \sm \ov \Om_r$ such that $V_t \subset \hat V_t \subset (U_1 \sm \Om_r) 
\times U_1$ where $\hat V_t$ is a pure two dimensional complex analytic 
set described by $P_1(\ti w_1, w) = P_2(\ti w_2, w) = 0$ 
(here $w = (w_1, w_2)$ and $\ti w = (\ti w_1, \ti w_2)$) with
\begin{align*}
P_1(\ti w_1, w) &= \ti w^{m_t}_1 + a_1(w, t) \ti w^{m_t - 1}_1 + \ldots + 
a_{m_t}(w, t), \\
P_2(\ti w_2, w) &= \ti w^{k_t}_2 + b_1(w, t) \ti w^{k_t - 1}_1 + \ldots + 
b_{k_t}(w, t).
\end{align*}
At this stage, the functions $a_{\mu}(w, t), b_{\nu}(w, t)$ are known to 
be holomorphic in $w$ for each fixed $t$. Their dependence on $t$, let 
alone any joint regularity in $(w, t)$ is not known. To remedy this 
situation, first observe that $V_t$ is the union of the graphs of 
correspondences that involve $\la^{-1}, \la$ and $L_t$. Since $\la$ and 
$L_t$ are single valued, the multiplicity of $\pi_t$ is the same as that 
of $\la$ and this shows that $m_t \equiv m$ and $k_t \equiv k$ for all $\vert 
t \vert < \eta$. Second, it follows by lemma A1.3 in \cite{Ch1} that the 
envelope of holomorphy of $U_1 \sm \ov \Om_r$ is $U_1$. Hence the 
functions $a_{\mu}(w, t), b_{\nu}(w, t)$ extend to $U_1$ for each fixed 
$t$ and this implies that $\hat V_t$ admits a closed analytic continuation 
to $U_1 \times U_1$. This extension will still be denoted by $\hat V_t$ 
and since the polynomials $P_1(\ti w_1, w), P_2(\ti w_2, w)$ are monic in 
their first arguments, the projections $\pi_t$ will continue to be proper 
over $U_1$. Define
\[
\Sigma = \bigcup_{z \in U_1 \cap T_1}Q_z
\]
which is a closed, three dimensional real analytic set in $U_2$ and is 
locally foliated by open pieces of Segre varieties at all of its regular 
points. The complement $U_1 \sm \Sigma$ is open and hence non-pluripolar. 
Fix $w_0 \in U_1 \sm \Sigma$ and a relatively compact neighbourhood 
$B(w_0, \delta) \subset U_1 \sm \Sigma$. Note that for any $w \in B(w_0, 
\delta), Q_w \cap T_1 = \emptyset$. Two cases arise; first if $Q_{w_0} 
\cap U_2 \cap \pa D = \emptyset$ then on shrinking $U_2$ if necessary it 
follows that $Q_w \cap U_2 \cap D$ is uniformly relatively oompact in $D$ 
for all $w \in B(w_0, \delta)$. The action of $L_t$ on $D$ is real 
analytic and hence for any fixed $w \in B(w_0, \delta)$, $L_t(Q_w \cap D) 
= L_t \circ \la(w)$ is real analytic in $t$. Second if $Q_{w_0} \cap \pa D 
\not= \emptyset$ then $Q_w \cap T_1 = \emptyset$ implies that $Q_w \cap 
\pa D \cap U_2$ consists of weakly pseudoconvex finite type points or 
those that belong to $\hat D$ for all $w$ close to $w_0$. As before it 
follows from \cite{Ber1}, \cite{V} that $L_t(Q_w \cap D)$ depends real 
analytically on $t$ for each fixed $w \in B(w_0, \delta)$. Let $\sigma$ be 
the branch locus of $\la$. Then $\la(\sigma)$ is a complex analytic 
subset of $\cal S(U_1, U_2)$ of strictly smaller dimension and hence it is 
possible to move $w_0$ slightly and to shrink $\delta > 0$ if necessary so 
that $L_t(Q_w \cap D) \cap \la(\sigma) = \emptyset$ for all $(w, t) \in 
B(w_0, \delta) \times (-\eta, \eta)$. Thus the various branches of 
$\pi_t^{-1}$ over $B(w_0, \delta)$ are well defined holomorphic functions 
for each fixed $t$. In particular $a_{\mu}(w, t), b_{\nu}(w, t)$ depend 
real analytically on $t$ for each fixed $w \in B(w_0, \delta)$. It follows 
from theorem 1 in \cite{Shi} that $a_{\mu}(w, t), b_{\nu}(w, t)$ are 
jointly real analytic in $(w, t) \in U_1 \times (-\eta, \eta)$. Let 
$\sigma_t$ be the branch locus of $w \mapsto \la^{-1} \circ L_t \circ 
\la$. Then $\sigma_t$ is the zero set of a universal polynomial function 
of the symmetric functions $a_{\mu}(w, t), b_{\nu}(w, t)$. Hence 
$\sigma_t$ also varies real analytically in $t$. We now work in normal 
coordinates around $p_{\infty} = 0$, i.e., coordinates in which the 
defining function for $U \cap \pa D$ becomes
\[
\rho(z) = 2 \Re z_2 + \sum_{j \ge 0} c_j(z_1, \ov z_1) (\Im z_2)^j
\]
where the coefficients $c_j(z_1, \ov z_1)$ are real analytic and whose 
complexifications $c_j(z_1, \ov w_1)$ satisfy $c_j(z_1, 0) = 0 = 
c_j(0, \ov w_1)$. It is shown in \cite{DP} that $\sigma$ enters $D$ and 
that $\la^{-1} \circ L_t \circ \la(\sigma_t) \subset \sigma$ for all 
$\vert t \vert < \eta$. By the invariance property of Segre varieties 
it follows that $\sigma_t$ also enters $D$ for all $\vert t \vert < \eta$. 
As in theorem 7.4 of \cite{DP} consider the discs
\[
\Delta'_{\ep} = \{(z_1, z_2) : \vert z_1 \vert < c, z_2 = -\ep\}
\]
for $c$ small and $\ep > 0$. Then $\Delta_{\ep} = (\la^{-1} \circ 
\la)(\Delta'_{\ep})$ is a family of discs in $U_1$ such that $\ov \Delta_0 
\cap \sigma_0 = \{0\}$. Hence by continuity
\[
\Big( \bigcup_{\vert t \vert < \eta} \sigma_t \Big) \cap \ov \Delta_{\ep} 
\Subset \Delta_{\ep}
\]
for all $0 < \ep < \ep_0$ and $\eta > 0$ small. This shows that all 
possible branching occurs only in the interior of $\Delta_{\ep}$. Thus 
there is a uniform neighbourhood of
\[
\bigcup_{0 \le \ep < \ep_0} \pa \Delta_{\ep}
\]
to which each $L_t$ extends for all $\vert t \vert < \eta$. The disc 
theorem applied to the family $\Delta_{\ep}, 0 < \ep < \ep_0$ shows that 
$L_t$ extends to a uniform neighbourhood of $p_{\infty}$ for all $\vert t 
\vert < \eta$. This finishes the first part of the proof. Assume now that 
$p_{\infty} \in T_0$ or $Q_{p_{\infty}} = T_1^{\mbf C}$. Note that the 
proof of lemma 2.4 depends on the fact that there is uniform control 
on $L_t$ at points of $Q_0 \sm \{0\}$. This is assured by the above 
arguments and thus lemma 2.4 holds even when $p_{\infty} \in T_0$ or 
$Q_{p_{\infty}} = T_1^{\mbf C}$. Once again we may argue as above to show 
the uniform extendability of $L_t$ across $p_{\infty}$ for all $\vert t 
\vert < \eta$. This completes the proof of proposition 2.3. Incidentally, 
this localises the main theorem in \cite{V}.

\medskip

We return to the biholomorphic equivalence $g : G \ra D$. Consider the 
sequence $(a_j, b_j) = g^{-1} \circ \phi_j(p) \in G$, which is the 
pullback of the orbit $\{ \phi_j(p) \} \in D$, and let $2 \ep_j = 
2 \Re b_j + P(a_j, \ov a_j)$. Note that $\ep_j < 0$ for all $j$.

\begin{prop}
Suppose that $\vert \ep_j \vert > c > 0$ uniformly for all large $j$. Then 
$\cal X$ vanishes to finite order at $p_{\infty}$.
\end{prop}

\begin{proof}
First note that $g^{-1} \circ \phi_j(p)$ can cluster only at $\pa G$ 
since $\phi_j(p) \ra p_{\infty} \in \pa D$. The condition that $2 \Re b_j + 
P(a_j, \ov a_j)$ is uniformly strictly positive in modulus ensures that 
$g^{-1} \circ \phi_j(p)$ clusters only at the point at infinity in $\pa 
G$. For $\tau \in \mbf C$ note that $(a_j, b_j + \vert \ep_j \vert \tau) \in 
G$ whenever $\Re \tau < 1$. With $\cal H = \{\tau \in \mbf C : \Re \tau 
< 1\}$, the analytic disc $f_j : \cal H \ra D$ where
\[
f_j(\tau) = g(a_j, b_j + \vert \ep_j \vert \tau)
\]
is well defined. Note that $f_j(0) = g(a_j, b_j) = \phi_j(p) \ra 
p_{\infty}$. The integral curves of $i \; \pa / \pa z_2$ passing through 
$(a_j, b_j)$ are of the form $\gamma_j(t) = (a_j, b_j + i t), t \in \mbf 
R$ and hence $g \circ \gamma_j(t)$ defines the integral curves of $\cal X$ 
through $\phi_j(p)$. For $j$ fixed, as $\tau$ varies on the imaginary axis 
in $\cal H$, $(a_j, b_j + \vert \ep_j \vert \tau)$ sweeps out the integral 
curves of $i \; \pa / \pa z_2$ through $(a_j, b_j)$. Moreover, since 
$\vert \ep_j \vert > c > 0$ the image of the interval $(-M, M)$, for any 
$M > 0$, under the map $t \mapsto (a_j, b_j + i \vert \ep_j \vert t)$ 
contains the line segment from $(a_j, b_j - i c M)$ to $(a_j, b_j + i c 
M)$ which equals $\gamma_j((-c M, c M))$ for all $j$. 

\medskip

The family $f_j : \cal H \ra D$ is normal and since $f_j(0) \ra 
p_{\infty}$, there exists a subsequence of $f_j$ that converges uniformly 
on compact subsets of $\cal H$ to a holomorphic mapping $f : \cal H \ra 
\ov D$ with $f(0) = p_{\infty}$. Identify $\cal H$ with the unit disc 
$\Delta(0, 1)$ via a conformal map $\psi : \Delta(0, 1) \ra \cal H$ such 
that $\psi(0) = 0$. Choose $r \in (0, 1)$ such that $f \circ 
\psi(\Delta(0, r)) \subset U \cap \pa D$ where $U \cap \pa D$ is a smooth 
real analytic, finite type hypersurface. Suppose that $f$ is non-constant. 
Then the finite type assumption on $U \cap \pa D$ implies that $f \circ 
\psi(\Delta(0, r)) \cap D \not= \emptyset$. The strong disk theorem in 
\cite{Vl} forces $p_{\infty} \in \hat D$ which is a contradiction. Thus 
$f(\tau) \equiv p_{\infty}$ for all $\tau \in \cal H$. For an arbitrarily 
small neighbourhood $V$ of $p_{\infty}$ and the compact interval $[-iM, 
iM] \subset \cal H$, this means that $g \circ \gamma_j((-cM, cM)) \subset 
f_j([-iM, iM]) \subset V$ for all large $j$ which implies that the 
integral curve of $\cal X$ through $\phi_j(p)$ is contained in $V$ for all 
$t \in (-cM, cM)$. Since the $f_j$'s form a normal family, it follows that 
$g \circ \gamma_j$ converge uniformly on $[-cM, cM]$ to a path 
$\gamma_{\infty} : [-cM, cM] \ra V \cap \ov D$ with $\gamma_{\infty}(0) = 
p_{\infty}$. Since $\cal X$ is holomorphic in a neighbourhood of 
$p_{\infty}$, it follows that $\gamma_{\infty}$ is the integral curve of 
$\cal X$ through $p_{\infty}$. As $V$ and $M$ are arbitrary the vector 
field $\cal X$ must vanish at $p_{\infty}$. If $\cal X$ vanishes to 
infinite order at $p_{\infty}$ then $L_t = \exp (t \; \Re \cal X)$ shows 
that the identity mapping and $L_t$, for all $\vert t \vert < \eta$, agree 
to infinite order at $p_{\infty}$. By proposition 2.3, each $L_t$ 
for $\vert t \vert < \eta$ extends to a uniform neighbourhood of 
$p_{\infty}$ and hence $L_t$ is the identity mapping for all small $t$. 
This is a contradiction since the action of the one-parameter group 
$(L_t)$ on $D$ does not have fixed points.
\end{proof}

\begin{prop}
Suppose that $\dim {\rm Aut}(G) = 1$ and $g \in {\rm Aut}(G)$. Then 
$g(z_1, z_2) = (\alpha z_1 + \beta, \phi(z_1) + a z_2 + b)$ where $\alpha, 
a  \in \mbf C \sm \{0\}, \beta, b \in \mbf C$ and $\phi(z_1)$ is entire.
\end{prop}

\begin{proof}
Since $\dim {\rm Aut}(G) = 1$ it follows that the connected component of 
the identity ${\rm Aut}(G)^c$ must be the group generated by the 
translations $T_t(z_1, z_2) = (z_1, z_2 + it)$ and each such $T_t$ 
evidently has the form mentioned above. So we may suppose that $g \notin 
{\rm Aut}(G)^c$. In this case, note that since ${\rm Aut}(G)^c$ is normal 
in ${\rm Aut}(G)$, it follows that for each $t \in \mbf R$ there exists 
$t' = f(t) \in \mbf R$ such that $g \circ T_t(z_1, z_2) = T_{f(t)} \circ 
g(z_1, z_2)$. If $g(z_1, z_2) = (g_1(z_1, z_2), g_2(z_1, z_2))$ this 
is equivalent to
\begin{align*}
g_1(z_1, z_2 + it) &= g_1(z_1, z_2),\\
g_2(z_1, z_2 + it) &= g_2(z_1, z_2) + i f(t).
\end{align*}
The first equation implies that $\pa g_1 / \pa z_2 \equiv 0$ and hence 
$g_1(z_1, z_2) = g_1(z_1)$. Moreover, if $\pi : \mbf C^2 \rightarrow \mbf 
C_{z_1}$ is the natural projection, then $\pi(G) = \mbf C_{z_1}$ as the 
defining function for $G$ shows. Therefore $g_1(z_1)$ is entire and 
furthermore $g_1(z_1) \in {\rm Aut}(\mbf C)$ since this reasoning applies 
to $g^{-1}$ as well. Hence $g_1(z_1) = \alpha z_1 + \beta$ for some 
$\alpha \in \mbf C \sm \{0\}$. For $g_2$ fix $t \in \mbf R$ arbitrarily 
and differentiate with respect to $z_1, z_2$. This gives
\begin{align*}
\pa g_2 / \pa z_1 (z_1, z_2 + it) &= \pa g_2 / \pa z_1 (z_1, z_2),\\
\pa g_2 / \pa z_2 (z_1, z_2 + it) &= \pa g_2 / \pa z_2 (z_1, z_2)
\end{align*}
both of which hold for all $t \in \mbf R$. From the first equation above 
it can be seen that $\pa g_2 / \pa z_1 (z_1, z_2)$ is independent of 
$z_2$, i.e., $g_2(z_1, z_2) = \phi(z_1) + \psi(z_2)$ for some holomorphic 
$\phi(z_1), \psi(z_2)$. Putting this in the second equation it follows 
that $\psi'(z_2 + it) = \psi'(z_2)$ and therefore $\psi(z_2) = a z_2 + b$ 
for some $a, b \in \mbf C$. Now observe that
\[
Dg(z_1, z_2) =
\begin{pmatrix}
\alpha & 0\\
\phi'(z_1) & a
\end{pmatrix}
\]
which must be non-singular for all $(z_1, z_2) \in G$ and so $\alpha a 
\not= 0$. It can then be noted that $g$ is in fact an automorphism of 
$\mbf C^2$ that stabilises $G$. The iterates of $g$ can also be computed. 
For $n \ge 1$, let $g^n = g \circ g \circ \cdots \circ g \in {\rm Aut}(G)$ 
be the $n$-fold composition of $g$ with itself. From the explicit form of 
$g$ it can be seen that $g^n(z_1, z_2) = (\alpha^n z_1 + \beta_n, \phi_n(z_1) 
+ a^n z_1 + b_n)$ where $\beta_n = \beta (1 - \alpha^n)/(1 - \alpha)$, 
$b_n = b(1 - a^n)/(1 - a)$, and $\phi_n(z_1)$ is entire for all $n \ge 1$. 
Note that it is possible to compute $\phi_n(z_1)$ explicitly in terms of 
$\phi(z_1), \alpha, \beta, a, b$ but this shall not be needed. A similar 
calculation can be done for $n \le -1$.
\end{proof}

\begin{prop}
The function $\phi(z_1)$ is a polynomial and $a \in \mbf R$. 
Moreover $\vert \alpha \vert = 1$ and consequently $\vert a \vert = 1$.
\end{prop}

\begin{proof}
Since $g$ extends to an automorphism of $\mbf C^2$, it must preserve $\pa 
G$ as well. Hence
\[
2 \Re(\phi(z_1) + a z_2 + b) + P(\alpha z_1 + \beta, \ov {\alpha z_1 + 
\beta}) = 0
\]
whenever $2 \Re z_2 + P(z_1, \ov z_1) = 0$. For $z_1 \in \mbf C$ and $t 
\in \mbf R$ note that the point $(z_1, -P(z_1, \ov z_1)/2 + it) \in \pa 
G$. Therefore
\[
2 \Re \Big( \phi(z_1) + a(-P(z_1, \ov z_1)/2 + it) + b \Big) + P(\alpha 
z_1 + \beta, \ov {\alpha z_1 + \beta}) = 0
\]
for all $z_1 \in \mbf C$ and $t \in \mbf R$. Let $a = \mu + i \nu$ so that
\begin{equation}
2 \Re (\phi(z_1) + b) + 2 (-\mu P(z_1, \ov z_1)/2 - t \nu) + P(\alpha z_1 
+ \beta, \ov {\alpha z_1 + \beta}) = 0. 
\end{equation}
By comparing the coefficient of $t$ on both sides it follows that $\nu = 
0$, i.e., $a = \mu \in \mbf R$. With this (2.4) becomes
\begin{equation}
2 \Re(\phi(z_1) + b) - \mu P(z_1, \ov z_1) + P(\alpha z_1 + \beta, \ov 
{\alpha z_1 + \beta}) = 0
\end{equation}
for all $z_1 \in \mbf C$. It is now possible to identify $\phi(z_1)$; to 
do this we simply write $z$ in place of $z_1$ and allow $z, \ov z$ to vary 
independently, i.e., consider
\[
F(z, w) = \phi(z) + b + \ov{\phi(\ov w)} + \ov b - \mu P(z, w) + P(\alpha 
z + \beta, \ov \alpha w + \ov \beta)
\]
which is holomorphic in $(z, w) \in \mbf C^2$. Now (2.5) shows that $F = 
0$ on $\{w = \ov z\}$ which is maximally totally real and hence $F \equiv 
0$ everywhere by the uniqueness theorem. Putting $w = 0$ and noting that 
$P(z, 0) \equiv 0$ (since $P(z, \ov z)$ does not have pure terms) it 
follows that
\begin{equation}
\phi(z) = - \ov{\phi(0)} - b - \ov b - P(\alpha z + \beta, \ov \beta)
\end{equation}
which shows that $\phi(z)$ is a polynomial. 

\medskip

This expression for $\phi(z)$ can be used in (2.5) to get
\begin{equation}
\mu P(z, \ov z) = P(\alpha z + \beta, \ov {\alpha z + \beta}) - 
P(\alpha z + \beta, \ov \beta) - \ov{P(\alpha z + \beta, \ov \beta)} - 
(\phi(0) + b + \ov{\phi(0) +  b})
\end{equation}
for all $z \in \mbf C$. Let $P_j, P_{j \ov q}$ denote derivatives of 
the form $\pa^j P / \pa z^j$ and $\pa^{j + q} P / \pa z^j \pa \ov z^q$ 
respectively. Note that
\[
P(\alpha z + \beta, \ov \alpha \ov z + \ov \beta) =  P(\beta, \ov 
\beta)  + \sum_{j \ge 1} \Big( \alpha^j P_j(\beta, \ov \beta) 
\frac{z^j}{j!}  + \ov \alpha^j \ov {P_j(\beta, \ov \beta)} \frac {\ov 
z^j}{j!} \Big)\\
+ \sum_{j, q > 0} P_{j \ov q}(\beta, \ov \beta) \alpha^j \ov \alpha^q 
\frac{z^j \ov z^q}{j!q!} 
\]
the right side of which will be written as $P(\beta, \ov \beta) + I + II$ 
where $I$ consists of all the pure terms and $II$ contains 
only mixed terms of the form $z^j \ov z^q$ with both $j, q > 0$. Putting 
this in (2.7) gives
\[
\mu P(z, \ov z) = P(\beta, \ov \beta) + I + II - P(\alpha z + \beta, \ov 
\beta) - \ov{P(\alpha z + \beta, \ov \beta)} - (\phi(0) + b + \ov{\phi(0) 
+  b}).
\]
The left side above has no harmonic terms and therefore
\[
P(\beta, \ov \beta) + I - P(\alpha z + \beta, \ov 
\beta) - \ov{P(\alpha z + \beta, \ov \beta)} - (\phi(0) + b + \ov{\phi(0)
+ b}) \equiv 0
\]
which shows that
\begin{equation}
\mu P(z, \ov z) = \sum_{j, q > 0} P_{j \ov q} (\beta, \ov \beta) \alpha^j 
\ov \alpha^q \frac{z^j \ov z^q}{j!q!}.
\end{equation}
The point to note now is that in deriving (2.8) only the observation that 
$g$ stabilises the boundary of $G$ was used and no special properties of 
$\phi(z)$ played a role. Exactly the same reasoning can therefore be 
applied to $g^m$, which also stabilises $\pa G$, and this shows 
that
\begin{equation}
\mu^m P(z, \ov z) = \sum_{j, q > 0} P_{j \ov q}(\beta_m, \ov \beta_m) 
(\alpha^m)^j (\ov \alpha^m)^q \frac{z^j \ov z^q}{j! q!}
\end{equation}
for all $m \in \mbf Z$. Note that if $P(z, \ov z)$ is homogeneous of 
degree $l > 0$ then $G$ will admit a one-parameter subgroup $s \mapsto 
S_s(z_1, z_2) = (\exp(s/l)z_1, \exp(s) z_2), s \in \mbf R$ in addition to 
the translations $T_t$ and hence $\dim {\rm Aut}(G) \ge 2$. Therefore, let 
\[
P(z, \ov z) = P^{J_1}(z, \ov z) + P^{J_2}(z, \ov z) + \ldots + P^{J_n}(z, 
\ov z)
\]
be the decomposition of $P(z, \ov z)$ into homogeneous summands of degree 
$J_i > 0$ ($1 \le i \le n$) where $J_1 < J_2 < \ldots < J_n$. Let
\[
P^{J_i}(z, \ov z) = \sum_{k + l = J_i} C^{J_i}_{k \ov l} \; z^k \ov z^l
\]
for $1 \le i \le n$. Two cases need to be considered:

\medskip

\no {\it Case 1:} When $\beta = 0$ rewrite (2.5) as
\[
\mu P(z, \ov z) = 2 \Re (\phi(z) + b) + P(\alpha z, \ov \alpha \ov z)
\]
and observe that the left side does not have harmonic terms. Therefore 
$\Re (\phi(z) + b) \equiv 0$. Pick pairs of positive indices $(k, l), (j, q)$ 
such that $k + l = J_n, j + q = J_{n - 1}$ and for which $C^{J_{n - 1}}_{j 
\ov q}, C^{J_n}_{k \ov l}$ are both non-zero. Comparing the coefficients 
of $z^k \ov z^l$ and $z^j \ov z^q$ in (2.8) gives
\[
\mu = \alpha ^k \ov \alpha^l = \alpha^j \ov \alpha^q
\]
which implies that
\[
\vert \mu \vert = \vert \alpha \vert^{k + l} = \vert \alpha \vert^{J_n} = 
\vert \alpha \vert^{j + q} = \vert \alpha \vert^{J_{n - 1}}.
\]
Since $J_{n - 1} < J_n$ it follows that $\vert \alpha \vert = 1$ and 
consequently $\vert \mu \vert = 1$.

\medskip

\no {\it Case 2:} Suppose that $\beta \not= 0$. Pick $j, q > 0$ such 
that $j + q = J_{n - 1}$ and for which $C^{J_{n - 1}}_{j \ov q} \not= 0$. 
Compare the coefficient of $z^j \ov z^q$ in (2.9) to get
\begin{align}
\mu^m C^{J_{n - 1}}_{j \ov q} &= P_{j \ov q}(\beta_m, \ov \beta_m) 
\frac{(\alpha^m)^j (\ov \alpha^m)^q}{j! q!}, \notag \\
 	&=(\alpha^m)^j (\ov \alpha^m)^q C^{J_{n - 1}}_{j \ov q} + 
P^{J_n}_{j \ov q}(\beta_m, \ov \beta_m) \frac{(\alpha^m)^j (\ov 
\alpha^m)^q}{j! q!}
\end{align}
for all $m \in \mbf Z$. To obtain a contradiction, assume that $0 < \vert 
\alpha \vert < 1$ and consider the above equation for $m \ge 1$. 
Similar arguments will work in the case $\vert \alpha \vert > 1$ by 
considering $m \le -1$. The argument again splits into two parts.

\medskip

\no {\it Sub-Case (a):} Suppose that $P^{J_n}(z, \ov z) = C^{J_n}_{\sigma 
\ov \sigma} z^{\sigma} \ov z^{\sigma}$. Comparing the coefficient of 
$z^{\sigma} \ov z^{\sigma}$ in (2.8) shows that $\mu = \alpha^{\sigma} \ov 
{\alpha}^{\sigma} = \vert \alpha \vert^{2 \sigma}$. Using this and 
recalling that $\beta_m = \beta(1 - \alpha^m)/(1 - \alpha)$ allows (2.10) 
to be rewritten as
\[
(E_m): \qquad \Big( 1 - (\alpha^m)^{\sigma - j} (\ov \alpha^m)^{\sigma - 
q} \Big) C^{J_{n - 1}}_{j \ov q} + C^{J_n}_{\sigma \ov \sigma} (\ast) 
\Big( \frac{\beta}{1 - \alpha} \Big)^{J_n - J_{n - 1}} 
(1 - \alpha^m)^{\sigma - j} (1 - \ov \alpha^m)^{\sigma - q} = 0
\]
for all $m \ge 1$, where $(\ast)$ denotes an unimportant but non-zero 
constant depending on $\sigma, j, q$. Then $(E_m)$ forms an infinite system of 
linear equations in the unknowns $C^{J_{n - 1}}_{j \ov q}$ and 
$C^{J_n}_{j \ov q}$. The rank of any pair $(E_m), (E_{m'})$ cannot be two 
as otherwise both $C^{J_{n - 1}}_{j \ov q}$ and $C^{J_n}_{\sigma \ov 
\sigma}$ will vanish which cannot be true. Therefore 
\[
\begin{vmatrix}
1 - (\alpha^m)^{\sigma - j}(\ov \alpha^m)^{\sigma - q} & (1 - 
\alpha^m)^{\sigma - j} (1 - \ov \alpha^m)^{\sigma - q} \\
1 - (\alpha^{2m})^{\sigma - j}(\ov \alpha^{2m})^{\sigma - q} & (1 - 
\alpha^{2m})^{\sigma - j} (1 - \ov \alpha^{2m})^{\sigma - q}
\end{vmatrix}
 = 0
\]
for all $m \ge 1$. After removing the non-vanishing factors $1 - 
(\alpha^m)^{\sigma - j}(\ov \alpha^m)^{\sigma - q}$ and $(1 -
\alpha^m)^{\sigma - j} (1 - \ov \alpha^m)^{\sigma - q}$ from the columns 
respectively, it follows that
\begin{equation}
(1 + \alpha^m)^{\sigma - j} (1 + \ov \alpha^m)^{\sigma - q} - 1 - 
(\alpha^m)^{\sigma - j} (\ov \alpha^m)^{\sigma - q} = 0
\end{equation}
for all $m \ge 1$. Hence the complex valued function
\[
f(z, \ov z) = (1 + z)^{\sigma - j} (1 + \ov z)^{\sigma - q} - (1 + 
z^{\sigma - j} \ov z^{\sigma - q})
\]
vanishes at $\alpha^m$ for all $m \ge 1$. However it can be seen that
\begin{align*}
f(z, \ov z) &= (\sigma - j)z + (\sigma - q) \ov z + \ldots\\
	&= \Big( x(J_n - J_{n - 1}) + \ldots \Big) + i\Big( y(q - j) + 
	\ldots \Big)
\end{align*}
where the lower dots indicate terms of higher order and $z = x + iy$. In 
case $j \not= q$ it follows that the real and imaginary parts of $f(z, \ov 
z)$ vanish along smooth real analytic arcs that intersect at the origin 
transversally. Hence the origin is an isolated zero of $f(z, \ov z)$. On 
the other hand $\vert \alpha \vert^m \ra 0$ as $m \ra \infty$ and this 
forces $\alpha = 0$ which is a contradiction. In case $j = q$, (2.11) 
becomes
\begin{equation}
\vert 1 + \alpha^m \vert^{2(\sigma - j)} = 1 + \vert \alpha^m 
\vert^{2(\sigma - j)}
\end{equation}
for all $m \ge 1$. This leads to the consideration of the real analytic 
germ at the origin defined by
\[
A = \{z \in \mbf C : \vert 1 + z \vert^{2(\sigma - j)} - 1 - \vert z 
\vert^{2(\sigma - j)} = 0\}.
\]
By looking at the lowest order terms, it follows that $A$ is smooth at the 
origin. Fix a small disc $\Delta(0, \epsilon)$ so that $A \cap \Delta(0, 
\epsilon)$ is a smooth real analytic arc and define $\theta : A \cap 
\Delta(0, \epsilon) \ra S^1$ by $\theta(z) = z / \vert z \vert$. When 
$\sigma - j = 1$, it can be seen that $A$ is just the $y$-axis and 
therefore the range of $\theta$ is a two point set on $S^1$. When $\sigma 
- j > 1$ the smoothness of $A$ implies that the range of $\theta$ is the 
disjoint union of two open arcs of total length strictly less than $2 
\pi$. Since $\alpha^m \in A \cap \Delta(0, \epsilon)$ for $m \gg 1$ it 
follows that $\alpha / \vert \alpha \vert$ must be a root of 
unity as otherwise the set $\{ \alpha^m / \vert \alpha \vert^m : m \ge 1 
\}$, which is contained in the range of $\theta$, will be dense in $S^1$. 
So let $\alpha^{\eta} \in \mbf R$ for some positive integer $\eta$. Now 
let $m$ range over all integral multiples of $\eta$ in (2.12) and this 
gives
\[
\Big( 1 + (\alpha^\eta)^m \Big)^{2(\sigma - j)} = 1 + \Big( 
(\alpha^\eta)^m \Big)^{2(\sigma - j)}
\]
for all $m \ge 1$. The polynomial $(1 + \delta)^{2(\sigma - j)} = 1 + 
\delta^{2(\sigma - j)}$ in the real indeterminate $\delta$ hence has 
infinitely many roots given by $(\alpha^\eta)^m$ for all $m \ge 1$. It 
follows that $(\alpha^\eta)^{m_1} = (\alpha^\eta)^{m_2}$ for some $m_1 
\not= m_2$ and this shows that either $\vert \alpha \vert = 0$ or $1$ 
which is a contradiction.

\medskip

\no {\it Sub-Case (b):} When $P^{J_n}(z, \ov z)$ is no longer a monomial 
such as the one considered above, pick $k, l > 0, k \not= l$ such that $k 
+ l = J_n$ and for which $C^{J_n}_{k \ov l} \not= 0$. Compare the 
coefficients of $z^k \ov z^l$ and $z^l \ov z^k$ in (2.8) to get $\mu = 
\alpha^k \ov \alpha^l = \alpha^l \ov \alpha^k$. If $\alpha = \vert \alpha 
\vert e^{i \theta}$ it follows that $\theta \in 2 \pi \mbf Q$, i.e., 
$\alpha / \vert \alpha \vert$ is a root of unity. Let $r = \vert \alpha 
\vert$. Now choose $j, q > 0$ with $j + q = J_{n - 1}$ such that $C^{J_{n 
- 1}}_{j \ov q} \not= 0$ and assume that $\vert \alpha \vert \not= 1$. 
Consider (2.10) for $m \ge 1$. Since $\mu = \alpha^k \ov \alpha^l$ and 
$\beta_m = \beta(1 - \alpha^m)/(1 - \alpha)$, it can be rewritten as
\[
\Big(1 - \Big(\frac{\mu}{\alpha^j \ov \alpha^q} \Big)^m \Big) C^{J_{n - 
1}}_{j \ov q} + \sum_{\sigma, \tau} C^{J_n}_{\sigma \ov \tau} 
(\ast) \Big(\frac{\beta}{1 - \alpha} \Big)^{\sigma - j} 
\Big(\frac{\ov \beta}{1 - \ov \alpha} \Big)^{\tau - q} 
(1 - \alpha^m)^{\sigma - j} (1 - \ov \alpha)^{\tau - q} = 0
\]
for all $m \ge 1$, where the $(\ast)$'s are non-zero constants that depend 
on $\sigma, \tau, j, q$ and the sum extends over only those pairs 
$(\sigma, \tau)$ for which $\sigma \ge j$ and $\tau \ge q$. If $\theta = 2 
\pi (\delta / \eta)$ for relatively 
prime integers $\delta, \eta$, we may let $m$ vary over all the positive 
multiples of $\eta$ to get
\[
(E'_m) : \qquad  \Big(1 - (r^{\eta})^{m(J_n - J_{n - 1})} \Big) C^{J_{n - 
1}}_{j \ov q} + \Big(1 - (r^{\eta})^m \Big)^{J_n - J_{n - 1}} 
\underbrace{\Big(\sum_{\sigma, \tau} C^{J_n}_{\sigma \ov \tau} (\ast) \Big( 
\frac{\beta}{1 - \alpha} \Big)^{\sigma - j} \Big( \frac{\ov \beta}{1 - \ov 
\alpha} \Big)^{\tau - q} \Big)}_{C} = 0
\]
for all $m \ge 1$. As above, the $(E'_m)$ form an infinite linear system 
in the unknowns $C^{J_{n - 1}}_{j \ov q}$ and $C$. The rank of the pair 
$(E'_m)$ and $(E'_{2m})$ must be less than two for all $m \ge 1$ as 
otherwise $C^{J_{n - 1}}_{j \ov q} = 0$ in particular, which is false by 
assumption. Therefore
\[
\begin{vmatrix}
1 - (r^{\eta})^{m(J_n - J_{n - 1})} & (1 - (r^{\eta})^m)^{J_n - J_{n - 
1}}\\
1 - (r^{\eta})^{2m(J_n - J_{n - 1})} & (1 - (r^{\eta})^{2m})^{J_n - J_{n - 
1}}
\end{vmatrix}
= 0
\]
for all $m \ge 1$. Factor the non-vanishing terms $1 - (r^{\eta})^{m(J_n - 
J_{n - 1})}$ and $(1 - (r^{\eta})^m)^{J_n - J_{n - 1}}$ from the columns 
and this implies
\[
1 + r^{\eta m(J_n - J_{n - 1})} - \Big(1 + (r^{\eta m}) \Big)^{J_n - J_{n 
- 1}} 
= 0
\]
for all $m \ge 1$. The reasoning used in sub-case (a) now shows that $r = 
0$ or $1$ which is a contradiction.

\medskip

To conclude, pick $k, l > 0$ such that $k + l = J_n$ (where possibly $k = 
l$) and compare the coefficient of $z^k \ov z^l$ in (2.8) to get $\mu = 
\alpha^k \ov \alpha^l$ and thus $\vert a \vert = \vert \mu \vert = \vert 
\alpha \vert^{k + l} = 1$.

\end{proof}

\begin{lem}
Suppose that $\dim {\rm Aut}(G) = 1$. Let $g(z_1, z_2) = (g_1(z_1, z_2), 
g_2(z_1, z_2)) = (\alpha z_1 + \beta, \phi(z_1) + a z_2 + b) \in {\rm 
Aut}(G)$ and $q = (q_1, q_2) \in G$ be an arbitrary point. Define
\[
E = 2 \Re (g_2(q_1, q_2)) + P(g_1(q_1, q_2), \ov {g_1(q_1, q_2)}).
\]
Then $E$ is independent of the parameters involved in $g(z_1, z_2)$ and in 
fact $\vert E \vert = \vert 2 \Re q_2 + P(q_1, \ov q_1) \vert$.
\end{lem}

\begin{proof}
Observe that
\begin{align*}
E &= 2 \Re(\phi(q_1) + a q_2 + b) + P(\alpha q_1 + \beta, \ov{\alpha q_1 + 
\beta})\\
 &= 2 a \Re q_2 + 2 \Re(\phi(q_1) + b) + P(\alpha q_1 + \beta, \ov{\alpha 
q_1 + \beta})\\
 &= 2 a \Re q_2 + a P(q_1, \ov q_1)
\end{align*}
where the second equality uses the fact that $a \in \mbf R$ while the 
third follows from (2.5). It remains to note that $\vert a \vert = 1$ and 
this completes the proof.
\end{proof}

\begin{prop}
The dimension of ${\rm Aut}(D)$ is at least two.
\end{prop}

\begin{proof}
Suppose not. Then $\dim {\rm Aut}(D) = \dim {\rm Aut}(G) = 1$. Observe 
that $(a_j, b_j) = g^{-1} \circ \phi_j(p) = g^{-1} \circ \phi_j \circ g 
(g^{-1}(p))$ where $g^{-1} \circ \phi_j \circ g \in {\rm Aut}(G)$ for all 
$j \ge 1$. Let $g^{-1}(p) = q = (q_1, q_2) \in G$. It follows from 
propositions 2.6, 2.7 and lemma 2.8 that
\[
\vert 2 \Re b_j + P(a_j, \ov a_j) \vert = \vert 2 \Re q_2 + P(q_1, \ov 
q_1) \vert > 0
\]
for all $j$. This also shows that $\{ g^{-1} \circ \phi_j(p) \}$ can 
cluster only at the point at infinity in $\pa G$. Now proposition 2.5 
shows that $\cal X(p_{\infty}) = 0$ and 
by lemma 3.5 in \cite{BP4} the intersection of the zero set of $\cal X$ 
with $\pa D$ contains $p_{\infty}$ as an isolated point. Moreover $\cal X$ 
does not vanish in $D$ as the action of $(L_t)$ on $D$ does not have fixed 
points. The sequence $g^{-1} \circ \phi_j(p)$ converges to the point at 
infinity in $\pa G$ and its image under $g$, namely $\phi_j(p)$ 
converges to $p_{\infty}$. Since the cluster set of the point at infinity 
in $\pa G$ under $g$ is connected it must equal $\{p_{\infty}\}$. By 
defining $g(\infty) = p_{\infty}$, the mapping $g : G \cup \{\infty\} 
\ra D \cup \{p_{\infty}\}$ is therefore continuous. Now
\[
L_t(p) = L_t \circ g(q_1, q_2) = g \circ T_t(q_1, q_2) = 
g(q_1, q_2 + i t)
\]
shows that $\lim_{\vert t \vert \ra \infty} L_t(p) = p_{\infty}$. For any 
other $p' \in D$ and $g^{-1}(p') = (q'_1, q'_2) \in G$ it follows 
that
\[
L_t(p') = L_t \circ g(q'_1, q'_2) = g \circ T_t(q'_1, q'_2) = 
g(q'_1, q'_2 + i t)
\]
where $g(q'_1, q'_2 + i t) \ra p_{\infty}$ since $\vert (q'_1, 
q'_2 + i t) \vert \ra \infty$. Thus $\lim_{\vert t \vert \ra \infty} 
L_t(p') = p_{\infty}$. This establishes the parabolicity of the action of 
$(L_t)$ on $D$ and now the arguments in \cite{BP4} show that $D \backsimeq 
\{(z_1, z_2) \in \mbf C^2: \vert z_1 \vert^2 + \vert z_2 \vert^{2m} < 1\}$ 
for some integer $m \ge 1$. This shows that $\dim{\rm Aut}(D) = 4$ which 
is a contradiction.
\end{proof}

\section{Model domains when ${\rm Aut}(D)$ is two dimensional}

\no Since $g : G \ra D$ is biholomorphic and $G$ is invariant under 
the translations $T_t(z_1, z_2) = (z_1, z_2 + it), t \in \mbf R$, it 
follows that if ${\rm Aut}(D)^c$ is abelian then it must be isomorphic to 
$\mbf R^2$ or $\mbf R \times S^1$. The model domains corresponding to 
these cases can be computed as follows.

\begin{prop}
Let $D$ be as in the main theorem and suppose that $\dim {\rm Aut}(D) = 
2$. If ${\rm Aut}(D)^c$ is abelian then either $D \backsimeq \cal D_1 
= \{(z_1, z_2) \in \mbf C^2 : 2 \Re z_2 + P_1(\Re z_1) < 0 \}$ or $D 
\backsimeq \cal D_2 = \{(z_1, z_2) \in \mbf C^2 : 2 \Re z_2 + 
P_2(\vert z_1 \vert^2) < 0\}$ for some polynomials $P_1(\Re z_1)$ and 
$P_2(\vert z_1 \vert^2)$ that depend only on $\Re z_1$ and $\vert z_1 
\vert^2$ respectively.
\end{prop}

\begin{proof}
The hypotheses imply that ${\rm Aut}(G)^c$ is two dimensional. Choose a 
one-parameter subgroup of the form $S_s(z_1, z_2) = (S^1_s, S^2_s)$ 
that along with $T_t(z_1, z_2) = (z_1, z_2 + i t), t \in 
\mbf R$ generates ${\rm Aut}(G)^c$. Since ${\rm Aut}(G)^c$ is abelian, 
it follows that $S_s \circ T_t = T_t \circ S_s$ which in turn implies that
\begin{align*}
S^1_s(z_1, z_2 + i t) &= S^1_s(z_1, z_2), \; {\rm and}\\ 
S^2_s(z_1, z_2 + i t) &= S^2_s(z_1, z_2) + i t
\end{align*}
for all $s, t \in \mbf R$. The first equation above shows that $\pa S^1_s 
/ \pa z_2 \equiv 0$ in $G$ which means that $S^1_s(z_1, z_2)$ is 
independent of $z_2$, i.e., $S^1_s(z_1, z_2) = S^1_s(z_1)$ for all $s$. 
Now note that the projection $\pi(z_1, z_2) = z_1$ maps $G$ 
surjectively onto the $z_1$-axis and this implies that $(S^1_s(z_1)) 
\subset {\rm Aut}(\mbf C)$ is a non-trivial one-parameter subgroup. The 
second equality above shows that $\pa S^2_s / \pa z_2 \equiv 1$ in 
$\Om$, i.e., $S^2_s(z_1, z_2) = z_2 + h(s, z_1)$ where $h(s, z_1)$ 
is an entire function in $z_1$ for all $s \in \mbf R$. Moreover note 
that $h(0, z_1) = 0$. Thus 
$S_s$ is an automorphism of $\mbf C^2$ for all $s$. But since $S_s \in 
{\rm Aut}(G)$ it follows that
\[
\Re (S^2_s(z_1, z_2)) + P(S^1_s(z_1), \ov S^1_s(z_1)) = 0
\] 
whenever $\Re z_2 = -  P(z_1, \ov z_1)$. Hence
\[
\Re h(s, z_1) = P(z_1, \ov z_1) - P(S^1_s(z_1), \ov S^1_s(z_1))
\]
for all $z_1 \in \mbf C$. It follows that $h(s, z_1)$ is a polynomial for 
all $s \in \mbf R$.

\medskip

Returning to the one-parameter subgroup $(S^1_s(z_1)) \subset {\rm 
Aut}(\mbf C)$, it is possible to make an affine change of coordinates in 
$z_1$ so that for all $s \in \mbf R$, $S^1_s(z_1) = z_1 + i s$ or 
$S^1_s(z_1) = \exp(\alpha s) z_1$ for some $\alpha \in \mbf C \sm \{0\}$. 
This can be done using the known description of ${\rm Aut}(\mbf C)$. Two 
cases arise, the first being
\[
S_s(z_1, z_2) = (S^1_s(z_1), S^2_s(z_1, z_2)) = (z_1 + is, z_2 + h(s, 
z_1))
\]
for all $s \in \mbf R$ which implies that
\[
h(s_1 + s_2, z_1) = h(s_2, z_1) + h(s_1, z_1 + i s_2)
\]
for all $s_1, s_2 \in \mbf R$. Writing this as 
\[
(h(s_1 + s_2, z_1) - h(s_2, z_1))/s_1 = h(s_1, z_1 + i s_2)/ s_1
\]
for non-zero $s_1$ we get that $\pa h/\pa s (s_2, z_1) = \pa h/\pa s(0, 
z_1 + i s_2)$ for all $s_2$. Since $h(s, z_1)$ is a polynomial for all $s$ 
it follows that $\pa h/ \pa s(0, z_1 + is)$ is a polynomial as well. 
Integration gives
\[
h(s, z_1) = q(z_1 + is) + C
\]
for some polynomial $q(z)$ and since $h(0, z_1) = 0$, it follows that 
$h(s, z_1) = q(z_1 + is) - q(z_1)$. Hence
\[
S_s(z_1, z_2) = (z_1 + is, z_2 + q(z_1 + is) - q(z_1))
\]
for all $s \in \mbf R$. The automorphism $(z_1, z_2) \mapsto (\ti z_1, \ti 
z_2) = (z_1, z_2 - q(z_1))$ maps $G$ biholomorphically to $\ti G = 
\{(z_1, z_2) \in \mbf C^2: 2 \Re \ti z_2 + Q(\ti z_1, \ov {\ti z_1}) < 
0\}$ for some polynomial $Q(\ti z_1, \ov {\ti z_1})$ and conjugates the 
action of $S_s$ to the automorphism of $\ti G$ given by
\[
(\ti z_1, \ti z_2) \mapsto (\ti z_1 + is, \ti z_2).
\]
This shows that $Q(\ti z_1, \ov {\ti z_1})$ is invariant under the map 
$\ti z_1 \mapsto \ti z_1 + is$ and hence $Q(\ti z_1, \ov {\ti z_1}) = 
Q(\Re z_1)$. This realizes $G$ as a tube domain after a global change 
of coordinates.

\medskip

\noindent The other case to consider is when
\[
S_s(z_1, z_2) = (S^1_s(z_1), S^2_s(z_1, z_2)) = (\exp(\alpha s) z_1, z_2 + 
h(s, z_1))
\]
for some $\alpha \in \mbf C \sm \{0\}$. As above $h(s, z_1)$ is a 
polynomial for all $s \in \mbf R$ and hence $S_s(z_1, z_2) \in {\rm 
Aut}(\mbf C^2)$ that preserves $G$. This means that
\[
\Re(z_2 + h(s, z_1)) + P(\exp(\alpha s) z_1, \exp(\ov{\alpha}s) \ov 
z_1) = \Re z_2 + P(z_1, \ov z_1)
\]
for all $s \in \mbf R$ which shows that
\[
\Re(h(s, z_1)) +  P(\exp(\alpha s) z_1, \exp(\ov{\alpha} s) \ov z_1) = 
P(z_1, \ov z_1).
\]
Since the right side has no harmonic terms it follows that $\Re(h(s, z_1)) 
\equiv 0$ for all $s \in \mbf R$ and hence $h(s, z_1) \equiv i \beta s$ 
for some real $\beta$. Therefore $P(\exp(\alpha s) z_1, \exp(\ov 
{\alpha}s) \ov z_1) = P(z_1, \ov z_1)$ for all $s \in \mbf R$. This 
forces $\alpha = i \gamma$ for some non-zero real $\gamma$ and that 
$P(z_1, \ov z_1)$ must consist only of terms of the form $\vert z_1 
\vert^{2k}$ for some integer $k \ge 1$. Thus $P(z_1, 
\ov z_1) = P(\vert z_1 \vert^2)$.

\medskip

\no Thus if ${\rm Aut}(D)^c$ is two dimensional and abelian the following 
dichotomy holds:

	\begin{itemize}
	    \item $D \backsimeq \cal D_1 = \{(z_1, z_2) \in \mbf C^2 : 
	    2 \Re z_2 + P_1(\Re z_1) < 0 \}$ and ${\rm Aut}(\cal 
	    D_1)^c$ is generated by $T_t(z_1, z_2) = (z_1, z_2 + it)$ and 
	    $S_s(z_1, z_2) = (z_1 + is, z_2)$ for $s, t \in \mbf R$.

	    \item $D \backsimeq \cal D_2 = \{(z_1, z_2) \in \mbf C^2 : 
	    2 \Re z_2 + P_2(\vert z_1 \vert^2) < 0\}$ and ${\rm 
	    Aut}(\cal D_2)^c$ is generated by $T_t(z_1, z_2) = (z_1, 
	    z_2 + it)$ and $S_s(z_1, z_2) = (\exp(i \gamma s) z_1, z_2 + i 
	    \beta s)$ for $s, t \in \mbf R$ and $\beta, \gamma \not= 0$.
	\end{itemize}

\end{proof}

\begin{cor}
Suppose that $\dim {\rm Aut}(D) = 2$. If $\phi_j \in {\rm Aut}(D)^c$ for 
all large $j$, then ${\rm Aut}(D)^c$ cannot be abelian.
\end{cor}

\begin{proof}
If ${\rm Aut}(D)^c$ is abelian then the model domains and the generators 
for the connected component of the identity of the automorphism group are 
given above. Let $g_1 : D \ra \cal D_i$, $i = 1, 2$ be the biholomorphic 
equivalences. The sequence $\{g_i \circ \phi_j(p)\}_{j \ge 1}$ then clusters 
only at $\pa \cal D_i$ and in fact only at the point at infinity in $\pa 
\cal D_i$. To show this let us consider $g_1 : D \ra \cal D_1$. For each 
$j \gg 1$, $g_1 \circ \phi_j \circ g^{-1}_1 \in {\rm Aut}(\cal D_1)^c$ and 
thus there are unique reals $s_j, t_j$ such that $g_1 \circ \phi_j \circ 
g^{-1}_1 = S_{s_j} \circ T_{t_j}$. If $g_1(p) = (a, b) \in \cal D_1$ then
\[
g_1 \circ \phi_j(p) = S_{s_j} \circ  T_{t_j}(a, b) = (a + i s_j, b + i 
t_j)
\]
for all $j \gg 1$ and since $\{\phi_j\}$ is non-compact at least one of 
$\vert 
s_j \vert$ or $\vert t_j \vert \ra \infty$. Hence $\vert g_1 \circ 
\phi_j(p) \vert \ra \infty$. Moreover note that
\[
\vert 2 \Re(b + i t_j) + P_1(\Re(a + i s_j)) \vert = \vert 2 \Re b + 
P_1(\Re a) \vert > 0
\]
for all $j$. Similarly for $g_2 : D \ra \cal D_2$ we get
\[
g_2 \circ \phi_j(p) = S_{s_j} \circ T_{t_j}(a, b) = (\exp(i \gamma s_j) a, 
b + i t_j + i \beta s_j)
\]
for all $j \gg 1$. In this case $\vert t_j \vert \ra \infty$ and thus 
$\vert g_2 
\circ \phi_j(p) \vert \ra \infty$. Moreover
\[
\vert 2 \Re(b + i t_j + i \beta s_j) + P_2(\vert \exp(i \gamma s_j) a 
\vert^2) \vert = \vert 2 \Re b + P_2(\vert a \vert^2) \vert > 0
\]
for all $j$. The arguments used in proposition 2.5 and 2.9 now show that 
$D \backsimeq \{(z_1, z_2) \in \mbf C^2 : \vert z_1 \vert^2 + \vert z_2 
\vert^{2m} < 1\}$ for some integer $m \ge 1$ which means that $\dim {\rm 
Aut}(D) = 4$. This is a contradiction.
\end{proof}

\begin{prop}
Let $\dim {\rm Aut}(D) = 2$ and suppose that ${\rm Aut}(D)^c$ 
is non-abelian. Then $D$ is biholomorphic to a domain of the form
\[
\cal D_3 = \{(z_1, z_2) \in \mbf C^2 : 2 \Re z_2 + P_{2m}(z_1, \ov z_1) < 
0\}
\]
where $P_{2m}(z_1, \ov z_1)$ is a homogeneous polynomial of degree $2m$ 
without harmonic terms.
\end{prop}

\begin{proof}
We will again work with $G$ instead of $D$ as in the previous 
proposition. The first step is to show that the one-parameter subgroup 
$(T_t)$ is normal in ${\rm Aut}(G)^c$. Suppose not. The real vector field 
that generates $(T_t)$ is 
\[
X = i/2 \; \pa / \pa z_2 - i/2 \; \pa / \pa \ov z_2
\]
and it is possible to find another real vector field $Y$ such that $X, Y$ 
generate $\mathfrak{g}(G)$ the Lie algebra of ${\rm Aut}(G)^c$ and 
\begin{equation}
[X, Y] = \mu Y 
\end{equation}
for some non-zero real $\mu$, i.e., the one-parameter subgroup generated 
by $Y$ is normal in ${\rm Aut}(G)^c$. Let
\[
Y = \Re \Big( f_1(z_1, z_2) \; \pa / \pa z_1 + f_2(z_1, z_2) \; \pa / \pa 
z_2 \Big)
\]
for some $f_1(z_1, z_2), f_2(z_1, z_2) \in \cal O(G)$. Now (3.1) is 
equivalent to
\[
\pa f_1 / \pa z_2 = - 2i \mu f_1 \;\; {\rm and} \;\; \pa f_2 / \pa z_2 = - 
2i \mu f_2
\]
which yield
\[
f_1(z_1, z_2) = \exp(-2i \mu z_2)\; h_1(z_1) \;\; {\rm and} \;\; f_2(z_1, 
z_2) = \exp(-2i \mu z_2)\; h_2(z_1).
\]
Both $h_1(z_1)$ and $h_2(z_1)$ are entire functions since $G$ projects 
surjectively onto the $z_1$-axis. The special form of $f_1(z_1, z_2), 
f_2(z_1, z_2)$ shows that $Y$ generates a one-parameter subgroup of ${\rm 
Aut}(\mbf C^2)$. Now suppose that $h_1(z_0) = 0$ for some $z_0 \in \mbf 
C$. Then the restrictions of both $X, Y$ to the line $L = \{(z_1, z_2) \in 
\mbf C^2 : z_1 = z_0\}$ are multiples of $\pa / \pa z_2$ and hence ${\rm 
Aut}(G)$ acts on $L$ and this action leaves the half-plane 
\[
G \cap L = \{(z_0, z_2) \in \mbf C^2 : 2 \Re z_2 + P(z_0, \ov z_0) < 0\}
\]
invariant. Hence the restriction of ${\rm Aut}(G)^c$ to $L$ can be 
identified with a two dimensional subgroup of ${\rm Aut}(\mbf C)$ that 
preserves a half-plane. Such a subgroup clearly contains the restriction 
of $(T_t)$ to $L$ as a normal subgroup and this contradicts the assumption 
that $(T_t)$ is not normal in ${\rm Aut}(G)^c$. The conclusion is that 
$h_1(z_1)$ is a non-vanishing entire function. 

\medskip

Now $Y$ is well defined on $\mbf C^2$ and so the one-parameter subgroup 
generated by it maps the boundary of $G$ to itself. This implies that 
$Y$ is a tangential vector field, i.e.,
\[
Y( 2 \Re z_2 +  P(z_1, \ov z_1) ) = 0
\]
whenever $2 \Re z_2 = - P(z_1, \ov z_1)$. This simplifies as
\[
\Big(h_1(z_1) \;\pa P / \pa z_1 + h_2(z_1) \Big) + 
\exp(-2i \mu  P(z_1, \ov z_1)) \Big(\ov{h_1(z_1)} \;\pa P / \pa \ov 
z_1 + \ov{h_2(z_1)} \Big) = 0
\]
for all $z_1 \in \mbf C$.

\medskip

\noindent {\it Claim:} Define $F(z, \ov z) = \Big(h_1(z) \;\pa P / \pa 
z + h_2(z) \Big) + \exp(-2i \mu P(z, \ov z)) 
\Big(\ov{h_1(z)} \;\pa P / \pa \ov z + \ov{h_2(z)} \Big)$ a real 
analytic, complex valued function on the plane. If $F(z, \ov z) \equiv 0$ 
for all $z \in \mbf C$, then $h_1(z)$ must vanish somewhere.

\medskip

\noindent To see this let 
\[
h_1(z) = \sum_{j \ge 0} c_j z^j \;\; {\rm and} \;\; h_2(z) = \sum_{j \ge 
0} d_j z^j
\]
where $c_0 = h_1(0) \not= 0$ by assumption. Also write 
\[
P(z, \ov z) = \ov z^r q_r(z) + \ov z^{r + 1} q_{r + 1}(z) + \ldots
\]
where the sum is finite and the coefficients $q_r(z), q_{r + 1}(z), 
\ldots$ are holomorphic polynomials. Note that $r \ge 1$ since $P(z, \ov z)$
does not have harmonic terms. For the same reason none of 
$q_r(z), q_{r + 1}(z), \ldots$ have a constant term nor are any of them 
identically equal to a constant. Suppose that $r > 1$. Then the 
coefficient of $\ov z^{r - 1}$ in 
\begin{align*}
	F(z, \ov z) &= \Big( (c_0 + c_1 z + \ldots)(\ov z^r q'_r(z) + 
	\ov z^{r + 1} q'_{r + 1}(z) + \ldots) + (d_0 + d_1 z + \ldots) \Big)\\
		    &+ \Big(1- 2i \mu(\ov z^r q_r(z) + \ov z^{r + 1} q_{r 
	+ 1}(z) + \ldots) + \ldots \Big) \Big( (\ov c_0 + \ov c_1 \ov z + 
	\ldots)(r \ov z^{r - 1} q_r(z) + (r + 1)\ov z^r q_{r + 1}(z)\\
                    & + \ldots) + (\ov d_0 + \ov d_1 \ov z + \ldots) \Big)
\end{align*}
\noindent is $\ov d_{r - 1} + \ov c_0 r q_r(z)$ and this must be identically 
zero. Thus $q_r(z)$ is a constant which is not possible. It follows that $r = 
1$. In this case the holomorphic part of the above expansion, which must 
also be identically zero, can be shown to be $(h_2(z) + \ov c_0 q_1(z) + 
\ov d_0)$. Hence
\begin{equation}
h_2(z) = -(\ov c_0 q_1(z) + \ov d_0)
\end{equation}
which implies that $h_2(z)$ is a polynomial. The coefficient of $\ov z$, 
which must also be identically zero, is $(h_1(z) q'_1(z) + \ov c_1 q_1(z) 
+ 2 \ov c_0 q_2(z) + \ov d_1 - 2i \mu (\ov c_0 q_1(z) + \ov d_0))$. Hence
\begin{equation}
h_1(z) = \Big(2i \mu(\ov c_0 q_1(z) + \ov d_0) - (\ov c_1 q_1(z) + 2 \ov 
c_0 q_2(z) + \ov d_1) \Big) \Bigl/ q'_1(z)
\end{equation}
which implies that $h_1(z)$ is a rational function. If the degree of this 
rational function is atleast one then $h_1(z)$ will vanish somewhere. The 
only other possibility is that the rational function is a constant which 
forces $h_1(z) \equiv c_0$. 

\medskip

In case $h_1(z) \equiv c_0$, consider the new holomorphic function 
obtained by complexifying $F(z, \ov z)$, i.e., by replacing $\ov z $ by an 
independent complex variable $w$ we can consider
\[
\ti F(z, w) = \Big(c_0 \;\pa P / \pa z \;(z, w) + h_2(z) \Big) + \exp(-2i 
\mu P(z, w) \Big(\ov c_0 \;\pa P / \pa w \;(z, w) + \ov h_2(w) \Big) \in 
\cal O(\mbf C^2)
\]
where $\ov h_2(w) = \ov{h_2(\ov w)}$. Note that $\ti F(z, w)$ vanishes 
when $w = \ov z$ and so $\ti F(z, w) \equiv 0$. This shows that
\begin{equation}
\exp(-2i \mu P(z, w)) \equiv - \Big(c_0 \;\pa P / \pa z \;(z, w) + h_2(z) 
\Big) \Bigl/ \Big(\ov c_0 \;\pa P / \pa w \;(z, w) + \ov h_2(w) \Big).
\end{equation}
The left side above is the exponential of a non-constant polynomial $P(z, 
w)$ and hence there is $\lambda \in \mbf C^{\ast}$ such that the 
restriction of the left side to the line $L = \{w = \lambda z\}$ is 
non-constant. Moreover it is non-vanishing as well. However the right side 
in (3.4) is a rational function on $L$ and this must vanish somewhere. 
This is a contradiction and the claim follows.

\medskip

Now $(T_t)$ is normal in ${\rm Aut}(G)^c$ and let $(S_s)$ be a 
one-parameter subgroup that generates ${\rm Aut}(G)^c$ along with $(T_t)$. 
Then for all real $s, t$
\begin{equation}
S_s \circ T_t = T_{\theta(s, t)} \circ S_s
\end{equation}
for some smooth function $\theta(s, t)$. Composing with $T_t$ once more on 
the right gives
\[
S_s \circ T_{2t} = T_{\theta(s, t)} \circ S_s \circ T_t = T_{2 \theta(s, 
t)} \circ S_s
\]
which shows that $\theta(s, 2t) = 2 \theta(s, t)$ and inductively we get 
$\theta(s, mt) = m \theta(s, t)$ for all integers $m \ge 1$. Putting $t = 
0$ it follows that $\theta(s, 0) = 0$ for all $s$. Now for $t_1, t_2, s 
\in \mbf R$, (3.5) shows that
\[
S_s \circ T_{t_1 + t_2} = S_s \circ T_{t_1} \circ T_{t_2} = T_{\theta(s, 
t_1) + \theta(s, t_2)} \circ S_s
\]
which gives $\theta(s, t_1 + t_2) = \theta(s, t_1) + \theta(s, t_2)$ and 
hence that $\theta(s, t) = tf(s)$ for some smooth function $f(s)$. Again 
for $s_1, s_2, t \in \mbf R$, (3.5) shows that
\[
S_{s_1 + s_2} \circ T_t = S_{s_1} \circ S_{s_2} \circ T_t = T_{\theta(s_1, 
\theta(s_2, t))} \circ S_{s_1}
\]
which gives $\theta(s_1 + s_2, t) = \theta(s_1, \theta(s_2, t))$. In terms 
of $f(s)$ this means that $f(s_1 + s_2) = f(s_1)f(s_2)$. Hence $f(s) = 
\exp(\alpha s)$ for some non-zero real $\alpha$. If $S_s(z_1, z_2) = 
(S^1_s(z_1, z_2), S^2_s(z_1, z_2))$ then (3.5) is equivalent to
\begin{align*}
 S^1_s(z_1, z_2 + it) &= S^1_s(z_1, z_2), {\rm and}\\
 S^2_s(z_1, z_2 + it) &= S^2_s(z_1, z_2) + i \exp(\alpha s)t
\end{align*}
\noindent for all real $s, t$. As in proposition 3.1, the first equation 
forces $S^1_s(z_1, z_2) = S^1_s(z_1)$ and that $(S^1_s(z_1))$ is a 
non-trivial subgroup of ${\rm Aut}(\mbf C)$. The second equation implies 
that $\pa S^2_s / \pa z_2 = \exp(\alpha s)$, i.e., $S^2_s(z_1, z_2) = 
\exp(\alpha s) z_2 + f(s, z_1)$ where $z_1 \mapsto f(s, z_1)$ is entire 
for all $s \in \mbf R$. It follows that $(S_s)$ is a one-parameter 
subgroup of ${\rm Aut}(\mbf C^2)$ that preserves $G$. An argument 
similar to that in proposition 3.1 shows that $f(s, z_1)$ is a holomorphic 
polynomial for all $s$. Moreover after a change of coordinates in $z_1$, 
it follows that $S^1_s(z_1) = z_1 + is$ or $S^1_s(z_1) = \exp(\beta s) 
z_1$ for some $\beta \in \mbf C \sm \{0\}$.

\medskip

If $S_s(z_1, z_2) = (z_1 + is, \exp(\alpha s)z_2 + f(s, z_1))$ then it 
follows that
\[
P(z_1, \ov z_1) = 2 \exp(-\alpha s)(\Re f(s, z_1)) + \exp(-\alpha s) 
P(z_1 + is, \ov z_1 - is)
\]
for all $z_1 \in \mbf C$ and $s \in \mbf R$. This cannot hold by 
considering $P_N(z_1, \ov z_1)$ the homogeneous summand of highest degree 
in $P(z_1, \ov z_1)$. On the other hand, if $S_s(z_1, z_2) = 
(\exp(\beta s) z_1, \exp(\alpha s) z_2 + f(s, z_1))$ it follows that
\[
P(z_1, \ov z_1) = 2 \exp(- \alpha s)(\Re f(s, z_1)) + \exp(\alpha s) 
P(\exp(\beta s) z_1, \exp(\ov \beta s) \ov z_1)
\]
for all $z_1 \in \mbf C$ and $s \in \mbf R$. But $P(z_1, \ov z_1)$ 
has no harmonic summands and this forces
\[
P(z_1, \ov z_1) = \exp(-\alpha s) P(\exp(\beta s) z_1, \exp(\ov 
\beta s) \ov z_1).
\]
Since $P(z_1, \ov z_1)$ is not identically zero there is at least one 
non-vanishing homogeneous summand in it. Let the degree of this summand be 
$k > 0$. By comparing terms of degree $k$ on both sides it follows that 
$\alpha = k(\Re \beta)$. This shows that $P(z_1, \ov z_1) = P_k(z_1, \ov 
z_1)$ is homogeneous of degree $k$. 

\medskip

If $k$ is odd it is known that the envelope of holomorphy of the model 
domain
\[
G = \{(z_1, z_2) \in \mbf C^2 : 2 \Re z_2 + P_k(z_1, \ov z_1) < 0\}
\]
contains a full open neighbourhood of the origin in $\mbf C^2$. Let $f : D 
\ra G$ be a biholomorphism. Then $f^{-1}$ extends holomorphically 
across $0 \in \pa G$ and it follows that $f^{-1}(0) \in \pa D$ belongs 
to $\hat D$. Now note that the dilations $S_s(z_1, z_2) = (\exp(s/k) z_1, 
\exp(s) z_2), \; s \in \mbf R$ are automorphisms of $G$ that cluster 
at the origin and thus the orbit of 
$f^{-1}(z_1, z_2) \in D$ under the one-parameter group $f^{-1} \circ S_s 
\circ f \in {\rm Aut}(D)$ clusters at $f^{-1}(0) \in \hat D$. Since $D$ is 
bounded, this contradicts the Greene-Krantz observation mentioned earlier. 
Hence $k = 2m$ is even.
\end{proof}

\section{Model domains when ${\rm Aut}(D)$ is three dimensional}

\noindent Consider the natural action of ${\rm Aut}(D)^c$ on $D$. It was 
shown in \cite{I1} that for every $p \in D$, the ${\rm Aut}(D)^c$-orbit 
$O(p)$, which is a connected closed submanifold in $D$, has real 
codimension $1$ or $2$. In case $O(p)$ 
has codimension $2$, then it is either a complex curve biholomorphically 
equivalent to $\Delta$, the unit disc in the plane or else maximally 
totally real in $D$. In case the codimension is $1$, then $O(p)$ is either 
a strongly pseudoconvex hypersurface or else Levi flat everywhere. In the 
latter case the leaves of the Levi foliation are biholomorphically 
equivalent to $\Delta$ and the explicit analysis in proposition 4.1 of 
\cite{I1} shows that each leaf is closed in $D$. Moreover it was also 
shown that if there are no codimension $1$ orbits then $D \backsimeq 
\Delta \times R$ where $R$ is any hyperbolic Riemann surface with a 
discrete group of automorphisms. In our case $R$ is forced to be simply 
connected since $D$ is so and hence $R \backsimeq \Delta$. This is a 
contradiction. Thus there is at least one codimension $1$ orbit in $D$.

\medskip

The main goal in this section will be to identify which of the examples 
that occur in \cite{I1} can be equivalent to $D$ as in the main theorem. 
It will turn out that $D \backsimeq \cal D_4 = \{(z_1, z_2) \in \mbf C^2 : 
2 \Re z_2 + (\Re z_1)^{2m} < 0\} \backsimeq R_{1/2m, -1, 1}$, which is 
defined below.
For this purpose the classification in \cite{I1} is divided into groups, 
the examples in each group having a distinguishing property. Whether $D$ 
can sustain such properties is then checked case by case. Only the 
relevant properties of the examples have been listed for ready reference 
in the following subsections. The reader is referred 
to \cite{I1} for more detailed information and it must be pointed out that 
the notation used to describe the domains below is the same as in 
\cite{I1}. Moreover unless stated otherwise, the word `orbit' in what 
follows will always refer to the ${\rm Aut}(D)^c$-orbit.

\subsection{Some examples with Levi flat orbits:}

The following domains admit Levi flat orbits:

\vskip8pt

\noindent $\bullet$ Fix $b \in \mbf R$, $b \not=  0, 1$ and let 
$-\infty \le s < 0 < t \le \infty$ where at least one of $s, t$ is finite. 
Define
\[
R_{b, s, t} = \Big\{ (z_1, z_2) \in \mbf C^2: s(\Re z_1)^b < \Re z_2 < 
t(\Re z_1)^b, \; \Re z_1 > 0 \Big\}.
\]
which has a unique Levi flat orbit given by
\[
{\cal O}_1 = \Big \{(z_1, z_2) \in \mbf C^2 : \Re z_1 > 0, \; \Re z_2 = 
0 \Big\}.
\]
Note that $R_{1/2, s, -s} \backsimeq \mbf B^2$ for all $s < 0$ and 
therefore these values of the parameters $b, s, t$ will not be considered. 

\vskip8pt

\noindent $\bullet$ For $b > 0, b \not= 1$ and $-\infty < s < 0 < t < 
\infty$ define
\[
\hat R_{b, s, t} = R_{b, s, \infty} \cup \Big \{(z_1, z_2) \in \mbf C^2 : 
\Re z_2 > t (- \Re z_1)^b, \; \Re z_1 < 0 \Big\} \cup {\hat{\cal O}}_1
\]
where
\[
R_{b, s, \infty} = \Big \{(z_1, z_2) \in \mbf C^2 : s(\Re z_1)^b < \Re 
z_2, \; \Re z_1 > 0 \Big\}
\]
and 
\[
\hat{\cal O}_1 = \Big\{ (z_1, z_2) \in \mbf C^2 : \Re z_1 = 0, \; \Re z_2 
> 0 \Big\}.
\]

\vskip8pt

\noindent $\bullet$ For $-\infty < t < 0 < s < \infty$ define
\[
\hat U _{s, t} = U_{s, \infty} \cup \Big\{ (z_1, z_2) \in \mbf C^2 : \Re 
z_1 > \Re z_2 \cdot \ln(t \; \Re z_2), \; \Re z_2 < 0 \Big\} \cup {\cal 
O}_1
\]
where 
\[
U_{s, \infty} = \Big\{ (z_1, z_2) \in \mbf C^2  : \Re z_2 \cdot \ln(s \; 
\Re z_2) < \Re z_1, \; \Re z_2 > 0 \Big\}
\]
and ${\cal O}_1$ is as in the definition of $R_{b, s, t}$ above.

\vskip8pt

\begin{prop}
Let $D$ be as in the main theorem. If $D$ admits a Levi flat orbit then it 
is equivalent to $\cal D_4 = \{(z_1, z_2) \in \mbf C^2 : 
2 \Re z_2 + (\Re z_1)^{2m} < 0\}$.
\end{prop}

\noindent Recall that $D$ is not pseudoconvex near $p_{\infty}$. By 
\cite{Ber1} this is equivalent to saying that none of the orbits in $D$ 
accumulate at the weakly pseudoconvex points near $p_{\infty}$. The proof 
depends on several intermediate steps. First assume 
that the orbit accumulation point $p_{\infty} \in T_1$. Choose coordinates 
around $p_{\infty} = 0$ so that $T_1$ coincides with the imaginary 
$z_1$-axis near the origin and fix a polydisk $U = \{\vert z_1 \vert 
< \eta, \vert z_2 \vert < \eta\}$ centered at the origin with $\eta > 0$ 
small enough so that $U \cap T_1$ is an embedded piece of the $y_1$-axis. 
The function $\tau(z) = (\Re z_1)^2 + \vert z_2 \vert^2$ is then a 
non-negative, strongly plurisubharmonic function in $U$ whose zero locus 
is exactly $U \cap T_1$. Also note that $i\pa \ov \pa \tau = 
i\pa \ov \pa \vert z \vert^2$. For $s > 0$ and $E \subset \mbf C^2$ 
let $\cal H_s(E)$ denote the $s$-dimensional Hausdorff measure of $E$. 
Finally for a sequence of closed sets $E_j$ in some domain 
$W \subset \mbf C^2$, define their cluster set ${\rm cl}(E_j)$ as
\[
{\rm cl}(E_j) = \{z \in W : {\rm there \; exists} \;z_j \in E_j \;{\rm 
such \; that} \; z \;{\rm is \; an \; accummulation \; point \; of \; the 
\; sequence} \;\{z_j\}\}.
\]

\noindent The relevant case will be the one in which the $\{E_j\}$ are 
complex analytic sets in $W$ of a pure fixed dimension. In this situation the 
following theorem that was proved by Diederich-Pinchuk in \cite{DP1} will 
be useful.

\medskip

\noindent{\it Let $M \subset W \subset \mbf C^2 $ be a smooth real 
analytic hypersurface of finite type. Let $E_j \subset W$ be closed complex 
analytic sets of pure dimension 1. Then ${\rm cl}(E_j)$ is not contained
in $M$}.

\begin{prop}
Let $A_j \subset U \cap D$ be a sequence of pure one dimensional 
complex analytic sets such that $\pa A_j \subset \pa U \cap 
D$ for all $j$. Assume that there are points $a_j \in A_j$ such that $a_j 
\ra a \in U \cap T_1$. Then there exists a compact set $K \subset U$ 
disjoint from $T_1$ such that $A_j \cap K \not= \emptyset$ for all large $j$. 
\end{prop}

\begin{proof}
Observe that each $A_j \subset U$ is analytic since none of them clusters 
at points of $U \cap \pa D$. The domain $\Om_r = \{(z \in U : \tau(z) < 
r\}$ is a strongly pseudoconvex tubular neighbourhood of $U \cap T_1$. Let 
$0 < r \ll \eta$ whose precise value will be fixed later. Then only 
finitely many $A_j$ can be contained in $\Om_r$. Indeed if this does not 
hold then there is a sequence, still to be denoted by $A_j$, for which 
each $A_j \subset \Om_r$. Define $\varrho_j(z) = \tau(z) - \vert z - a_j 
\vert^2 / 2$ and note that $\varrho_j(a_j) = \tau(a_j) > 0$. Moreover
\[
i \pa \ov \pa \varrho_j = i \pa \ov \pa \tau - i \pa \ov \pa \vert z - a_j 
\vert^2 / 2 = i \pa \ov \pa \tau - i \pa \ov \pa \vert z \vert^2 / 2 = i 
\pa \ov \pa \vert z \vert^2 / 2 > 0
\]
shows that the restriction of $\varrho_j$ to $A_j$ is subharmonic for all 
$j$. Now fix $j$ and let $w \in \pa A_j \subset \pa U \cap D$. Then
\[
\varrho_j(w) = \tau(w) - \vert w - a_j \vert^2 /2 \le r - \vert w - a_j 
\vert^2 / 2 < 0
\]
the last inequality holding whenever $r >0$ is chosen to satisfy
\[
2r < (\eta - \vert a \vert)^2 \approx (\eta - \vert a_j \vert)^2 \le 
(\vert z \vert - \vert a_j \vert)^2 \le \vert z - a_j \vert^2
\]
for all $z \in \pa U$. This contradicts the maximum principle and hence 
all except finitely many $A_j$ must intersect $U \cap \{ \tau(z) \ge r 
\}$.
\end{proof}

\noindent {\it Remark:} The above lemma is a version for sequences of 
analytic sets of the well known fact (see \cite{Ch2} for example) that an 
analytic set of positive dimension cannot approach a totally real manifold 
tangentially. Moreover the choice of $r$ depended upon the distance of the 
limit point $a$ from $\pa U$. It is clear from the proof that a uniform 
$r$ can be chosen if $a$ is allowed to vary in a relatively compact subset 
of $U \cap T_1$. Finally, the theorem on the cluster set of analytic sets 
mentioned above could have been used at this stage to conclude that 
${\rm cl}(A_j)$ intersects $U^-$. However a more precise description of 
the sub-domain in $U^-$ that intersects ${\rm cl}(A_j)$ is afforded by the 
above lemma and this will be needed in the sequel.

\begin{lem}
Let $q \in D$ be such that $O(q) \subset D$ is Levi flat. Fix a ball 
$B_w(r)$ around some $w \in \pa D$ for some $r > 0$ and suppose that $O(q) 
\cap B_w(r) \not= \emptyset$. Let $C_q \subset O(q) \cap B_w(r)$ be a 
connected component. Then $C_q$ is itself a closed Levi flat hypersurface 
that admits a codimension one foliation by closed hyperbolic Riemann 
surfaces each of which is itself closed in $B_w(r) \cap D$. Suppose further 
that there is a leaf $R$ in $C_q$ that does not cluster at points of 
$B_w(r) \cap \pa D$. Then there is an arbitrarily small neighbourhood 
$W$ of $R$, $W$ open in $\mbf C^2$ and compactly contained in $D$ with the 
property that if $S \subset C_q$ is leaf with $S \cap W \not= \emptyset$ 
then $S \subset W$.
\end{lem}

\begin{proof}
$O(q)$ has all the above mentioned properties and hence $C_q$ inherits 
them as well. Now choose an arbitrary neighbourhood $\ti W$ of $R$, $\ti 
W$ open in $\mbf C^2$ and compactly contained in $D$. If needed shrink it 
so that $C_q \cap \ti W$ is connected. Note that the leaves of $C_q \cap 
\ti W$ are exactly the CR orbits on $C_q \cap \ti W$. Choose a real one 
dimensional section $T \subset C_q \cap \ti W$ that passes through some $r 
\in R$ and which is transverse to the leaves of $C_q \cap \ti W$ near $r$. 
Since the foliation of $C_q$ exists in a neighbourhood of the closure of 
$\ti W$, it follows that if $T$ is small enough, the CR orbits of 
$C_q \cap \ti W$ through points in $T$ stay close to $R$. The existence of 
$W \subset \ti W$ follows.
\end{proof}

\begin{prop}
Suppose there exists $q \in D$ such that $\ov {O(q)} \cap (U \cap T_1) 
\not= \emptyset$. Then it is not possible to find a component 
of $O(q) \cap U$, say $C_q$ such that $\ov C_q \cap (U \cap T_1) \not= 
\emptyset$. In particular the number of components of $O(q) \cap U$ is not 
finite.
\end{prop}

\begin{proof}
First observe that $O(q)$ cannot cluster at points of $(U \cap \pa D) \sm 
T_1$ for these contain pseudoconvex or pseudoconcave points and while all 
model domains are known in the former case, the latter possibility is 
ruled out as these points belong to $\hat D$. The hypotheses therefore 
imply that $\ov{O(q)} \cap U \subset T_1$. To argue by contradiction, 
suppose that $C_q \subset O(q) \cap U$ is a component that satisfies $\ov 
C_q \cap (U \cap T_1) \not= \emptyset$. Pick $a_j \in C_q$ such that $a_j 
\ra a \in U \cap T_1$. Two cases now arise:

\medskip

\no {\it Case 1:} None of the leaves of $C_q$ clusters at points of $U 
\cap T_1$.

\medskip

Let $R_j$ be the leaves of $C_q$ that contain $a_j$ for $j \ge 1$. It 
follows from proposition 4.2 that each $R_j$ (except finitely many which 
can be ignored) must intersect $U^- \cap \{\tau(z) \ge r\}$ for a fixed 
small $r > 0$. For $0 < \delta \ll r$, let 
\[
U^-_{\delta} = \{z \in U^- : - \delta < \rho(z) < 0\}
\]
where $\rho(z)$ is a defining function for $U \cap \pa D$. $U^-_{\delta}$ 
is then a one-sided collar around $U \cap \pa D$ of width $\delta$. Since 
$C_q$ cannot cluster at points of $(U \cap \pa D) \sm T_1$ it follows that 
each $R_j$ must intersect
\begin{equation}
K^-_{r, \delta} = (U^- \cap \{ \tau(z) \ge r \}) \sm U^-_{\delta}
\end{equation}
which is compact in $D$. Choose $\alpha_j \in R_j \cap K^-_{r, \delta}$ 
and let $\alpha_j \ra \alpha \in K^-_{r, \delta}$ after perhaps passing to 
a subsequence. Since $C_q$ is closed in $U \sm T_1$ it follows that 
$\alpha \in C_q$. Let $R_{\alpha}$ be the leaf of $C_q$ that contains 
$\alpha$. By the hypothesis of case 1, $R_{\alpha}$ does not cluster at 
$U \cap T_1$ and hence it stays uniformly away from $U \cap \pa D$. 
By lemma 4.3 there is a neighbourhood $W$ of $R_{\alpha}$ that is compactly 
contained in $D$ such that $R_j \subset W$ for all large $j$. But then 
$a_j \in R_j \subset W$ and this contradicts the fact that 
$a_j \ra a \in U \cap T_1$.

\medskip

\no {\it Case 2:} There is at least one leaf $R \subset C_q$ that clusters 
at points of $U \cap T_1$.

\medskip

There are two subcases to consider. First if $\cal H_1(\ov R \cap U 
\cap T_1) = 0$ it follows from Shiffman's theorem that $\ov R \subset U$ 
is a closed, one dimensional complex analytic set. But then $\ov R \subset 
\ov{U^-}$ and so the strong disk theorem shows that all 
points in $\ov R \cap U \cap T_1$ lie in the envelope of holomorphy of 
$D$. This is a contradiction. Second if $\cal H_1(\ov R \cap U \cap T_1) > 
0$ then $R$ admits analytic continuation across $T_1$, i.e., there is a 
neighbourhood $V$ of $U \cap T_1$ and a closed complex analytic set 
$A \subset V$ of pure dimension one such that $R \cap V \subset A$. In 
fact the uniqueness theorem shows that $A = T_1^{\mbf C}$ the 
complexification of $U \cap T_1$ and $R \cup T_1^{\mbf C}$ is analytic in 
$U^- \cup V$. In particular $U \cap T_1 \subset \ov R$.

\medskip

\no Let $R'$ be a leaf in $C_q$ that is distinct from $R$. Then $R'$ 
cannot cluster at $U \cap T_1$. Indeed if $\cal H_1(\ov R' \cap U \cap 
T_1) = 0$ then as before all points in $\ov R' \cap U \cap T_1$ will be in 
the envelope of holomorphy of $D$ which cannot happen or else both $R'$ 
and $R$ have the same analytic continuation, namely $T_1^{\mbf C}$. The 
uniqueness theorem shows that $R' = R$ which is a contradiction. The 
conclusion is that no leaf apart from $R$ can cluster at $U \cap T_1$.

\medskip

By the remark before lemma 4.3 it is possible to choose $r > 0$ and $0 < 
\delta \ll r$ such that $R \cap K^-_{r, \delta} \not= \emptyset$ where 
$K^-_{r, \delta}$ is as in (4.1). Fix $\alpha \in R \cap K^-_{r, \delta}$ 
and choose $\alpha_j \in C_q \cap K^-_{r, \delta}$ such that $\alpha_j \ra 
\alpha$. Let $R_j$ be the leaves of $C_q$ that contain $\alpha_j$. Choose 
$\beta_j \in R$ such that $\beta_j \ra \beta \in U \cap T_1$. By the proof 
of lemma 4.3 it is possible to find a subsequence, still denoted by $R_j$, 
and points $r_j \in R_j$ such that $\vert r_j - \beta_j \vert \ra 0$. This 
implies that $r_j \ra \beta$. Since $R_j \cap K^-_{r, \delta} \not= 
\emptyset$ for all $j$, it follows that their cluster set 
${\rm cl}(R_j) \not= \emptyset$ in $U^-$. Choose 
$c \in {\rm cl}(R_j) \cap K^-_{r, \delta}$ and let $R_c$ be the leaf of 
$C_q$ that contains it. If $R_c$ does not cluster at $U \cap \pa D$ then lemma 
4.3 shows that $R_j$ are compactly contained in $D$ contradicting the fact 
that $r_j \in R_j$ is such that $r_j \ra \beta$. Hence $R_c$ must cluster 
at $U \cap T_1$ and therefore $R_c = R$. This shows that the cluster set 
${\rm cl}(R_j)$ in $U^-$ is exactly $R$.

\medskip

To conclude translate coordinates so that $\beta = 0$ and let $L$ be a 
complex line that is tangent to $T_1^{\mbf C}$ at the origin. Since $R$ 
coincides with $T_1^{\mbf C}$ near the origin it follows that $L$ must 
intersect $U^-$. Choose a small polydisk $U_1 \times U_2$ around the 
origin such that $U_1$ is parallel to $L$. The projection
\[
\pi : T_1^{\mbf C} \cap (U_1 \times U_2) \ra U_1
\]
is then proper for an appropriate choice of $U_1, U_2$. This is equivalent 
to the condition that $T_1^{\mbf C}$ has no limit points on $U_1 \times 
\pa U_2$. Now $R_j \cap (U_1 \times U_2)$ are pure one dimensional 
analytic sets in $U^- \cap (U_1 \times U_2)$ that do not cluster at 
$U \cap \pa D$. Hence they are analytic in $U_1 \times U_2$. Moreover the 
cluster set of $\{R_j \cap (U_1 \times U_2)\}$ in $U^-$ is exactly $R \cap 
U^- \cap (U_1 \times U_2) = T_1^{\mbf C} \cap U^- \cap (U_1 \times U_2)$. 
There can be no other points in $(U \cap \pa D) \sm T_1$ that lie in 
${\rm cl}(R_j \cap (U_1 \times U_2))$ as the Diederich-Pinchuk theorem 
mentioned above shows and hence $R_j \cap (U_1 \times U_2)$ has no limit 
points on $U_1 \times \pa U_2$ for all large $j$, i.e., the projection
\[
\pi : R_j \cap (U_1 \times U_2) \ra U_1
\]
is proper for all large $j$. Therefore $\pi(R_j \cap (U_1 \times U_2)) = 
U_1$ and this forces $R_j$ to cluster at points of $U \cap \pa D$ as the disk 
$U_1$ intersects both $U^{\pm}$. This is a contradiction.
\end{proof}

\no {\it Proof of Proposition 4.1}: Suppose $q \in D$ is such that $O(q)$ 
is Levi flat. By lemma 2.2 it is known that $\phi_j(q) \ra p_{\infty} = 0$ 
after a translation of coordinates. Two cases need to be considered.

\medskip

\no {\it Case 1:} For large $j$ none of the $\phi_j$'s belong to ${\rm 
Aut}(D)^c$.

\medskip

Since ${\rm Aut}(D)^c$ is normal in ${\rm Aut}(D)$ it follows that 
$\phi_j(O(q)) = O(\phi_j(q))$ for all $j$. Thus $O(\phi_j(q))$ is a 
family of distinct Levi flat hypersurfaces for all large $j$. Fix an 
arbitrarily small neighbourhood $U$ of the origin. Then 
observe that $O(\phi_j(q)) \cap U^- \not= \emptyset$ for $j$ large and let 
$C_j \subset O(\phi_j(q)) \cap U^-$ be the connected components that 
contain $\phi_j(q)$. The proof of proposition 4.4 shows that none of the 
$C_j$'s clusters at $U \cap T_1$. Let $R_j \subset C_j$ be the leaves such 
that $\phi_j(q) \in R_j$. Then $R_j \subset U$ are pure one dimensional 
analytic sets and it follows from the Diederich-Pinchuk result and 
proposition 4.2 that ${\rm cl}(R_j) \cap U^- \cap \{ \tau(z) \ge r \} 
\not= \emptyset$ for a suitable $r > 0$. In fact more can be said about 
${\rm cl}(R_j)$; indeed let $a \in {\rm cl}(R_j) \cap U^-$ and pick $a_j 
\in R_j$ such that $a_j \ra a$ after perhaps passing to a subsequence. 
Choose $f \in {\rm Aut}(D)^c$ such that $f_j \circ \phi_j(q) = a_j$. Then 
$\{ f_j \circ \phi_j \} \ra f \in {\rm Aut}(D)$ since $f_j \circ \phi_j(q) 
= a_j \ra a \in U^-$. It follows that ${\rm cl}(R_j) \cap U^-$ is 
contained in the ${\rm Aut}(D)$-orbit of $q$. 

\medskip

This observation now shows that ${\rm cl}(R_j) \cap ((U \cap \pa D) \sm 
T_1) = \emptyset$ as otherwise the orbit of $q$ will cluster at either 
pseudoconvex or pseudoconcave points. All model domains are known in the 
former case while the latter possibility does not hold. Therefore it is 
possible to find $r > 0$ and $0 < \delta \ll r$ such that ${\rm cl}(R_j) 
\cap K^-_{r, \delta} \not= \emptyset$ where $K^-_{r, \delta}$ is as in 
(4.1). Choose $c_j \in R_j \cap K^-_{r, \delta}$ such that $c_j \ra c_0 
\in {\rm cl}(R_j) \cap K^-_{r, \delta}$ after passing to a subsequence and 
re-indexing and let $d_j \in R_j$ converge to $d_0 \in U^-$. Let $g_j \in 
{\rm Aut}(D)^c$ be such that $g_j(c_j) = d_j$. Let $g$ be the holomorphic 
limit of the $g_j$'s and observe that $g(c_0) = d_0$. Therefore $g \in 
{\rm Aut}(D)^c$, this being closed in ${\rm Aut}(D)$. This shows that 
${\rm cl}(R_j) \cap U^- \subset O(c_0) \cap U^-$. This strengthens the 
observation made above that ${\rm cl}(R_j) \cap U^-$ is contained in the 
${\rm Aut}(D)$-orbit of $q$. Moreover since the analytic sets $\{ R_j \}$ 
contain points arbitrarily close to the origin it follows that $O(c_0)$ 
also clusters there. 

\medskip

To conclude, note that there are infinitely many distinct components of 
$O(c_0) \cap U^-$ by proposition 4.4. If $a_j \in O(c_0) \cap U^-$ 
converges to the origin, we can consider the distinct components $S_j 
\subset O(c_0) \cap U^-$ that contain $a_j$ and the one dimensional 
analytic sets $A_j \subset S_j$ that contain $a_j$. By proposition 4.2 all 
except finitely many $A_j$ (and hence the same for $S_j$) intersect 
$K^-_{r, \delta}$ for some $r > 0$ and $0 < \delta \ll r$. This argument 
shows that there are infinitely many components of $O(c_0) \cap U^-$ that 
intersect a compact subset of $D$ and this forces $O(c_0)$ to cluster on 
itself in $D$. This is a contradiction since $O(c_0)$ is a closed 
submanifold of $D$.

\medskip

\no {\it Case 2:} After passing to a subsequence and re-indexing, all the 
$\phi_j$'s are contained in ${\rm Aut}(D)^c$.

\medskip

In this case it can be seen that $O(q)$ itself clusters at the origin. By 
proposition 4.4 there are infinitely many distinct components of $O(q) 
\cap U^-$ and now the arguments in the last paragraph of case 1 apply to 
provide a contradiction. In particular the intermediate step about finding 
$c_0 \in K^-_{r, \delta}$ is not needed.

\medskip

\noindent The steps leading up to this point were all proved under the 
assumption that $p_{\infty} \in T_1$. In case $p_{\infty} \in T_0$, let 
$\gamma \subset U \cap \pa D$ be a germ of an embedded real analytic arc 
that contains $p_{\infty}$ in its interior and which intersects any 
stratum in $T_1$ clustering at $p_{\infty}$ in a discrete set only. Such a 
choice is possible by the local finiteness of the semi-analytic 
stratification of ${\cal B}$, the border between the pseudoconvex and 
pseudoconcave points. The complement of $\gamma$ in $U \cap \pa D$ near 
$p_{\infty}$ consists of pseudoconvex or pseudoconcave points or those 
that lie on $T_1$. The arguments mentioned above show that if there is 
a Levi flat orbit in $D$ then it cannot cluster at points on $T_1$. The 
same steps now can be applied with $T_1$ replaced by $\gamma$ and this 
finally shows that if $D$ is not pseudoconvex then it 
cannot admit Levi flat orbits. It should be noted that the arguments used 
to obtain this conclusion did not use the fact that $\dim{\rm Aut}(D) = 
3$. Only the existence of a Levi flat orbit with closed leaves each of 
which is equivalent to the unit disc was used.

\medskip

To continue, if there is one such orbit then 
$D$ must be a pseudoconvex domain, i.e., the boundary $\pa D$ 
near $p_{\infty}$ is weakly pseudoconvex and of finite type. It follows by 
\cite{Ber1} that $D \backsimeq \ti D = \{(z_1, z_2) \in \mbf C^2 : 2 \Re 
z_2 + P_{2m}(z_1, \ov z_1) < 0\}$ where $P_{2m}(z_1, \ov z_1)$ is a 
homogeneous subharmonic polynomial of degree $2m$ without harmonic terms. 
Now observe that $\ti D$ is invariant under the one-parameter subgroups
$T_t(z_1, z_2) = (z_1, z_2 + it)$ and $S_s(z_1, z_2) = (\exp(s/2m)z_1, 
\exp(s)z_2)$ and the corresponding real vector fields are $X = \Re(i \; 
\pa / \pa z_2)$ and $Y = \Re((z_1/2m) \; \pa / \pa z_1 + z_2 \pa / \pa z_2)$.
It can be seen that $[X, Y] = X$. Let $\mathfrak{h} \subset 
\mathfrak{g}(\ti D)$ be the Lie subalgebra generated by $X, Y$. 
Let $X, Y, Z$ be the generators of $\mathfrak{g}(\ti D)$. Using the 
Jacobi identity it can be seen that $\mathfrak{g}(\ti D)$ must be 
isomorphic to one of the following:

\begin{enumerate}
\item[(a)] $[X, Y] = X, \; [Z, X] = 0, \; [Z, Y] = \alpha Z$ \;for some
real $\alpha$
\item[(b)] $[X, Y] = X, \; [Z, X] = 0, \; [Z, Y] = X + Z$
\item[(c)] $[X, Y] = X, \; [Z, X] = Y, \; [Z, Y] = -Z$
\end{enumerate}

There are two possibilities that arise in case (a), namely when $\alpha =
0$ and when $\alpha \not= 0$. In the former case, lemma 5.3 of \cite{O}
can be applied to show that $\ti D \backsimeq \{(z_1, z_2) \in \mbf C^2 :
2 \Re z_2 + \vert z_1 \vert^{2m} < 0\}$ and thus $\dim {\rm Aut}(\ti D) =
4$ which is a contradiction. In the latter case, lemma 5.6 of \cite{O}
shows that $\ti D \backsimeq \{(z_1, z_2) \in \mbf C^2 : 2 \Re z_2 + (\Re
z_1)^{2m} < 0\}$ which is isomorphic to 
\[
R_{1/2m, -1, 1} = \Big\{ (z_1, z_2) \in \mbf C^2 : -(\Re z_1)^{1/2m} < \Re 
z_2 < (\Re z_1)^{1/2m}, \; \Re z_1 > 0 \Big\}.
\]
This domain has a unique Levi flat orbit namely $\cal O_1$. The structure of 
$\mathfrak{g}(\ti D)$ in case (b) is similar to that of case (a) when 
$\alpha \not= 0$. As above we have a contradiction. Finally in case (c) 
we have $\mathfrak{g}(\ti D) \backsimeq \mathfrak{so}_{2, 1}(\mbf R) 
\backsimeq \mathfrak{sl}_2(\mbf R)$. A detailed calculation for this case 
has been done in lemma 5.8 in \cite{O} (for this situation only 
the calculations done there are needed; none of the arguments that lead up 
to lemma 5.8 are needed here as the conditions are fulfilled in our case)
which again shows that $\ti D \backsimeq \{(z_1, z_2) \in \mbf C^2 : 2 \Re
z_2 + \vert z_1 \vert^{2m} < 0\}$ and hence that $\dim {\rm Aut}(\ti D) =
4$. This is not possible.

\vskip10pt

\noindent {\it Remark:} Proposition 4.1 shows that $D$ cannot be 
equivalent to either $\hat R_{b, s, t}$ or 
$\hat U_{s, t}$. A different class of examples 
with Levi flat orbits are also mentioned in \cite{I1}, namely those that 
are obtained by considering finite and infinite sheeted covers of $D_{s, t}$ 
and $\Om_{s, t}$. These will be defined later.

\vskip10pt

\subsection{Two tube domains:} Here we consider the following examples 
each of which is a tube domain over an unbounded base in the $(\Re z_1, 
\Re z_2)$ plane.

\vskip8pt

\noindent $\bullet$ For $0 \le s < t \le \infty$ where either $s > 0$ or 
$t < \infty$ define
\[
U_{s, t} = \Big \{(z_1, z_2) \in \mbf C^2 : \Re z_2 \cdot \ln(s \Re z_2) < 
\Re z_1 < \Re z_2 \cdot \ln(t \Re z_2), \; \Re z_2 > 0 \Big\}.
\]
The group ${\rm Aut}(U_{s, t})$ is connected and is generated by  
one-parameter subgroups of the form 
\begin{align*}
 \sigma^1_t(z_1, z_2) &= (\exp(t)\; z_1 + t \exp(t)\; z_2, \exp(t)\; z_2),\\
 \sigma^2_{\beta}(z_1, z_2) &=  (z_1 + i \beta, z_2),\\
 \sigma^3_{\gamma}(z_1, z_2) &=  (z_1, z_2 + i \gamma)
\end{align*}
where $t, \beta, \gamma \in \mbf R$. The holomorphic vector 
fields corresponding to these are
\[
X = (z_1 + z_2)\; \pa / \pa z_1 + z_2 \; \pa / \pa z_2, \;\; Y = i\; \pa / 
\pa z_1, \;\; Z = i \;\pa / \pa z_2
\]
respectively and they satisfy the relations
\[
[X, Y] = - Y, \;\;\; [Y, Z] = 0, \;\;\; [Z, X] = Y + Z.
\]
The commutator ${\rm Aut}(U_{s, t})' \backsimeq (\mbf R^2, +)$ is 
generated by $\sigma^2_{\beta}, \sigma^3_{\gamma}$ and it is evident from 
the relations between $X, Y, Z$ that the subgroup $(\sigma_{\beta}) 
\subset {\rm Aut}(U_{s, t})'$ is normal in ${\rm Aut}(U_{s, t})$. 
Moreover, the Lie algebra $\mathfrak{g}(U_{s, t})$ has a unique two 
dimensional abelian subalgebra, namely the one generated by $Y, Z$.

\vskip8pt

\noindent $\bullet$ Fix $b > 0$ and for $0 < t < \infty$ and $\exp(-2 \pi 
b) t < s < t$ define 
\[
V_{b, s, t} = \Big \{(z_1, z_2) \in \mbf C^2 : s \exp(b \phi) < r < t 
\exp(b \phi) \Big\}
\]
where $(r, \phi)$ are polar coordinates in the $(\Re z_1, \Re z_2)$ plane 
and $\phi \in (-\infty, \infty)$. The group ${\rm Aut}(V_{b, s, t})$ is 
connected and is generated by one-parameter subgroups of the form
\begin{align*}
 \sigma^1_{\psi}(z_1, z_2) &= (\exp(b \psi) \cos \psi \; z_1 + \exp(b 
 \psi) \sin \psi \; z_2, - \exp(b \psi) \sin \psi \; z_1 + \exp(b \psi) 
 \cos \psi \; z_2)\\
 \sigma^2_{\beta}(z_1, z_2) &= (z_1 + i \beta, z_2)\\
 \sigma^3_{\gamma}(z_1, z_2) &= (z_1, z_2 + i \gamma)
\end{align*}
where $\psi, \beta, \gamma \in \mbf R$. The holomorphic vector fields 
corresponding to these are
\[
X = (b z_1 + z_2) \; \pa / \pa z_1 + (-z_1 + b z_2) \; \pa / \pa z_2, \; 
\; Y = i \; \pa / \pa z_1, \; \; Z = i \; \pa / \pa z_2
\]
respectively and they satisfy the relations
\[
[X, Y] = -b Y + Z, \;\;\; [Y, Z] = 0, \;\;\;[Z, X] = Y + b Z.
\]
As before ${\rm Aut}(V_{b, s, t})' \backsimeq (\mbf R^2, +)$ is generated 
by $\sigma^2_{\beta}, \sigma^3_{\gamma}$. The matrix 
formed by the non-trivial structure constants
\[
M = 
\begin{pmatrix}
-b & 1\\
1 & b
\end{pmatrix}
\]
is triangularisable over $\mbf R$ and hence there is a change of 
coordinates involving $Y, Z$ (and which does not affect $X$) after which 
it is possible to conclude that there is a one-parameter subgroup in ${\rm 
Aut}(V_{b, s, t})'$ that is normal in ${\rm Aut}(V_{b, s, t})$. Finally, 
the Lie algebra $\mathfrak{g}(V_{b, s, t})$ contains a unique two 
dimensional abelian subalgebra, namely the one generated by $Y, Z$.

\begin{prop}
$D$ cannot be equivalent to either $U_{s, t}$ or $V_{b, s, t}$.
\end{prop}

\begin{proof}
The arguments for either of $U_{s, t}$ or $V_{b, s, t}$ are the same and 
it will suffice to work with say $V_{b, s, t}$. So let 
\[
f : V_{b, s, t} \ra G \backsimeq D
\]
be a biholomorphism where $G$ is as in (2.3). For a subgroup $H$ of 
${\rm Aut}(V_{b, s, t})$ let $f_{\ast}H \subset {\rm Aut}(G)$ be the 
subgroup given by
\[
f_{\ast}H = \{f \circ h \circ f^{-1} : h \in H \}
\]
and likewise define $f^{\ast}S = (f^{-1})_{\ast}S$ if $S$ is a subgroup of 
${\rm Aut}(G)$. Let $(\psi_t)$ be the one-parameter subgroup in ${\rm 
Aut}(V_{b, s, t})'$ that is normal in ${\rm Aut}(V_{b, s, t})$. Recall 
that $G$ is invariant under the translations $(T_t)$ along the imaginary 
$z_2$-axis. 

\medskip

Suppose that $(T_t)$ is not contained in $f_{\ast}({\rm Aut}(V_{b, s, 
t})') \backsimeq (\mbf R^2, +)$. Let $N$ be the subgroup of 
${\rm Aut}(G)$ that is generated by $f_{\ast}(\psi_t)$ and $(T_t)$. Then 
$N$ is non-abelian and it contains $f_{\ast}(\psi_t)$ as a normal subgroup. 
This situation cannot happen as the proof of proposition 3.2 shows. So 
$(T_t) \subset f_{\ast}({\rm Aut}(V_{b, s, t})') \backsimeq (\mbf R^2, 
+)$. By proposition 3.1 it follows that
\[
G \backsimeq \ti G = \{(z_1, z_2) \in \mbf C^2 : 2 \Re z_2 + P(\Re 
z_1) < 0\}
\]
after a global change of coordinates and $P(\Re z_1)$ is a polynomial 
without harmonic terms that depends only on $\Re z_1$, i.e., $\ti G$ is 
itself a tube domain. Let $f$ still denote the equivalence between $V_{b, 
s, t}$ and $\ti G$. Evidently $f$ induces an isomorphism of the 
corresponding Lie algebras. Since there is a unique two dimensional 
abelian subalgebra in $\mathfrak{g}(V_{b, s, t})$, the same must be true 
of $\mathfrak{g}(\ti G)$. But the subalgebra generated by the 
translations along the imaginary $z_1, z_2$-axes is one such in 
$\mathfrak{g}(\ti G)$ and so it must be the only one. Thus the 
corresponding unique two dimensional abelian subalgebras must be mapped 
into each other. These abelian subalgebras are the Lie algebras of the 
subgroups formed by translations along the imaginary $z_1, z_2$-axes in 
both $V_{b, s, t}$ and $\ti G$. Being isomorphic to $(\mbf R^2, +)$ they 
are simply connected and hence the isomorphism between the Lie 
algebras extends to an isomorphism between the subgroups of 
imaginary translations. It follows (see for example \cite{Oe} or 
\cite{KS}) that
\[
f : V_{b, s, t} \ra \ti G
\]
must be affine. This however cannot happen as $\pa \ti G$ is algebraic 
(even polynomial) while $\pa V_{b, s, t}$ is evidently not.

\end{proof}

\subsection{Another tube domain and its finite and infinite sheeted 
covers:} For $0 \le s < t < \infty$ define
\[
\mathfrak{S}_{s, t} = \Big \{(z_1, z_2) \in \mbf C^2 : s < (\Re z_1)^2 + 
(\Re z_2)^2 < t \Big\}
\]
and for $0 < t < \infty$ define
\[
\mathfrak{S}_t = \Big\{ (z_1, z_2) \in \mbf C^2 : (\Re z_1)^2 + (\Re 
z_2)^2 < t \Big\}.
\]
Both ${\rm Aut}(\mathfrak{S}_{s, t})^c$ and ${\rm Aut}(\mathfrak{S}_t)^c$ 
are generated by maps of the form
\begin{align*}
\sigma^1_{\psi}(z_1, z_2) &= (\cos \psi \;z_1 - \sin \psi \; z_2, \cos 
\psi \;z_1 + \sin \psi \;z_2)\\
\sigma^2_{\beta}(z_1, z_2) &= (z_1 + i \beta, z_2)\\
\sigma^3_{\gamma}(z_1, z_2) &= (z_1, z_2 + i \gamma)
\end{align*}
where $\psi, \beta, \gamma \in \mbf R$. The holomorphic 
vector fields corresponding to these are
\[
X = - z_2 \; \pa / \pa z_1 + z_1 \; \pa / \pa z_2, \;\; Y = i \; \pa / \pa 
z_1, \;\; Z = i \; \pa / \pa z_2
\]
respectively and they satisfy the relations
\[
[X, Y] = -Z, \;\;\; [Y, Z] = 0, \;\;\; [Z, X] = -Y.
\]
Hence the one-parameter subgroup corresponding to $Y + Z$, for example, is 
contained in the commutator of ${\rm Aut}(\mathfrak{S}_t)^c$ and is normal 
in ${\rm Aut}(\mathfrak{S}_t)^c$. As before there is a unique two 
dimensional abelian subalgebra in $\mathfrak{g}(\mathfrak{S}_t)$, namely 
that generated by $Y, Z$.

\begin{prop}
D cannot be equivalent to either $\mathfrak{S}_{s, t}$ or $\mathfrak{S}_t$ 
for any $0 \le s < t < \infty$.
\end{prop}

\begin{proof}
The domain $\mathfrak{S}_{s, t}$ is not simply connected and hence $D$ 
cannot be equivalent to it. On the other hand, the properties of 
$\mathfrak{S}_t$ as listed above are similar to those of $U_{s, t}, V_{b, 
s, t}$. The arguments of proposition 4.5 can be applied in this case as 
well and they show that if $D \backsimeq \mathfrak{S}_t$, then there is an 
affine equivalence between $\ti G$ and $\mathfrak{S}_t$. However, the 
explicit description of $\pa \ti G$ and $\pa \mathfrak{S}_t$ shows that 
this is not possible.
\end{proof}

\noindent Next finite and infinite sheeted covers of $\mathfrak{S}_{s, 
t}$ can be considered. We start with the finite case first. For an integer 
$n \ge 2$, consider the $n$-sheeted covering self map $\Phi^{(n)}_{\chi}$ 
of $\mbf C^2 \sm \{\Re z_1 = \Re z_2 = 0\}$ whose components are given by
\begin{align*}
\ti z_1 &= \Re \Big( (\Re z_1 + i \Re z_2)^n \Big) + i \Im z_1,\\
\ti z_2 &= \Im \Big( (\Re z_1 + i \Re z_2)^n \Big) + i \Im z_2.
\end{align*}
Let $M^{(n)}_{\chi}$ denote the domain of $\Phi^{(n)}_{\chi}$ endowed with 
the pull-back complex structure using $\Phi^{(n)}_{\chi}$. For $0 \le s < 
t < \infty$ and $n \ge 2$ define
\[
\mathfrak{S}^{(n)}_{s, t} = \Big \{(z_1, z_2) \in M^{(n)}_{\chi} : s^{1/n} 
< (\Re z_1)^2 + (\Re z_2)^2 < t^{1/n} \Big\}.
\]
It can be seen that $\mathfrak{S}^{(n)}_{s, t}$ is an $n$-sheeted cover of 
$\mathfrak{S}_{s, t}$, the holomorphic covering map being $\Phi^{(n)}_{\chi}$.

\begin{prop}
D cannot be equivalent to $\mathfrak{S}^{(n)}_{s, t}$ for $n \ge 2$ and $0 
\le s < t < \infty$.
\end{prop}

\begin{proof}
Let $f : D \ra \mathfrak{S}^{(n)}_{s, t}$ be a biholomorphism. Since 
$\mathfrak{S}^{(n)}_{s, t}$ inherits the pull-back complex structure using 
$\Phi^{(n)}_{\chi}$ it follows that 
\[
\Phi^{(n)}_{\chi} \circ f : D \ra \mathfrak{S}_{s, t}
\]
is an unbranched, proper holomorphic mapping between domains that are 
equipped with the standard complex structure. First suppose that 
$0 < s < t < \infty$. The boundary $\pa \mathfrak{S}_{s, t}$ has two 
components, namely 
\[
\pa \mathfrak{S}^+_{s, t} = \{(z_1, z_2) \in \mbf C^2 : (\Re z_1)^2 + 
(\Re z_2)^2 = t\}, \; {\rm and}\;
\pa \mathfrak{S}^-_{s, t} = \{(z_1, z_2) \in \mbf C^2 : (\Re z_1)^2 + 
(\Re z_2)^2 = s\}.
\]
Using the orientation induced on these by $\mathfrak{S}_{s, t}$, it can be 
seen that $\pa \mathfrak{S}^{\pm}_{s, t}$ are strongly pseudoconvex and 
strongly pseudoconcave hypersurfaces respectively. For brevity let $\pi = 
\Phi^{(n)}_{\chi} \circ f$ and work near $p_{\infty} \in \pa D$, the orbit 
accumulation point. By lemma 2.1 there is a two dimensional stratum $S$ of 
the Levi degenerate points that clusters at $p_{\infty}$ and $S \subset 
\hat D$. Choose $a \in S$ near $p_{\infty}$ and fix a neighbourhood $U$ of 
$a$ small enough so that $\pi$ extends holomorphically to $U$. Note that 
$(U \cap \pa D) \sm S$ consists of either strongly pseudoconvex or strongly 
pseudoconcave points. Since $\pa D$ is of finite type near $p_{\infty}$ and 
$\pa \mathfrak{S}^{\pm}_{s, t}$ are strongly pseudoconvex/pseudoconcave 
everywhere, the invariance property of Segre varieties shows that $\pi$ is 
proper near $a$. Now suppose that $\pi(a) \in \pa \mathfrak{S}^+_{s, 
t}$. Since $\pi : D \ra \mathfrak{S}_{s, t}$ is proper it follows that 
$\pi(U \cap \pa D) \subset \pa \mathfrak{S}^+_{s, t}$. But then $\pi$ is 
proper near $a$ as well and hence there is an open dense set of strongly 
pseudoconcave points in $(U \cap \pa D) \sm S$ that are mapped locally 
biholomorphically to points in $\pa \mathfrak{S}^+_{s, t}$ and this 
contradicts the invariance of the Levi form. Hence $\pi(a) \not \in \pa 
\mathfrak{S}^+_{s, t}$. The same argument shows that $\pi(a) \not \in \pa 
\mathfrak{S}^-_{s, t}$.

The only possibility then is that there are no pseudoconcave points near 
$p_{\infty}$, i.e., the boundary $\pa D$ is weakly pseudoconvex near 
$p_{\infty}$. In this case it follows from \cite{Ber1} that $D$ is 
equivalent to the model domain at $p_{\infty}$ which means that
\begin{equation}
D \backsimeq \ti D = \Big\{ (z_1, z_2) \in \mbf C^2 : 2 \Re z_2 + 
P_{2m}(z_1, \ov z_1) < 0 \Big\}
\end{equation}
where $P_{2m}(z_1, \ov z_1)$ is a homogeneous subharmonic polynomial of 
degree $2m$ (this being the $1$-type of $\pa D$ at $p_{\infty}$) without 
harmonic terms. In particular $D$ is globally pseudoconvex and since $\pi 
: D \ra \mathfrak{S}_{s, t}$ is proper it follows that $\mathfrak{S}_{s, 
t}$ is also pseudoconvex. This evidently is not the case.

\medskip

\noindent Now suppose that $0 = s < t < \infty$. The two components of the 
boundary of $\mathfrak{S}_{0, t}$ are
\[
\pa \mathfrak{S}^+_{0, t} = \{(z_1, z_2) \in \mbf C^2 : (\Re z_1)^2 + (\Re 
z_2)^2 = t\}, \; {\rm and}\; i \mbf R^2 = \{\Re z_1 = \Re z_2 = 0\}.
\]
Choose neighbourhoods $U, U'$ of $a, \pi(a)$ respectively so that $\pi : U 
\ra U'$ is a well defined holomorphic mapping. Suppose that $\pi(a) \in i 
\mbf R^2$. Let $Z_{\pi} \subset U$ be the closed analytic set defined by 
the vanishing of the Jacobian determinant of $\pi$. Note that if it is 
non-empty, $Z_{\pi} \subset U \sm \ov D$ since $\pi : D \ra 
\mathfrak{S}_{0, t}$ is unbranched. Since $\pa D$ is of finite type near 
$p_{\infty}$ it follows that no open piece of $Z_{\pi}$ can be contained 
in $\pa D$. Then there are points in $U \cap \pa D$ that are mapped 
locally biholomorphically to points in $i \mbf R^2$. This is not possible. 
For similar reasons $\pi(a) \not \in \pa \mathfrak{S}^+_{0, t}$. Hence the 
only possibility is that $\pa D$ is weakly pseudoconvex near $p_{\infty}$. 
As before it follows from \cite{Ber1} that $D \backsimeq \ti D$ where $\ti 
D$ is as in (4.2). Let $\pi$ still denote the proper mapping
\[
\pi : \ti D \ra \mathfrak{S}_{0, t}.
\]
\noindent {\it Claim:} There exists a point on $\pa \ti D$ whose cluster 
set under $\pi$ intersects $\pa \mathfrak{S}^+_{0, t}$.

\medskip

\noindent This may be proved as follows by a suitable adaption of the 
ideas that were used in \cite{CP}. Pick $a' \in \pa \mathfrak{S}^+_{0, t}$ 
and let $\psi$ be a local holomorphic peak function at $a'$. Choose a 
neighbourhood $U'$ of $a'$ such that $\psi(a') = 1$ and $\vert \psi 
\vert < 1$ on $(\ov U' \cap \ov{\mathfrak{S}}_{0, t}) \sm \{a'\}$. Then 
$\vert \psi \vert \le 1 - 2\alpha$ on $\pa U' \cap \ov{\mathfrak{S}}_{0, 
t}$ for some $\alpha > 0$. Choose $b' \in U'$ such that $\vert \psi(b') 
\vert = 1 - \alpha$ and let $b \in \ti D$ be such that $\pi(b) = b'$. By 
proposition 1.1 of \cite{CP} there is a smooth analytic disc $h : 
\ov {\Delta}(0, r) \ra \ti D$ such that $h(0) = b$ and $h(\pa \Delta(0, 
r)) \subset \pa \ti D$. Let $E \subset \Delta(0, r)$ be the connected 
component of $(\pi \circ h)^{-1}(U' \cap \mathfrak{S}_{0, t})$ that 
contains the origin. If $E$ is compactly contained in $\Delta(0, r)$ then 
$\psi \circ \pi \circ h$ is subharmonic on $\ov E$ and $\vert \psi \circ 
\pi \circ h \vert \le 1 - 2 \alpha$ on $\pa E$ while
\[
\vert \psi \circ \pi \circ h(0) \vert = \vert \psi(b') \vert = 1 - \alpha
\]
which contradicts the maximum principle. So there is a sequence $\la_j \in 
E$ that converges to $\la_0 \in \pa \Delta(0, r)$. Then $\pi \circ 
h(\la_j) \subset U' \cap \mathfrak{S}_{0, t}$ for all $j$ and if $h(\la_j) 
\ra h(\la_0) \in \pa \ti D$ it follows that the cluster set of $h(\la_0)$ 
under $\pi$ contains a point in $\ov U' \cap \pa \mathfrak{S}_{0, t}$. 

\medskip

It is now known that $\pi$ extends continuously up to $\pa \ti 
D$ near $h(\la_0)$. This extension is even locally biholomorphic across 
strongly pseudoconvex points that are known to be dense on $\pa \ti D$ by 
\cite{PT}. Now \cite{We} shows that $\pi$ is algebraic. Denote the 
coordinates in the domain and the range of $\pi$ by $z = (z_1, z_2)$ and 
$w = (w_1, w_2)$ respectively. Then there are polynomials
\begin{align*}
P_1(z_1, w_1, w_2) &= a_k(w_1, w_2) z_1^k + a_{k - 1}(w_1, w_2) z_1^{k - 
1} + \ldots + a_0(w_1, w_2),\\
P_2(z_2, w_1, w_2) &= b_l(w_1, w_2) z_2^l + b_{l - 1}(w_1, w_2) z_2^{l - 
1} + \ldots + b_0(w_1, w_2)
\end{align*}
of degree $k, l$ respectively and the coefficients $a_{\mu}, b_{\nu}$ are 
polynomials in $w_1, w_2$, with the property that $P_1(z_1, w_1, w_2) = 
P_2(z_2, w_1, w_2) = 0$ whenever $\pi^{-1}(w) = z$. Away from an algebraic 
variety $V_1 \subset \mbf C^2_w$ of dimension one, the zero locus of these 
polynomials defines a correspondence $\hat F : \mbf C^2_w \ra \mbf C^2_z$ 
such that ${\rm Graph}(\pi^{-1}) \subset \mathfrak{S}_{0, t} \times \ti D$ 
is an irreducible component of ${\rm Graph}(\hat F) \cap (\mathfrak{S}_{0, 
t} \times \ti D)$. Let $V_2 \subset \mbf C^2_w$ be the branching locus of 
$\hat F$ which is again of dimension one. Note that near each point of 
$\mbf C^2_w \sm (V_1 \cup V_2)$ the correspondence $\hat F$ splits as the 
union of locally well defined holomorphic functions. Also observe that the 
real dimension of $i \mbf R^2 \cap (V_1 \cup V_2)$ is at most one. So pick 
$a' \in i \mbf R^2 \sm (V_1 \cup V_2)$ and choose a neighbourhood $U'$ of 
$a'$ that does not intersect $V_1 \cup V_2$. All branches of $\hat F$ are 
then well defined on $U'$ and some of these are exactly the various 
branches of $\pi^{-1}$. Let $\pi^{-1}_1$ be one such branch of $\pi^{-1}$ 
that is holomorphic on $U'$. Since $\pi$ is globally proper, $\pi^{-1}_1$ 
maps $U' \cap i \mbf R^2$ into $\pa \ti D$. The branch locus of 
$\pi^{-1}_1$, if non-empty has dimension at most one since $\pi$ is 
unbranched. Therefore the intersection of the branch locus with $i \mbf 
R^2$ has real dimension at most one. Hence it is possible to find points 
on $i \mbf R^2$ near $a'$ that are mapped locally biholomorphically by 
$\pi^{-1}_1$ to points on $\pa \ti D$. This cannot hold as $\pa \ti D$ is 
not totally real. 
\end{proof}

\vskip10pt

\noindent Next consider the infinite sheeted covering
\[
\Phi^{(\infty)}_{\chi} : \mbf C^2 \ra \mbf C^2 \sm \{\Re z_1 = \Re z_2 = 
0\}
\]
whose components are given by
\begin{align*}
\ti z_1 &= \exp(\Re z_1) \cos(\Im z_1) + i \Re z_2,\\
\ti z_2 &= \exp(\Re z_1) \sin(\Im z_1) + i \Im z_2.
\end{align*}
The domain of $\Phi^{(\infty)}_{\chi}$ is equipped with the pull-back 
complex structure and the resulting complex manifold is denoted by 
$M^{(\infty)}_{\chi}$. For $0 \le s < t < \infty$ define
\[
\mathfrak{S}^{(\infty)}_{s, t} = \Big\{(z_1, z_2) \in M^{(\infty)}_{\chi} 
: (\ln s)/2 < \Re z_1 < (\ln t)/2 \Big\}
\]
and this is seen to be an infinite sheeted covering of $\mathfrak{S}_{s, 
t}$, the holomorphic covering map being $\Phi^{(\infty)}_{\chi}$.

\begin{prop}
D cannot be equivalent to $\mathfrak{S}^{(\infty)}_{s, t}$ for $0 \le s < 
t < \infty$.
\end{prop}

\begin{proof}
Let $f : D \ra \mathfrak{S}^{(\infty)}_{s, t}$ be a biholomorphism. As 
before 
\[
\Phi^{(\infty)}_{\chi} \circ f : D \ra \mathfrak{S}_{s, t}
\]
is then a holomorphic infinite sheeted covering (in particular it 
is non-proper) between domains with the standard complex structure. The 
explicit description of $\Phi^{(\infty)}_{\chi}$ shows that 
$\Phi^{(\infty)}_{\chi} \Big(\pa \mathfrak{S}^{(\infty)}_{s, t} \Big) 
\subset \pa \mathfrak{S}_{s, t}$ and this implies that the cluster set of 
$\pa D$ under $\pi = \Phi^{(\infty)}_{\chi} \circ f$ is contained in 
$\pa \mathfrak{S}_{s, t}$. This observation is an effective replacement 
for the properness of $\pi$ in the previous proposition.

\medskip

In case $0 < s < t < \infty$, this observation allows the use of 
the arguments in the previous proposition and they show that
\[
\pi : D \backsimeq \ti D \ra \mathfrak{S}_{s, t}
\]
is a covering with $\ti D$ as in (4.2). Note that the Kobayashi metric on 
$\ti D$ is complete and hence the same holds for the base 
$\mathfrak{S}_{s, t}$. Completeness then forces $\mathfrak{S}_{s, t}$ to 
be pseudoconvex which however is not the case. Now suppose that $0 = s < t 
< \infty$. Again the arguments used before apply, including the claim made 
there. So there are points on $\pa \ti D$ whose cluster set under the 
covering map $\pi$ intersects $\pa \mathfrak{S}^+_{0, t}$. It follows from 
\cite{Su} that $\pi$ extends continuously up to $\pa \ti D$ near such 
points. Moreover the extension is locally proper near $\pa \ti D$. The 
rest of the proof proceeds exactly as in the previous proposition.
\end{proof}

\subsection{A domain in $\mbf P^2$:} Let $\cal{Q}_+ \subset \mbf C^3$ 
be the smooth complex analytic set given by
\[
z_0^2 + z_1^2 + z_2^2 = 1
\]
and for $t \ge 0$ consider the spheres
\[
\Sigma_t = \Big\{ (z_0, z_1, z_2) \in \mbf C^3 : \vert z_0 \vert ^2 + 
\vert z_1 \vert^2 + \vert z_2 \vert^2 = t \Big\}.
\]
Note that $\Sigma_t \cap \cal{Q}_+ = \emptyset$ for $0 \le t < 1$. If 
$t > 1$ it can be checked that $\Sigma_t$ intersects $\cal{Q}_+$ 
transversally everywhere. On the other hand $\Sigma_1 \cap \cal{Q}_+ 
= \mbf R^3 \cap \cal{Q}_+$ which is totally real in $\mbf C^3$.\\

\noindent For $1 \le s < t < \infty$ define
\[
E^{(2)}_{s, t} = \Big\{(z_0, z_1, z_2) \in \mbf C^3 : s < \vert z_0 
\vert^2 + \vert z_1 \vert^2 + \vert z_2 \vert^2 < t \Big\} \cap \cal Q_+.
\]
This domain is a $2$-sheeted covering of
\[
E_{s, t} = \Big\{ [z_0 : z_1 : z_2] \in \mbf P^2 : s \; \vert z_0^2 + 
z_1^2 + z_2^2 \vert < \vert z_0 \vert^2 + \vert z_1 \vert^2 + 
\vert z_2 \vert^2 < t \; \vert z_0^2 + z_1^2 + z_2^2 \vert \Big\}
\]
via the natural map $\psi(z_0, z_1, z_2) = [z_0 : z_1 : z_2]$. In the same 
way, for $1 < t < \infty$ the domain
\[
E_t =  \Big\{ [z_0 : z_1 : z_2] \in \mbf P^2 : \vert z_0 \vert^2 + \vert 
z_1 \vert^2 + \vert z_2 \vert^2 < t \; \vert z_0^2 + z_1^2 + z_2^2 \vert 
\Big\}
\]
is covered in a $2:1$ manner by
\[
E^{(2)}_t = \Big \{(z_0, z_1, z_2) \in \mbf C^3 : \vert z_0 \vert^2 + 
\vert z_1 \vert^2 + \vert z_2 \vert^2 < t \Big\} \cap \cal Q_+
\]
by the same map $\psi$. Note that $E_{1, t} \cup \psi(\Sigma_1 \cap \cal 
Q_+) = E_t$ for all $1 < t < \infty$.\\

\noindent Consider the map $\Phi_{\mu} : \mbf C^2 \sm \{0\} \ra \cal Q_+$ 
given by
\begin{align*}
\ti z_1 &= -i(z_1^2 + z_2^2) + i (z_1 \ov z_2 - \ov z_1 z_2)/(\vert z_1 
\vert^2 + \vert z_2 \vert^2)\\
\ti z_2 &= z_1^2 - z_2^2 - (z_1 \ov z_2 + \ov z_1 z_2)/(\vert z_1 \vert^2 
+ \vert z_2 \vert^2)\\
\ti z_3 &= 2 z_1 z_2 + (\vert z_1 \vert^2 - \vert z_2 \vert^2)/(\vert z_1 
\vert^2 + \vert z_2 \vert^2)
\end{align*}
which is a two sheeted covering onto $\cal Q_+ \sm \mbf R^3$. The domain 
of $\Phi_{\mu}$ is now equipped with the pull back complex structure using 
$\Phi_{\mu}$ and the resulting complex manifold is denoted by 
$M^{(4)}_{\mu}$. For $1 \le s < t < \infty$ define
\[
E^{(4)}_{s, t} = \Big \{(z_1, z_2) \in M^{(4)}_{\mu} : \sqrt{(s - 1)/2} < 
\vert z_1 \vert^2 + \vert z_2 \vert^2 < \sqrt{(t - 1)/2} \Big\}.
\]
Note that $E^{(4)}_{s, t}$ is a four-sheeted cover of $E_{s, t}$, the 
covering map being $\psi \circ \Phi_{\mu}$.

\begin{prop}
There cannot exist a proper holomorphic mapping between $D$ and $E_{s, 
t}$ for $1 \le s < t < \infty$. In particular $D$ cannot be equivalent to 
either $E^{(2)}_{s, t}$ or $E^{(4)}_{s, t}$ for $1 \le s < t < \infty$.
\end{prop}

\begin{proof}
Let $f : D \ra E_{s, t}$ be a proper holomorphic mapping. First consider 
the case when $1 < s < t < \infty$. The boundary of $E_{s, t}$ has two 
components which are covered in a 2:1 manner by $\Sigma_t \cap \cal Q_+$ 
and $\Sigma_s \cap \cal Q_+$ respectively using $\psi$. Let
\begin{align*}
\pa E^+_{s, t} &= \Big\{[z_0 : z_1 : z_2] \in \mbf P^2 : \vert z_0 \vert^2 
+ \vert z_1 \vert^2 + \vert z_2 \vert^2 = t \vert z_0^2 + z_1^2 + z_2^2 
\vert \Big\} = \psi(\Sigma_t \cap \cal Q_+), \; {\rm and}\\
\pa E^-_{s, t} &= \Big\{[z_0 : z_1 : z_2] \in \mbf P^2 : \vert z_0 
\vert^2 + \vert z_1 \vert^2 + \vert z_2 \vert^2 = t \vert z_0^2 + z_1^2 + 
z_2^2 \vert \Big\} = \psi(\Sigma_s \cap \cal Q_+)
\end{align*}
which are strongly pseudoconvex and strongly pseudoconcave hypersurfaces 
respectively. Once again the structure of $\pa D$ near $p_{\infty}$ can be 
exploited exactly as in proposition 4.7, the conclusion being that $\pa D$ 
must be weakly pseudoconvex near $p_{\infty}$. By \cite{Ber1} it follows 
that $D \backsimeq \ti D$ where $\ti D$ is as in (4.2). Thus
\[
f : D \backsimeq \ti D \ra E_{s, t}
\]
is proper. Now all branches of $f^{-1}$ will extend across $\pa E^-_{s, 
t}$ and the extension will map $\pa E^-_{s, t}$ into $\pa \ti D$. This 
implies that there are strongly pseudoconcave points on $\pa \ti D$ which 
is not possible.

\medskip

Now suppose that $1 = s < t < \infty$. Then
\[
\pa E_{1, t} = \pa E^+_{1, t} \cup \psi(\Sigma_1 \cap \cal Q_+)
\]
where $\psi(\Sigma_1 \cap \cal Q_+)$ is maximally totally real. Let $\phi$ 
be a holomorphic function on $\ti D$ that peaks at the point at infinity 
in $\pa \ti D$ (see \cite{BF} for example) and denote by $f^{-1}_1, 
f^{-1}_2, \ldots, f^{-1}_l$ the locally defined branches of $f^{-1}$ that 
exist away from a closed analytic set of dimension one in $E_{1, t}$. Then
\[
\ti \phi = (\phi \circ f^{-1}_1) \cdot (\phi \circ f^{-1}_2) \cdot \ldots 
\cdot (\phi \circ f^{-1}_l)
\]
is a well defined holomorphic function on $E_{1, t}$ such that $\vert \ti 
\phi \vert < 1$ there.

\medskip

\noindent {\it Claim:} For each $p' \in \psi(\Sigma_1 \cap \cal Q_+)$ 
there exists $p \in \pa \ti D$ such that the cluster set of $p$ under $f$ 
contains $p'$.

\medskip

\noindent Indeed it has been noted that $E_{1, t} \cup \psi(\Sigma_1 \cap 
\cal Q_+) = E_t$ and that $\psi(\Sigma_1 \cap \cal Q_+)$ has real 
codimension two. Therefore $\ti \phi \in \cal O(E_{1, t})$ extends 
holomorphically to $E_t$, and the extension that is still denoted by $\ti 
\phi$ satisfies $\vert \ti \phi \vert \le 1$ there. If $\vert \ti \phi(p') 
\vert = 1$ for some $p' \in \psi(\Sigma_1 \cap \cal Q_+)$ then the maximum 
principle implies that $\vert \ti \phi \vert \equiv 1$ on $E_{1, t} 
\subset E_t$ and this is a contradiction. The claim follows.

\medskip

The claim made in proposition 4.7 holds here as well; indeed the only 
property in the range space that is used is the existence of a local 
holomorphic peak function and these exist near each point of $\pa E^+_{1, 
t}$. Therefore there are points $a \in \pa \ti D$ and $a' \in \pa E^+_{1, 
t}$ such that the cluster set of $a$ under $f$ contains $a'$. Working in a 
coordinate system around $a'$ it is known that (see for example \cite{Su} 
or \cite{CPS}) for sufficiently small neighbourhoods $U, U'$ around 
$a, a'$ respectively, 
\begin{equation}
{\rm dist}(f(z), U' \cap \pa E^+_{1, t}) \lesssim {\rm dist}(z, U \cap \pa 
\ti D)
\end{equation}
whenever $z \in U \cap \ti D$ is such that $f(z) \in U' \cap E^+_{1, 
t}$. Using the estimates on the infinitesimal Kobayashi metric near $a, a'$ 
it follows that $f$ extends continuously up to $\pa \ti D$ near $a$ and $f(a) 
= a'$. Therefore $f$ extends locally biholomorphically across all the 
strongly pseudoconvex points near $a$. Let $w = (w_1, w_2)$ denote affine 
coordinates near $a'$ and $f^{-1}_1$ be a branch of $f^{-1}$ that extends 
locally biholomorphically across near $a'$. Then by \cite{We} the graph of 
$f^{-1}_1$ near $(a, a')$ is contained in the zero locus of the 
polynomials $P_1(z_1, w_1, w_2), P_2(z_2, w_1, w_2)$ as in proposition 4.7 
and where $f^{-1}_1(w) = (z_1, z_2) = z$. Projectivize these polynomials 
by replacing $w_i \mapsto w_i / w_0, z_i \mapsto z_i / z_0$ for $i = 1, 2$ 
to get an algebraic variety $A \subset \mbf P^2_z \times \mbf P^2_w$ such 
that ${\rm Graph}(f)$ is an irreducible component of $A \cap (\ti D \times 
E_{1, t})$. Away from an algebraic variety $V \subset \mbf P^2_w$ of 
dimension one, $A$ is the union of the graphs of locally well defined 
holomorphic functions. Since $\psi(\Sigma_1 \cap \cal Q_+)$ is totally 
real it follows that $\psi(\Sigma_1 \cap \cal Q_+) \cap V$ has real 
dimension at most one. Choose $p' \in \psi(\Sigma_1 \cap \cal Q_+) \sm V$ 
and a neighbourhood $U'$ of $p'$ that does not intersect $V$. Then all 
branches of the correspondence defined by $A$ are well defined in $U'$ and 
some of these will coincide with those of $f^{-1}$. By the claim above, 
there is at least one branch of $f^{-1}$, call it $\ti f^{-1}$ such that 
$\ti f^{-1}(p') \not \in H_0 = \{z_0 = 0\}$ the hyperplane at infinity in 
$\mbf P^2_z$, i.e., $\ti f^{-1}(p') \in \pa \ti D$. Therefore near $p'$ 
there are points on $\psi(\Sigma_1 \cap \cal Q_+)$ that are mapped locally 
biholomorphically by $\ti f^{-1}$ to points on $\pa \ti D$ and this is a 
contradiction since $\pa \ti D$ is not totally real.

\medskip

To conclude, if $D \backsimeq E^{(2)}_{s, t}$ or $E^{(4)}_{s, t}$ then 
this would imply the existence of an unbranched proper mapping between $D$ 
and $E_{s, t}$ and this cannot happen by the arguments given above.
\end{proof}

\begin{prop}
There cannot exist a proper holomorphic mapping between $D$ and $E_t$ for 
$1 < t < \infty$. In particular $D$ cannot be equivalent to $E^{(2)}_t$ 
for $1 < t < \infty$.
\end{prop}

\begin{proof}
Let $f : D \ra E_t$ be proper. Working in affine coordinates near each 
point in $\pa E_t$ it can be seen that $E_t$ is described by $\{\varrho(z) 
< 0\}$ where $\varrho(z) = 1 + \vert z_1 \vert^2 + \vert z_2 \vert^2 - t 
\vert 1 + z_1^2 + z_2^2 \vert$ is strongly plurisubharmonic. It follows 
that $E_t$ must be holomorphically convex (see for example \cite{Fu}) and 
thus $D$ is pseudoconvex. Choose $p_j \in D$ converging to $p_{\infty}$ and 
let $f(p_j) \ra p'_{\infty} \in \pa E_t$. As in the previous proposition, 
by working in affine coordinates around $p'_{\infty}$ it is possible to 
show that (4.3) holds near $p_{\infty}, p'_{\infty}$. Hence $f$ extends 
continuously up to $\pa D$ near $p_{\infty}$ with 
$f(p_{\infty}) = p'_{\infty}$. It follows from theorem 1.2 of \cite{CPS} 
that $\pa D$ is weakly spherical near $p_{\infty}$, i.e., the defining 
function for $\pa D$ near $p_{\infty} = 0$ has the form
\[
\rho(z) = 2 \Re z_2 + \vert z_1 \vert^{2m} + \;{\rm higher \; order\; terms}.
\]
Since $p_{\infty} \in \pa D$ is an orbit accumulation point it follows 
from \cite{Ber1} that $D$ is equivalent to the model domain at 
$p_{\infty}$, i.e., $D \backsimeq \{(z_1, z_2) \in \mbf C^2 : 2 \Re z_2 + 
\vert z_1 \vert^{2m} < 0\}$. This shows that $\dim {\rm Aut}(D) = 4$ which 
is a contradiction.

\medskip

To conclude, if $D \backsimeq E^{(2)}_t$ then there would exist an 
unbranched proper mapping between $D$ and $E_t$ and this is not possible.
\end{proof}

\vskip10pt

\subsection{Domains constructed by using an analogue of Rossi's map:} Let 
$\cal Q_- \subset \mbf C^3$ be the smooth complex analytic set given by
\[
z_1^2 + z_2^2 - z_3^2 = 1.
\]
Let $\Omega = \Big \{(z_1, z_2) \in \mbf C^2 : \vert z_1 \vert^2 - \vert 
z_2 \vert^2 \not= 0 \Big\}$ and consider the map $\Phi : \Om \ra \cal Q_-$ 
given by
\begin{align*}
\ti z_1 &= -i(z_1^2 + z_2^2) - i(z_1 \ov z_2 + \ov z_1 z_2)/(\vert z_1 
\vert^2 - \vert z_2 \vert^2),\\
\ti z_2 &= z_1^2 - z_2^2 + (z_1 \ov z_2 - \ov z_1 z_2)/(\vert z_1 \vert^2 
- \vert z_2 \vert^2),\\
\ti z_3 &= -2i z_1 z_2 - i(\vert z_1 \vert^2 + \vert z_2 \vert^2)/(\vert 
z_1 \vert^2 - \vert z_2 \vert^2).
\end{align*}

\noindent It can be checked that $\Phi(\Om) = \cal Q_- \sm (\cal Q_- \cap 
\cal W)$ where $\cal W = i \mbf R^3 \cup \mbf R^3 \cup \Big \{(z_1, z_2, 
z_3) \in \mbf C^3 \sm \mbf R^3 : \vert i z_1 + z_2 \vert = 
\vert i z_3 - 1 \vert, \; \vert i z_1 - z_2 \vert = \vert i z_3 + 1 \vert
\Big\}$. Set 
\[
\Om^> = \Big \{(z_1, z_2) \in \mbf C^2 : \vert z_1 \vert^2 - \vert z_2 
\vert^2 > 0 \Big\}
\]
and let $\Phi^>$ be the restriction of $\Phi$ to $\Om^>$. Then $\Phi^> : 
\Om ^> \ra \Phi^>(\Om^>)$ is a two sheeted covering where $\Phi^>(\Om^>)$ 
is the union of
\begin{align*}
\Sigma^{\nu} &= \Big \{(z_1, z_2, z_3) \in \mbf C^3 : -1 < \vert z_1 
\vert^2 + \vert z_2 \vert^2 - \vert z_3 \vert^2 < 1, \; \Im z_3 < 0 \Big\} 
\cap \cal Q_-,\\
\Sigma^{\eta} &= \Big \{(z_1, z_2, z_3) \in \mbf C^3 : \vert z_1 \vert^2 + 
\vert z_2 \vert^2 - \vert z_3 \vert^2 > 1, \; \Im(z_2(\ov z_1 + \ov z_3)) 
>0 \Big\} \cap \cal Q_-,\\
O_5 &= \Big\{(z_1, z_2, z_3) \in \mbf C^3 \sm \mbf R^3 : \vert i z_1 + 
z_2 \vert = \vert i z_3 + 1 \vert, \vert i z_1 - z_2 \vert = \vert i z_3 - 1 
\vert, \Im z_3 < 0 \Big\} \cap \cal Q_-.
\end{align*}
In particular if $\Phi^{\nu}, \Phi^{\eta}$ are the restrictions of 
$\Phi^>$ to $\Om^{\nu} = \Big\{(z_1, z_2) \in \mbf C^2 : 0 < \vert z_1 
\vert^2 - \vert z_1 \vert^2 < 1 \Big\}$ and 
$\Om^{\eta} = \Big \{(z_1, z_2) \in \mbf C^2 : \vert z_1 \vert^2 - \vert 
z_2 \vert^2 > 1 \Big\}$ respectively, then $\Phi^{\nu} : \Om^{\nu} \ra 
\Sigma^{\nu}$ and $\Phi^{\eta} : \Om^{\eta} \ra \Sigma^{\eta}$ are also 
two sheeted coverings. It follows from the definition that $z_3 \not= 0$ on 
$\Sigma^{\nu}$ and therefore $\Psi^{\nu} : \Sigma^{\nu} \ra \mbf C^2$ 
given by $\Psi^{\nu}(z_1, z_2, z_3) = (z_1/z_3, z_2/z_3)$ is holomorphic. 
Moreover $\Psi^{\nu}$ is injective as well on $\Sigma^{\nu}$ for if 
\[
\Psi^{\nu}(z^{\ast}_1, z^{\ast}_2, z^{\ast}_3) = (z^{\ast}_1/ z^{\ast}_3, 
z^{\ast}_2/z^{\ast}_3) = (z_1/z_3, z_2/z_3) = \Psi^{\nu}(z_1, z_2, z_3)
\]
then $z^{\ast}_1 = \la z^{\ast}_3, z_1 = \la z_3$ and $z^{\ast}_2 = \mu 
z^{\ast}_3, z_2 = \mu z_3$ for scalars $\la, \mu$. But then $z_1^2 + z_2^2 
- z_3^2 = 1 = (z^{\ast}_1)^2 + (z^{\ast}_2)^2 - (z^{\ast}_3)^2$ implies 
that $z_3 = \pm z^{\ast}_3$ and only the first alternative can hold since 
$\Im z_3$ and $\Im z^{\ast}_3$ are both positive. Therefore 
\[
D^{\nu} = \Psi^{\nu}(\Sigma^{\nu}) = \Big \{(z_1, z_2) \in \mbf C^2 : 
-\vert z_1^2 + z_2^2 - 1 \vert < \vert z_1 \vert^2 + \vert z_2 \vert^2 - 1 < 
\vert z_1^2 + z_2^2 - 1 \vert \Big \}\backsimeq \Sigma^{\nu}.
\]
Likewise note that for $-1 \le s < t \le 1$, the domain
\[
\Sigma^{\nu}_{s, t} = \Big\{(z_1, z_2, z_3) \in \mbf C^3 : s < \vert z_1 
\vert^2 + \vert z_2 \vert^2 - \vert z_3 \vert^2 < t, \; \Im z_3 < 0 \Big\} 
\cap \cal Q_- \subset \Sigma^{\nu}
\]
is mapped biholomorphically by $\Psi^{\nu}$ onto
\[
\Om_{s, t} = \Psi^{\nu}(\Sigma^{\nu}_{s, t}) = \Big\{(z_1, z_2) \in \mbf 
C^2 : s \vert z_1^2 + z_2^2 - 1 \vert < \vert z_1 \vert^2 + 
\vert z_2 \vert^2 - 1 < t \vert z_1^2 + z_2^2 - 1 \vert \Big\}.
\]
It is also possible to consider the domain
\[
\Om_t = \Big \{(z_1, z_2) \in C^2 : \vert z_1 \vert^2 + \vert z_2 \vert^2 
- 1 < t \vert z_1^2 + z_2^2 - 1 \vert \Big\}
\]
for $t \in (-1, 1)$. It has been observed in \cite{I1} that this domain 
has a unique maximally totally real ${\rm Aut}(\Om_t)^c$-orbit, namely
\[
\cal O_5 = \Big\{(\Re z_1, \Re z_2) \in \mbf R^2 : (\Re z_1)^2 + (\Re 
z_2)^2 < 1 \Big\}
\]
for all $t \in (-1, 1)$. Furthermore note that $\Om_t = \Om_{-1, t} \cup 
\cal O_5$ for all $t \in (-1, 1)$. 

\medskip

On the other hand observe that $z_1 \not= 0$ on $\Sigma^{\eta}$ and 
therefore $\Psi^{\eta} : \Sigma^{\eta} \ra \mbf C^2$ given by 
$\Psi^{\eta}(z_1, z_2, z_3) = (z_2/z_1, z_3/z_1)$ is holomorphic. A 
similar calculation shows that $\Psi^{\eta}(z_1, z_2, z_3) = 
\Psi^{\eta}(-z_1, -z_2, -z_3)$ and therefore 
\[
\Psi^{\eta} : \Sigma^{\eta} 
\ra D^{\eta} = \Psi^{\eta}(\Sigma^{\eta}) = \Big\{(z_1, z_2) \in \mbf C^2 
: \vert 1 + z_1^2 - z_2^2 \vert < 1 + \vert z_1 \vert^2 - \vert z_2 \vert^2, 
\; \Im(z_1(1 + \ov z_2)) > 0 \Big\}
\]
is a two sheeted covering. Likewise note that for $1 \le s < t \le \infty$ 
the domain
\[
\Sigma^{\eta}_{s, t} = \Big\{(z_1, z_2, z_3) \in \mbf C^3 : s < \vert z_1 
\vert^2 + \vert z_2 \vert^2 - \vert z_3 \vert^2 < t, \; \Im(z_2(\ov z_1 + 
\ov z_3)) > 0 \Big\} \cap \cal Q_- \subset \Sigma^{\eta}
\]
is a two sheeted covering of
\[
D_{s, t} = \Psi^{\eta}(\Sigma^{\eta}_{s, t}) = \Big \{(z_1, z_2) \in \mbf 
C^2 : s \vert 1 + z_1^2 - z_2^2 \vert < 1 + \vert z_1 \vert^2 - \vert z_2 
\vert^2 < t \vert 1 + z_1^2 - z_2^2 \vert, \; \Im(z_1(1 + \ov z_2)) > 0 
\Big\}.
\]
When $t = \infty$ the domain $D_{s, \infty}$ is assumed not to include 
the complex curve $\cal O = \Big \{(z_1, z_2) \in \mbf C^2 : 1 + z_1^2 - 
z_2^2 = 0, \; \Im(z_1(1 + \ov z_2)) > 0 \Big\}$. It is also possible to 
consider the domain
\[
D_s = \Big \{(z_1, z_2) \in \mbf C^2 : s \vert 1 + z_1^2 - z_2^2 \vert < 1 
+ \vert z_1 \vert^2 - \vert z_2 \vert^2, \; \Im(z_1(1 + \ov z_2)) > 0 
\Big\}
\]
for $1 \le s < \infty$. Here $\cal O$ is allowed to be in $D_s$ and 
hence $D_s = D_{s, \infty} \cup \cal O$.

\medskip

\begin{lem}
$D$ cannot be equivalent to either $\Omega_{s, t}$ or $D_{s, 
t}$ for all permissible values of $s, t$.
\end{lem}

\begin{proof}
First suppose that $D \backsimeq \Om_{s, t}$. For $-1 \le s < t \le 1$ the 
discussion above shows that
\[
\Psi^{\nu} \circ \Phi^{\nu} : \Big\{(z_1, z_2) \in \mbf C^2 : \sqrt{(s + 
1)/2} < \vert z_1 \vert^2 - \vert z_2 \vert^2 < \sqrt{(t + 1)/2} \Big\} 
\ra \Om_{s, t} \backsimeq D
\]
is a two sheeted cover. Since $D$ is simply connected it follows that 
$\Psi^{\nu} \circ \Phi^{\nu}$ is a homeomorphism which is evidently not 
the case.

\medskip

On the other hand if $1 \le s < t \le \infty$, it again follows that
\[
\Psi^{\eta} \circ \Phi^{\eta} : \Big\{(z_1, z_2) \in \mbf C^2 : \sqrt{(s +
1)/2} < \vert z_1 \vert^2 - \vert z_2 \vert^2 < \sqrt{(t + 1)/2} \Big\}
\ra D_{s, t} \backsimeq D
\]
is a four sheeted cover. As above, since $D$ is simply connected it 
follows that $\Psi^{\eta} \circ \Phi^{\eta}$ is forced to be a 
homeomorphism which is not the case.
\end{proof}

\begin{prop}
There cannot exist a proper holomorpic map from $D$ onto $\Om_t$ for $t 
\in (-1, 1)$ or to $D_s$ for $1 \le s < \infty$.
\end{prop}

\begin{proof}
First work with $\Om_t$. Write $z_1 = x + iy,\; z_2 = u + iv$ so that 
$\cal O_5 = \{(x, u) \in \mbf R^2 : x^2 + u^2 < 1\}$. Note that 
$\pa \cal O_5 = \{(x, u) \in \mbf R^2 : x^2 + u^2 = 1\} \subset \pa \Om_t$ 
for all $t \in (-1, 1)$ and that $\pa \Om_t \sm \pa \cal O_5$ is a smooth 
strongly pseudoconvex hypersurface. Let $f : D \ra \Om_t$ be proper. As in 
proposition 
4.7, we work near $p_{\infty} \in \pa D$ and choose $a \in S$, where $S$ 
is a two dimensional stratum of the Levi degenerate points that clusters 
at $p_{\infty}$ and which is contained in $\hat D$. Then $f$ extends 
holomorphically near $a$ and $f(a) \in \pa \Om_t$. Choose neighbourhoods 
$U, U'$ of $a, f(a)$ respectively so that $f : U \ra U'$ is a well 
defined holomorphic mapping with $f(U \cap \pa D) \subset U' \cap \pa 
\Om_t$. The closed complex analytic set $Z_f \subset U$ defined by 
the vanishing of the Jacobian of $f$ has dimension one and the finite type 
assumption on $\pa D$ near $p_{\infty}$ implies that the real dimension of 
$Z_f \cap \pa D$ is at most one. Therefore it is possible to choose $p \in 
(U \cap S) \sm Z_f$. Suppose that $f(p) \in \pa \Om_t \sm \pa \cal O_5$. 
Then there are strongly pseudoconcave points near $p$ that are mapped 
locally 
biholomorphically to points near $f(p)$ which however is a strongly 
pseudoconvex point. This cannot happen. The other possibility is that 
$f(p) \in \pa \cal O_5 = \{(x, u) \in \mbf R^2 : x^2 + u^2 = 1\}$ which is 
totally real. It follows that $f^{-1}(U' \cap \pa \cal O_5) \subset U \cap 
\pa D$ is nowhere dense and therefore there are strongly pseudoconcave 
points near $p$ that are mapped locally biholomorphically to strongly 
pseudoconvex points near $f(p)$. This is again a contradiction. 

\medskip

Hence the boundary $\pa D$ near $p_{\infty}$ is weakly pseudoconvex and by 
\cite{Ber1} it follows that $D \backsimeq \ti D$ where $\ti D$ is as in 
(4.2). Then $f : D \backsimeq \ti D \ra \Om_t$ is still biholomorphic. 
Observe that $\ti D$ is invariant under the one-parameter subgroups 
$T_t(z_1, z_2) = (z_1, z_2 + it)$ and $S_s(z_1, z_2) = (\exp(s/2m)z_1, 
\exp(s)z_2)$ and the corresponding real vector fields are 
$X = \Re(i \; \pa / \pa z_2)$ and $Y = \Re((z_1/2m) \; \pa / \pa z_1 + 
z_2 \pa / \pa z_2)$. It can be seen that $[X, Y] = X$. By the reasoning 
given in the last part of proposition 4.1 it follows that 
$\ti D \backsimeq \cal D_4 = \{(z_1, z_2) \in 
\mbf C^2 : 2 \Re z_2 + (\Re z_1)^{2m} < 0\}$. Let $f : \cal D_2 \ra 
\Om_t$ still denote the proper map. Choose ${\ti p}' \in \pa \Om_t \sm 
\pa \cal O_5$ a strongly pseudoconvex point. Then the claim made in 
proposition 4.7 shows the existence of $\ti p \in \pa \cal D_2$ such that 
the 
cluster set of $\ti p$ under $f$ contains ${\ti p}'$. Then $f$ will extend 
continuously up to the boundary $\pa \cal D_2$ near $\ti p$ and $f(\ti p) 
= {\ti p}'$. Now it follows from \cite{CPS} that $\ti p \in \pa \cal D_2$ 
must be a weakly spherical point, i.e., there exists a coordinate system 
around $\ti p$ such that the defining equation for $\pa \Om_2$ near 
$\ti p$ is of the form
\[
\rho(z) = 2 \Re z_2 + \vert z_1 \vert^{2m} + \; {\rm higher \; order \; 
terms}.
\]
However the explicit form of $\pa \cal D_2$ shows that no point on it is 
weakly spherical. 

\medskip

On the other hand observe that if $s > 1$ then $\pa D_s$ is the disjoint 
union of 
\begin{align*}
\cal C^1 &= \{1 + \vert z_1 \vert^2 - 
\vert z_2 \vert^2 = s \vert 1 + z_1^2 - z_2^2 \vert, \; 
\Im(z_1(1 + \ov z_2)) > 0\},\\ 
\cal C^2 &= \{1 + \vert z_1 \vert^2 - \vert z_2 
\vert^2 > s \vert 1 + z_1^2 - z_2^2 \vert, \; \Im(z_1(1 + \ov z_2)) = 
0\},\\ 
\cal C^3 &= \{1 + \vert z_1 \vert^2 - \vert z_2 \vert^2 = s \vert 1 + 
z_1^2 - z_2^2 \vert, \; \Im(z_1(1 + \ov z_2)) = 0\}. 
\end{align*}
For an arbitrary set 
$E \subset \mbf C^2$ and $e \in E$, let $E_e$ denote the germ of $E$ at 
$e$. Note that $\cal C^1$ is a strongly pseudoconvex hypersurface. Also 
$\Im(z_1(1 + \ov z_2)) = 0$ has an isolated singularity at $(z_1, z_2) = 
(0, -1)$ away from which it is a smooth Levi flat hypersurface. Observe 
that $(0, -1) \notin \cal C^2$ since $s > 1$. As above choose 
neighbourhoods $U, U'$ of $a, f(a)$ respectively and note that the real 
dimension of $Z_f \cap U$ is at most one. Pick $p \in (U \cap S) \sm Z_f$. 
Since $\cal C^1$ is strongly pseudoconvex it follows that $f(p) \notin 
\cal C^1$. Moreover $f(p) \notin \cal C^2$ since it is Levi flat and all 
points of $U \sm S$ are Levi non-degenerate. The only possibility is that 
$f(p) \in \cal C^3$. We now work near $p$. Suppose that $f((\pa D)_p) 
\subset \cal C^3_{f(p)} \subset \{ \Im(z_1(1 + \ov z_2)) = 0\}_{f(p)}$. 
The first and third sets are real analytic germs of dimension 3 and hence 
they must coincide. It follows that $f$ maps a neighbourhood of $p$ 
locally biholomorphically onto a neighbourhood of $f(p)$ on the Levi flat 
hypersurface $\{\Im(z_1(1 + \ov z_2)) = 0\}$. This cannot happen as there 
are Levi non-degenerate points near $p$. This means that there is an open 
dense set of points near $p$ that are mapped by $f$ locally 
biholomorphically into $\cal C^1$ or $\cal C^2$. Both possibilities 
however violate the invariance of the Levi form. The rest of the argument 
runs along in the same way as mentioned above.

\medskip

When $s = 1$ it was observed in \cite{I1} that $\Psi_{\eta} : D^{(2)}_1 
\ra 
D_1$ is proper holomorphic and that $D^{(2)}_1 \backsimeq \Delta^2$ the 
bidisc. The domain $D^{(2)}_1$ will be discussed later but assuming this 
for the moment, it 
follows that if $D \backsimeq D_1$, then there will exist a proper 
holomorphic mapping from $\Delta^2$ onto $D$. Since there are strongly 
pseudoconvex points on $\pa D$ near $p_{\infty}$ it follows from \cite{He} 
that such a proper mapping cannot exist. 
\end{proof}

\subsection{Finite and infinite sheeted covers of $D_{s, t}$, $\Omega_{s, 
t}$:} Let $[z_0 : z_1 : z_2 : z_3]$ and $[\zeta : z: w]$ denote 
coordinates on $\mbf P^3$ and $\mbf P^2$ respectively where $\{z_0 = 0\}$ 
and $\{\zeta = 0\}$ are to be regarded as the respective hyperplanes at 
infinity. Let $\cal Q_- \subset \mbf P^3$ be the smooth projective variety 
given by $z_0^2 = z_1^2 + z_2^2 - z_3^2$ and set 
$\Sigma = \{[\zeta : z : w] \in \mbf P^2 : \vert w \vert < \vert z 
\vert\}$. For each integer $n \ge 2$ consider the map $\Phi^{(n)} : \Sigma 
\ra \cal Q_-$ whose components are given by
\begin{align*}
z_0 &= {\zeta}^n,\\
z_1 &= -i(z^n + z^{n - 2}w^2) - i(z \ov w + w \ov z) {\zeta}^n/(\vert z 
\vert^2 - \vert w \vert^2),\\
z_2 &= z^n - z^{n - 2} w^2 + (z \ov w - w \ov z) {\zeta}^n/(\vert z 
\vert^2 - \vert w \vert^2),\\
z_3 &= -2i z^{n - 1} w - i(\vert z \vert^2 + \vert w \vert^2) 
{\zeta}^n/(\vert z \vert^2 - \vert w \vert^2).
\end{align*}
\noindent Define
\begin{align*}
\cal A^{(n)}_{\nu} &= \Big\{(z, w) \in \mbf C^2 : 0 < \vert z \vert^n - 
\vert z \vert^{n - 2} \vert w \vert < 1 \Big\},\\
\cal A^{(n)}_{\eta} &= \Big\{(z, w) \in \mbf C^2 : \vert z \vert^n - \vert 
z \vert^{ n - 2} \vert w \vert^2 > 1 \Big\}
\end{align*}
and regard them as subsets of the affine part of $\mbf P^2$ where $\zeta = 
1$. Then $\cal A^{(n)}_{\nu}, \cal A^{(n)}_{\eta} \subset \Sigma$ for all 
$n \ge 2$. Let $\Phi^{(n)}_{\nu}, \Phi^{(n)}_{\eta}$ be the restrictions 
of $\Phi^{(n)}$ to $\cal A^{(n)}_{\nu}, \cal A^{(n)}_{\eta}$ respectively. 
Then these maps are $n$-sheeted covering maps from $\cal A^{(n)}_{\nu}, 
\cal A^{(n)}_{\eta}$ onto 
\begin{align*}
\cal A_{\nu} &= \Big\{(z_1, z_2, z_3) \in \mbf C^3 : - 1 < \vert z_1 
\vert^2 + \vert z_2 \vert^2 - \vert z_3 \vert^2 < 1, \; \Im z_3 < 0 \Big\} 
\cap \cal Q_-,\\
\cal A_{\eta} &= \Big \{(z_1, z_2, z_3) \in \mbf C^3 : \vert z_1 \vert^2 + 
\vert z_2 \vert^2 - \vert z_3 \vert^2 > 1, \; \Im(z_2(\ov z_1 + \ov z_3)) 
> 0 \Big\} \cap \cal Q_-
\end{align*}
respectively. Equip $\cal A^{(n)}_{\nu}, \cal A^{(n)}_{\eta}$ with the 
pull back complex structures using $\Phi^{(n)}_{\nu}, \Phi^{(n)}_{\eta}$ 
respectively and call the resulting complex manifolds $M^{(n)}_{\nu}, 
M^{(n)}_{\eta}$ respectively. For $-1 \le s < t \le 1$ and $n \ge 2$ 
define
\[
\Om^{(n)}_{s, t} = \Big\{(z, w) \in M^{(n)}_{\nu} : \sqrt{(s + 1)/2} < 
\vert z \vert^n - \vert z \vert^{n - 2} \vert w \vert^2 < \sqrt{(t + 1)/2} 
\Big\}.
\]
Moreover if $\Psi_{\nu} : \cal A_{\nu} \ra \mbf C^2$ is given by 
$(z_1, z_2, z_3) \mapsto (z_1/z_3, z_2/z_3)$ then 
$\Psi_{\nu} \circ \Phi^{(n)}_{\nu} : \Om^{(n)}_{s, t} \ra \Om_{s, t}$ is 
an $n$-sheeted cover.

\medskip

On the other hand, let $\Lambda : \mbf C \times \Delta \ra \Sigma \cap \{ 
\zeta = 1\}$ be the covering map given by $\Lambda(z, w) = (e^z, we^z)$ 
where $(z, w) \in \mbf C \times \Delta$. Define
\begin{align*}
U_{\nu} &= \Big\{(z, w) \in \mbf C^2 : \vert w \vert < 1, \; \exp(2 \Re z) 
(1 - \vert w \vert^2) < 1 \Big\},\\
U_{\eta} &= \Big \{(z, w) \in \mbf C^2 : \vert w \vert < 1, \; \exp(2 \Re 
z) (1 - \vert w \vert^2) > 1 \Big\}
\end{align*}
and denote by $\Lambda_{\nu}, \Lambda_{\eta}$ the restrictions of 
$\Lambda$ to $U_{\nu}, U_{\eta}$ respectively. Then $\Lambda_{\nu} : 
U_{\nu} \ra M^{(2)}_{\nu}$ and $\Lambda_{\eta} : U_{\eta} \ra 
M^{(2)}_{\eta}$ are infinite coverings. Equip $U_{\nu}, U_{\eta}$ with 
the pull back complex structures using $\Lambda_{\nu}, \Lambda_{\eta}$ 
respectively and call the resulting complex manifolds $M^{(\infty)}_{\nu}$ 
and $M^{(\infty)}_{\eta}$ respectively. For $-1 \le s < t \le 1$ define
\[
\Om^{(\infty)}_{s, t} = \Big \{(z, w) \in M^{(\infty)}_{\nu} : \sqrt{(s + 
1)/2} < \exp(2 \Re z) (1 - \vert w \vert^2) < \sqrt{(t + 1)/2} \Big\}
\]
and note that $\Psi_{\nu} \circ \Phi^{(2)}_{\nu} \circ \Lambda_{\nu} : 
\Om^{(\infty)}_{s, t} \ra \Om_{s, t}$ is an infinite covering.

\begin{prop}
D cannot be equivalent to either $\Om^{(n)}_{s, t}$ for $n \ge 2$ or 
$\Om^{(\infty)}_{s, t}$ for all $-1 \le s < t \le 1$.
\end{prop}

\begin{proof}
The reasoning for $\Om^{(n)}_{s, t}$ is subsumed by that for 
$\Om^{(\infty)}_{s, t}$ and so it will suffice to focus on 
$\Om^{(\infty)}_{s, t}$. Let $f : D \ra \Om^{(\infty)}_{s, t}$ be a 
biholomorphism. Then $\pi = \Psi_{\nu} \circ \Phi^{(2)}_{\nu} \circ 
\Lambda_{\nu} \circ f : D \ra \Om_{s, t}$ is then a holomorphic covering 
between domains with the standard complex structure. Note that the cluster 
set of $\pa D$ under $\pi$ is contained in $\pa \Om_{s, t}$. First suppose 
that $-1 < s < t < 1$. Write $z_1 = x + iy, z_2 = u + iv$. The smooth 
hypersurfaces $\nu_t = \{\vert z_1 \vert^2 + \vert z_2 \vert^2 - 1 = t 
\vert z_1^2 + z_2^2 - 1 \vert\} \sm \{x^2 + u^2 = 1\}$ and $\nu_s = 
\{\vert z_1 \vert^2 + \vert z_2 \vert^2 - 1 = s \vert z_1^2 + z_2^2 - 1 
\vert\} \sm \{x^2 + u^2 = 1\}$ are strongly pseudoconvex and strongly 
pseudoconcave pieces respectively of $\pa \Om_{s, t}$. As in proposition 
4.7 choose $a \in S \subset T$ such that $p_{\infty} \in \ov S$ and $S 
\subset \hat D$. Then $\pi$ extends holomorphically across $a$ and $\pi(a) 
\in \pa \Om_{s, t}$. Choose neighbourhoods $U, U'$ of $a, \pi(a)$ 
respectively so that $\pi : U \ra U'$ is well defined holomorphic and 
$\pi(U \cap \pa D) \subset U' \cap \pa \Om_{s, t}$. If $\pi(a) \in \nu_t$ 
then it is possible to find strongly pseudoconcave points near $a$ that 
are mapped locally biholomorphically by $\pi$ to points on $\nu_t$ which 
violates the invariance of the Levi form. Likewise $\pi(a) \notin \nu_s$. 
The remaining possibility is that $\pi(a) \in \{(x, u) \in \mbf R^2 : x^2 
+ u^2 = 1\} = \pa \cal O_5$. Let $Z_{\pi} \subset U$ be the closed 
analytic set defined by the vanishing of the Jacobian determinant of 
$\pi$. Since $\pi$ is a covering it follows that $Z_{\pi} \cap (U \cap D) 
= \emptyset$ and hence $\dim Z_{\pi} \le 1$. Since $\pa D$ is of finite 
type near $p_{\infty}$ it follows that the real dimension of $Z_{\pi} \cap 
\pa D$ is at most one. Choose $p \in (U \cap S) \sm (Z_{\pi} \cap \pa D)$. 
For reasons discussed above $\pi(p) \notin \nu_t$ or $\nu_s$. Therefore 
$\pi(p) \in \pa \cal O_5$ and since $p$ is arbitrary it follows that 
$\pi((U \cap S) \sm (Z_{\pi} \cap \pa D)) \subset \pa \cal O_5$. This 
shows that an open piece of $S$ is mapped locally biholomorphically into 
$\pa \cal O_5$ and this cannot happen by dimension considerations.
The only possibility that remains is that $\pa D$ is weakly pseudoconvex 
near $p_{\infty}$. By \cite{Ber1} it follows that $D \backsimeq \ti D$ 
where $\ti D$ is as in (4.2). Since $\ti D$ is complete hyperbolic, the 
same must be true of $\Om_{s, t}$. But then completeness implies that 
$\Om_{s, t}$ must be pseudoconvex which cannot be true since all points on 
$\nu_s \subset \pa \Om_{s, t}$ are strongly pseudoconvex.

\medskip

Now suppose that $-1 = s < t < 1$. Then the boundary of $\Om_{-1, t}$ has 
a strongly pseudoconvex piece $\nu_t$ and $\cal O_5 = \{(x, u) \in \mbf 
R^2 : x^2 + u^2 < 1\}$ which is maximally totally real. For reasons discussed 
above $\pi(a) \notin \nu_t$ and therefore $\pi(a) \in \ov {\cal O}_5$. As 
mentioned above $\dim Z_{\pi} \le 1$ and hence the real dimension of 
$Z_{\pi} \cap \pa D$ is at most one. Choose $p \in (U \cap S) \sm (Z_{\pi} 
\cap \pa D)$ and note that $\pi(p) \notin \nu_t$. Hence $\pi(p) \in \ov 
{\cal O}_5$. The strongly pseudoconcave points near $p$ are then mapped 
locally biholomorphically by $\pi$ to points on $\pa \Om_{-1, t}$ near 
$\pi(p)$. However note that a neighbourhood of $\pi(p)$ on $\pa \Om_{-1, 
t}$ consists entirely of either totally real points or those that are 
strongly pseudoconvex which again leads to a contradiction. The only 
possibility is that $D$ is pseudoconvex near $p_{\infty}$ and hence that 
$D \backsimeq \ti D$. Thus
\[
\pi : D \backsimeq \ti D \ra \Om_{-1, t}
\]
is also an infinite covering. Now pick $p' \in \nu_t \subset \pa \Om_{-1, 
t}$ and note that there is a local holomorphic peak function at $p'$. The 
claim made in proposition 4.7 applies here as well and it shows the 
existence of $p \in \pa \ti D$ such that the cluster set of $p$ under 
$\pi$ contains $p'$. It follows that $\pi$ extends locally 
biholomorphically across the strongly pseudoconvex points near $p$ and 
hence that $\pi$ is algebraic. In particular, for a generic $z' \in 
\Om_{-1, t}$ the cardinality of $\pi^{-1}(z')$ must be finite which 
contradicts the fact that $\pi$ is an infinite covering map.

\medskip

Similar arguments show that $D$ cannot be equivalent to $\Om_{s, 1}$ for 
$-1 < s < 1$. Finally it has been noted in \cite{I1} that $\Om_{-1, 
1} \backsimeq \Delta^2$ whose automorphism group is six dimensional and 
thus $D$ cannot be equivalent to $\Om_{-1, 1}$.

\end{proof}

\noindent On the other hand, for $1 \le s < t \le \infty$ and $n \ge 2$ 
define
\begin{align*}
D^{(2)}_{s, t} &= \Big \{(z_1, z_2, z_3) \in \mbf C^3 : s < \vert z_1 
\vert^2 + \vert z_2 \vert^2 - \vert z_3 \vert^2 < t, \; 
\Im(z_2(\ov z_1 + \ov z_3)) > 0 \Big\} \cap \cal Q_-,\\
D^{(2n)}_{s, t} &= \Big \{(z, w) \in M^{(n)}_{\eta} : \sqrt{(s + 1)/2} 
< \vert z \vert^n - \vert z \vert^{n - 2} \vert w \vert^2 < \sqrt{(t + 
1)/2} \Big\},\\
D^{(\infty)}_{s, t} &= \Big \{(z, w) \in M^{(\infty)}_{\eta} : \sqrt{(s + 
1)/2} < \exp(2 \Re z) (1 - \vert w \vert^2) < \sqrt{(t + 1)/2} \Big\}.
\end{align*}
The domain $D^{(2)}_{s, t}$ is a two sheeted cover of $D_{s, t}$, the 
covering map being $\Psi_{\eta} : \cal A_{\eta} \ra \mbf C^2$ given by 
$(z_1, z_2, z_3) \mapsto (z_2/z_1, z_3/z_1)$. Furthermore for $n \ge 2$ it 
can be seen that $\Psi_{\eta} \circ \Phi^{(n)}_{\eta} : D^{(2n)}_{s, t} 
\ra D_{s, t}$ is a $2n$-sheeted covering while $\Psi_{\eta} \circ 
\Phi^{(2)}_{\eta} \circ \Lambda_{\eta} : D^{(\infty)}_{s, t} \ra D_{s, t}$ 
is an infinite covering. The procedure for getting an $n$-sheeted cover of 
$D_{s, t}$ for $n$ odd has been explained in \cite{I1}. Briefly put, it 
is observed that there is a 
cyclic group of order $4$ that acts freely on $D^{(4n)}_{s, t}$. The 
quotient space $D^{(n)}_{s, t} = \Pi^{(n)}\Big( D^{(4n)}_{s, t} \Big)$ 
(where $\Pi^{(n)}$ is the factorization map) is then an $n$-sheeted 
cover of $D_{s, t}$.

\begin{prop}
$D$ cannot be equivalent to $D^{(n)}_{s, t}$ for $n \ge 2$ or to 
$D^{(\infty)}_{s, t}$ for $1 \le s < t \le \infty$.
\end{prop}

\begin{proof}
It will suffice to show that $D$ cannot be equivalent to $D^{(\infty)}_{s, t}$ 
as the reasoning for $D^{(n)}_{s, t}$ is the same. So suppose that $f : D \ra 
D^{(\infty)}_{s, t}$ is biholomorphic. Then $\pi = \Psi_{\eta} \circ 
\Phi^{(2)}_{\eta} \circ \Lambda_{\eta} \circ f : D \ra D_{s, t}$ is a 
holomorphic covering between domains with the standard complex structure. 
Observe that the cluster set of $\pa D$ under $\pi$ is contained in $\pa D_{s, 
t}$. First suppose that $1 < s < t < \infty$. The smooth hypersurfaces $\eta_s = 
\{ 1 + \vert z_1 \vert^2 - \vert z_2 \vert^2 = s \vert 1 + z_1^2 - z_2^2 \vert, 
\; \Im(z_1(1 + \ov z_2)) > 0\}$ and $\eta_t = \{1 + \vert z_1 \vert^2 - \vert 
z_2 \vert^2 = t \vert 1 + z_1^2 - z_2^2 \vert, \; \Im(z_1(1 + \ov z_2)) > 0\}$ 
are strongly pseudoconvex and strongly pseudoconcave pieces respectively of $\pa 
D_{s, t}$. The other component is $\pa D_{s, t} \cap \{\Im(z_1(1 + \ov z_2)) = 
0\}$. As before choose $a \in S \subset T$ such that $\pi$ extends 
holomorphically across $a$ with $\pi(a) \in \pa D_{s, t}$. It is evident that 
$\pi(a) \notin \eta_s, \eta_t$. Likewise the arguments in proposition 4.13 show 
that $\pi(a) \notin \pa D_{s, t} \cap \{\Im(z_1(1 + \ov z_2)) = 0\}$ as well. 
Thus $\pa D$ must be weakly pseudoconvex near $p_{\infty}$. 

\medskip

Now suppose that $1 = s < t < \infty$. Again $\pi(a) \notin \eta_t$ and for 
similar reasons as above $\pi(a) \notin \pa D_{1, t} \cap \{\Im(z_1(1 + \ov 
z_2)) = 0\}$. Therefore $\pi(a) \in \eta_1$. To study this possibility note 
that $D_{1, t} \subset D_1$ and $\Psi_{\eta}^{-1} : D_1 \ra D^{(2)}_1 
\backsimeq \Delta^2$ is a proper holomorphic correspondence. The holomorphic 
correspondence (no longer proper) $\hat F = \Psi_{\eta}^{-1} \circ \pi : 
D \ra \Delta^2$ extends across $a$ (since $a \in \hat D$) and the 
branches $\hat F_1, \hat F_2$ of $\hat F$ (there are two since $\Psi_{\eta}$
is generically two sheeted) 
satisfy $\hat F_1(a), \hat F_2(a) \in \pa \Delta^2$. The branch locus of $\hat 
F$ is of dimension at most one and since $\pa D$ is of finite type near $a$, it 
follows that there are strongly pseudoconvex/pseudoconcave points near $a$ that 
are mapped locally biholomorphically by the branches of $\hat F$ into $\pa 
\Delta^2$. This cannot happen. Again $\pa D$ must be weakly pseudoconvex near 
$p_{\infty}$. 

\medskip

When $1 < s < t = \infty$, the boundary $\pa D_{1, \infty}$ contains 
$\eta_s$ as 
a strongly pseudoconvex piece, the complex curve $\cal O$ (as defined in 
subsection 4.5) and the remaining piece is $\pa D_{1, \infty} \cap 
\{\Im(z_1(1 + \ov z_2)) = 0\} \subset \{\Im(z_1(1 + \ov z_2)) = 0\}$. It 
follows from the 
reasoning used before that $\pi(a)$ cannot belong to any of these components. 
The same holds when $1 = s < t = \infty$. In any event the conclusion is that 
$\pa D$ must be weakly pseudoconvex near $p_{\infty}$. By \cite{Ber1} it 
follows 
that $D \backsimeq \ti D$ and by the last part of proposition 4.1 it is known 
that $\ti D \backsimeq \cal D_4 = \{(z_1, z_2) \in \mbf C^2 : 2 \Re z_2 + 
(\Re z_1)^{2m} < 0\}$. By assumption $D \backsimeq D^{(\infty)}_{s, t}$ 
which implies that $D^{(\infty)}_{s, t}$ must have a unique Levi flat orbit. 
However it is noted in \cite{I1} that all orbits in $D^{(\infty)}_{s, t}$ 
are strongly pseudoconvex hypersurfaces and this is a contradiction.
\end{proof}

\subsection{Miscellaneous examples:} The domain $D_s = D_{s, \infty} \cup 
\cal O$ as defined in section 4.5 can be regarded as a completion of 
$D_{s, \infty}$ by attaching the complex curve $\cal O$. Indeed, the 
action of ${\rm Aut}(D_{s, \infty})^c$ on $D_{s, \infty}$ extends to $D_s$ 
and $\cal O$ is an orbit for this extended action. The domain $D_{s, 
\infty}$ admits a finite covering by $D^{(n)}_{s, \infty} \subset 
M^{(n)}_{\eta}$ and it is natural to ask whether there is a completion of 
$D^{(n)}_{s, \infty}$ in the sense discussed above. The last set of 
examples (cf. 11(a) and 11(b)) in \cite{I1} do precisely this. Briefly 
put, 
for $1 \le s < \infty$ and $n \ge 1$ define $\cal O^{(2n)} = \{[0: z : w] 
\in \mbf P^2 : \vert w \vert < \vert z \vert\}$ and put $D^{(2n)}_s = 
D^{(2n)}_{s, \infty} \cup \cal O^{(2n)}$. It is shown that the maps 
$\Psi_{\eta}$ and $\Psi_{\eta} \circ \Phi^{(2n)}_{\eta}$ that are apriori 
defined only on $D^{(2)}_{s, \infty}$ and $D^{(2n)}_{s, \infty}$ 
respectively extend to proper holomorphic branched coverings $\Psi_{\eta} 
: D^{(2)}_s \ra D_s$ and $\Psi_{\eta} \circ \Phi^{(2n)}_{\eta} : D^{(2n)}_s 
\ra D_s$ respectively. Likewise for $n$ odd and $n \ge 3$, the 
map $\Pi^{(n)}$ extends to a proper holomorphic branched covering 
$\Pi^{(n)} : D^{(4n)}_s \ra D^{(n)}_s$ and hence $\Psi_{\eta} \circ 
\Phi^{(2n)}_{\eta} \circ (\Pi^{(n)})^{-1} : D^{(n)}_s \ra D_s$ is a proper 
holomorphic correspondence. 

\begin{lem}
$D$ cannot be equivalent to $D^{(n)}_s$ for $n \ge 1$ and $1 \le s < 
\infty$.
\end{lem}

\begin{proof}
For $n$ even the existence of a biholomorphism between $D$ and $D^{(n)}_s$ 
would imply the existence of a proper holomorphic map from $D$ onto $D_s$. 
By proposition 4.12 this is not possible. On the other hand, for $n$ odd 
there would be a proper correspondence from $D$ onto $D_s$. The proof is 
the same as in proposition 4.12 once it is realised that the branch 
locus of the proper correspondence is of dimension one and near points of 
$S \subset \hat D$ ($S$ is the two dimensional stratum of the Levi 
degenerate points that clusters at $p_{\infty}$), the intersection of the 
branch locus with $\pa D$ has real dimension at most one. So there are 
points on $S$ near which the correspondence splits into well defined 
holomorphic mappings. Working with these mappings it is possible to show 
that $\pa D$ must be weakly pseudoconvex near $p_{\infty}$ and hence $D 
\backsimeq \ti D \backsimeq \cal D_4 = \{(z_1, z_2) \in \mbf C^2 : 2 \Re 
z_2 + (\Re z_1)^2 < 0\}$ which has a Levi flat orbit. Thus $D^{(n)}_s$ would 
also have a Levi flat orbit but it has been noted in \cite{I1} that there 
are none in $D^{(n)}_s$.
\end{proof}

\no The other examples may be briefly recalled as follows. 
\begin{align*}
\mathfrak{D}^{(1)}_s &= \Psi^{-1}_{\nu}(\Om_{s, 1}) \cup D^{(2)}_1 \cup 
\cal O^{(1)}_0,  \;-1 < s < 1\\
\hat{\mathfrak{D}}^{(1)}_t &= \Psi^{-1}_{\nu}(\Om_1) \cup D^{(2)}_{1, t} 
\cup \cal O^{(1)}_0, \;1 < t < \infty\\
\mathfrak{D}^{(1)}_{s, t} &= \Psi^{-1}_{\nu}(\Om_{s, 1}) \cup D^{(2)}_{1, 
t} \cup \cal O^{(1)}_0, \; -1 \le s < 1 < t \le \infty, \; \text{where 
$s = -1$ and $t = \infty$ do not hold simultaneously}\\
\mathfrak{D}^{(n)}_s &= \Om^{(n)}_{s, 1} \cup D^{(2n)}_1 \cup \cal 
O^{(n)}_0, \; -1 < s < 1\\
\mathfrak{D}^{(n)}_{s, t} &= \Om^{(n)}_{s, 1} \cup D^{(2n)}_{1, t} \cup 
\cal O^{(n)}_0, \; -1\le s < 1 < t \le \infty, \; \text{where
$s = -1$ and $t = \infty$ do not hold simultaneously}\\
\mathfrak{D}^{(\infty)}_s &= \Om^{(\infty)}_{s, 1} \cup D^{(\infty)}_{1, 
\infty} \cup \cal O^{(\infty)}_0, \; -1 < s < 1\\
\mathfrak{D}^{(\infty)}_{s, t} &= \Om^{(\infty)}_{s, 1} \cup 
D^{(\infty)}_{1, t} \cup \cal O^{(\infty)}_0, \; -1 < s < 1 < t \le 
\infty, \; \text{where $s = -1$ and $t = \infty$ do not hold simultaneously}
\end{align*}
where $\cal O^{(1)}_0$ is a Levi flat orbit and $\cal O^{(n)}_0, \cal 
O^{(\infty)}_0$ are its $n$-sheeted and infinite covering respectively. 
The details of this construction are given in \cite{I1}, but the relevant 
point here is that all have Levi flat orbits.

\begin{prop}
$D$ cannot be equivalent to any of the domains listed above.
\end{prop}

\begin{proof}
It will suffice to show that $D$ cannot be equivalent to 
$\mathfrak{D}^{(1)}_s$ for similar arguments can be applied in all the 
other cases. If $D \backsimeq \mathfrak{D}^{(1)}_s$ for $-1 < s < 1$ then 
$D$ must have a Levi flat orbit and by proposition 4.1 it follows that 
\[
D \backsimeq \cal D_4 \backsimeq R_{1/2m, -1, 1} = \Big\{(z_1, z_2) \in 
\mbf C^2 : -(\Re z_1)^{1/2m} < \Re z_2 < (\Re z_1)^{1/2m}, \; \Re z_1 > 0 
\Big\}.
\]
The orbits in $R_{1/2m, -1, 1}$, apart from a unique Levi flat orbit 
namely $\cal O_1$, are strongly pseudoconvex hypersurfaces of the form 
$O^{R_{1/2m}}_{\alpha} = \{\Re z_2 = \alpha (\Re z_1)^{1/2m} ; \; \Re z_1 
> 0\}$ for $-1 < \alpha < 1$. On the other hand the orbits of points in 
$\Psi^{-1}_{\nu}(\Om_{s, 1}) \subset \mathfrak{D}^{(1)}_s$ are equivalent 
to $\nu_{\alpha} = \{ \vert z_1 \vert^2 + \vert z_2 \vert^2 - 1 = \alpha 
\vert z_1^2 + z_2^2 - 1\vert \} \sm \{(x, u) \in \mbf R^2 : x^2 + u^2 = 
1\}$ for $s < \alpha < 1$. The hypersurfaces $O^{R_{1/2m}}_{\alpha}$ and 
${\nu}_{\alpha}$ are not CR-equivalent (cf. \cite{I3}) and this is a 
contradiction.
\end{proof}

\section{Model domains when ${\rm Aut}(D)$ is four dimensional}

\noindent It is shown in \cite{I2} (compare with the result in \cite{KV}) 
that there are exactly 7 isomorphism classes of Kobayashi hyperbolic 
manifolds of dimension two whose automorphism group has dimension four. 
These are listed below along with 
some properties that are relevant to this discussion and the idea once 
again will be to show that $D \backsimeq \cal D_5 = \{(z_1, z_2) \in \mbf 
C^2 : 2 \Re z_2 + \vert z_1 \vert^{2m} < 1\}$, where $m \ge 2$ is an integer, 
by eliminating all other possibilities from this list. 

\medskip

\noindent $\bullet$ Let $S_r = \{z \in \mbf C^2 : r < \vert z \vert < 1 
\}$ where $0 \le r < 1$ be a spherical shell. The automorphism group of this 
domain is the unitary group $U_2$. Evidently $D$ cannot be equivalent 
to $S_r$ since ${\rm Aut}(D)$ is non-compact by assumption. Quotients 
of $S_r$ can also be obtained by realising $\mbf Z_m, m \in \mbf N$ as a 
subgroup of scalar matrices in $U_n$ and considering $S_r/ \mbf Z_m$. This 
has fundamental group $\mbf Z_m$. Clearly $D$ being simply connected cannot 
be equivalent to $S_r/\mbf Z_m$.

\medskip

\noindent $\bullet$ Define $\cal E_{r, \theta} = \{(z_1, z_2) \in \mbf C^2 
: \vert z_1 \vert < 1, \; r(1 - \vert z_1 \vert^2)^{\theta} < \vert z_2 
\vert < (1 - \vert z_1 \vert^2)^{\theta} \}$ where $\theta \ge 0, 0 \le r 
< 1$ or $\theta < 0, r = 0$.
\begin{enumerate}
\item[(i)] When $\theta > 0$ and $0 < r < 1$ the boundary of $\cal 
E_{r, \theta}$ consists of the spherical hypersurfaces $\{\vert z_1 \vert 
< 1, \; \vert z_1 \vert^2 + (\vert z_2 \vert/r)^{1/\theta} = 1\}$ and 
$\{\vert z_1 \vert < 1, \; \vert z_1 \vert^2 + \vert z_2 \vert^{1/\theta} 
= 1\}$ and the circle $\{\vert z_1 \vert = 1, \; z_2 = 0\}$. Note that 
$z_2 \not= 0$ on either of the hypersurfaces as otherwise $\vert z_1 
\vert = 1$. 

\item[(ii)] When $\theta > 0$ and $r = 0$ the domain is $\cal E_{0, 
\theta} = \{(z_1, z_2) \in \mbf C^2 : \vert z_1 \vert < 1, \; 0 < \vert 
z_2 \vert < (1 - \vert z_1 \vert^2)^{\theta}\}$. The boundary of this 
domain consists of the spherical hypersurface $\{ \vert z_1 \vert < 1, \; 
\vert z_1 \vert^2 + \vert z_2 \vert^{1/\theta} = 1 \}$ and the closed unit 
disc $\{\vert z_1 \vert \le 1, \; z_2 = 0\}$.

\item[(iii)] When $\theta = 0$ and $0 \le r < 1$ the domain is $\cal 
E_{r, 0} = \{(z_1, z_2) \in \mbf C^2 : \vert z_1 \vert < 1, \; r < \vert 
z_2 \vert < 1\}$ which is not simply connected. The same holds when 
$\theta < 0$ and $r = 0$ in which case the domain is $\cal E_{0, \theta} 
= \{(z_1, z_2) \in \mbf C^2 : \vert z_1 \vert < 1, \; 0 < \vert z_2 \vert 
< (1 - \vert z_1 \vert^2)^{\theta} \}$. Hence $D$ cannot be
equivalent to either $\cal E_{r, 0}$ or $\cal E_{0, \theta}$.
\end{enumerate}

\medskip

\noindent $\bullet$ Define $\Om_{r, \theta} = \{(z_1, z_2) \in \mbf C^2 : 
\vert z_1 \vert < 1, \; r(1 - \vert z_1 \vert^2)^{\theta} < \exp(\Re z_2) 
< (1 - \vert z_1 \vert^2)^{\theta}\}$ where $\theta = 1, 0 \le r < 1$ or 
$\theta = -1, r = 0$.
\begin{enumerate}
\item[(i)] When $\theta = 1, 0 \le r < 1$ the domain is $\Om_{r, 1} = 
\{(z_1, z_2) \in \mbf C^2 : \vert z_1 \vert < 1, \; r(1 - \vert z_1 
\vert^2) < \exp(\Re z_2) < 1 - \vert z_1 \vert^2\}$ and its boundary has 
two components $\{\vert z_1 \vert < 1, \; r(1 - \vert z_1 \vert^2) = 
\exp(\Re z_2)\}$ and $\{ \vert z_1 \vert < 1, \; 1 - \vert z_1 \vert^2 = 
\exp(\Re z_2)\}$ both of which are spherical.

\item[(ii)] When $\theta = -1, r = 0$ the domain is $\Om_{0, -1} = 
\{(z_1, z_2) \in \mbf C^2 : \vert z_1 \vert < 1, \; 0 < \exp(\Re z_2) < 
1/(1 - \vert z_1 \vert^2)^{-1}\}$. The map $(z_1, z_2) \mapsto (z_1, 
-z_2)$ transforms $\Om_{0, -1}$ biholomorphically onto $\{(z_1, z_2) \in 
\mbf C^2 : \vert z_1 \vert < 1, \; 1 - \vert z_1 \vert^2 < \exp(\Re z_2) 
\}$ whose boundary $\{ \vert z_1 \vert < 1,\; 1 = \vert z_1 \vert^2 + 
\exp(\Re z_2) \}$ is spherical.
\end{enumerate}

\medskip

\noindent $\bullet$ Define $D_{r, \theta} = \{(z_1, z_2) \in \mbf C^2 : r 
\exp(\theta \vert z_1 \vert^2) < \vert z_2 \vert < \exp(\theta \vert z_1 
\vert^2)\}$ where $\theta = 1, 0 < r < 1$ or $\theta = -1, r = 0$.
\begin{enumerate}
\item[(i)] When $\theta = 1$ and $0 < r < 1$ the domain is $D_{r, 1} = 
\{(z_1, z_2) \in \mbf C^2 : r \exp(\vert z_1 \vert^2) < \vert z_2 \vert < 
\exp(\vert z_1 \vert^2) \}$ and its boundary has two components $\{\vert 
z_2 \vert = \exp(r \vert z_1 \vert^2)\}$ and $\{\vert z_2 \vert = \exp(r 
\vert z_1 \vert^2)\}$ both of which are spherical.

\item[(ii)] When $\theta = -1$ and $r = 0$ the domain is $D_{0, -1} = 
\{(z_1, z_2) \in \mbf C^2 : 0 < \vert z_2 \vert < \exp(-\vert z_1 
\vert^2)\}$ which is mapped biholomorphically by $(z_1, z_2) \mapsto (z_1, 
1/z_2)$ onto $\{(z_1, z_2) \in \mbf C^2 : \exp(\vert z_1 \vert^2) < \vert 
z_2 \vert\}$. Note that the boundary of this is again spherical.
\end{enumerate}

\medskip

\noindent $\bullet$ Define $\mathfrak{S} = \{(z_1, z_2) \in \mbf C^2 : -1 
+ \vert z_1 \vert^2 < \Re z_2 < \vert z_1 \vert^2\}$. It can be seen 
that both boundary components of this domain are spherical.

\medskip 

\noindent $\bullet$ Define $\cal E_{\theta} = \{(z_1, z_2) \in \mbf C^2 : 
\vert z_1 \vert < 1, \; \vert z_2 \vert < (1 - \vert z_1 
\vert^2)^{\theta}\}$ where $\theta < 0$. The boundary of $\cal E_{\theta}$ 
consists of $L = \{\vert z_1 \vert = 1\} \times \mbf C_{z_2}$ which is 
Levi flat and $S = \{\vert z_1 \vert < 1, \; \vert z_2 \vert = (1 - \vert z_1 
\vert^2)^{\theta}\}$ where $\theta < 0$. Choose $p = (p_1, p_2) \in S$. 
Note that $p_2 \not= 0$ and hence the mapping $(z_1, z_2) \mapsto (z_1, 
1/z_2)$ is well defined near $p$ and maps a germ of $S$ near $p$ 
biholomorphically to a germ of $\{\vert z_1 \vert < 1, \; (1 - \vert z_1 
\vert^2)^{-\theta}  = \vert z_2 \vert\}$ which is seen to be spherical. 
Moreover $S$ viewed from within $\cal E_{\theta}$ is a strongly 
pseudoconcave point as a straightforward computation shows. 

\medskip

\noindent $\bullet$ Define $E_{\theta} = \{(z_1, z_2) \in \mbf C^2 : \vert 
z_1 \vert^2 + \vert z_2 \vert^{\theta} < 1\}$ where $\theta > 0, \theta 
\not= 2$.

\begin{prop}
$D$ cannot be equivalent to any of $\cal E_{r, \theta}$ (where $\theta > 
0, 0 \le r < 1$), $\Om_{r, \theta}$ (where $\theta = 1, 0 \le r < 1$ or 
$\theta = -1, r = 0$), $D_{r, \theta}$ (where $\theta = 1, 0 < r < 
1$ or $\theta = -1, r = 0$), $\mathfrak{S}$ or to $\cal E_{\theta}$ (where 
$\theta < 0$).
\end{prop}
\begin{proof}
The domains listed above have a common feature namely that a large piece 
of their boundary (if not all) is spherical. The argument is essentially 
the same for all the domains and it will therefore suffice to illustrate 
the reasoning in two cases, namely $\cal E_{0, \theta}$ and $\cal 
E_{\theta}$. Suppose that $f : D \ra 
\cal E_{0, \theta}$ is a biholomorphism. Note that $\psi(z) = \log \vert 
z_2 \vert$ is plurisubharmonic everywhere and its negative infinity locus 
contains $\ov \Delta = \{\vert z_1 \vert \le 1, \; z_2 = 0\} \subset \pa 
\cal E_{0, \theta}$. Fix an open neighbourhood $U$ of $p_{\infty} \in 
\pa D$. Let $\Gamma \subset U \cap \pa D$ be the set of those points 
whose cluster set is entirely contained in $\ov \Delta$. If $\Gamma$ 
contains a relatively open subset of $\pa D$ then the uniqueness theorem 
shows that $\psi \circ f$, which is plurisubharmonic on $D$, must satisfy 
$\psi \circ f \equiv -\infty$ on $D$. This is a contradiction. Therefore 
$\Gamma \subset U \cap \pa D$ is nowhere dense and hence it is possible to 
choose a strongly pseudoconcave point $p \in (U \cap \pa D) \sm \Gamma$. 
Then $f$ extends to a neighbourhood of $p$ and $f(p) \in \Sigma = \{\vert 
z_1 \vert < 1, \; \vert z_1 \vert^2 + \vert z_2 \vert^{1/\theta} = 1\}$. 
Let $g$ be a local biholomorphism defined in an open neighbourhood of 
$f(p)$ that maps the germ of $\Sigma$ near $f(p)$ to $\pa \mbf B^2$. Then 
$g \circ f$ is a biholomorphic germ at $p$ that maps the germ of $\pa D$ 
at $p$ into $\pa \mbf B^2$. By \cite{Sha} this germ of a mapping can be 
analytically continued along all paths in $U \cap \pa D$ that start at 
$p$. In particular there is an open neighbourhood $\ti U$ of $p_{\infty}$, 
a holomorphic mapping $\ti f : \ti U \ra \mbf C^2$ such that $\ti f(\ti U 
\cap \pa D) \subset \pa \mbf B^2$. This shows that $\pa D$ must be weakly 
pseudoconvex near $p_{\infty}$ and moreover $p_{\infty}$ must a weakly 
spherical point by \cite{CPS}. By \cite{Ber1} it follows that $D 
\backsimeq 
\cal D_5 = \{(z_1, z_2) \in \mbf C^2 : 2 \Re z_2 + \vert z_1 \vert^{2m} < 
0\} \backsimeq E_{2m}$ and therefore $E_{2m} \backsimeq \cal E_{0, 
\theta}$ which is a contradiction.

\medskip

On the other hand suppose that $f : D \ra \cal E_{\theta}$ is 
biholomorphic. Choose $p \in \pa D$ a strongly pseudoconcave point near 
$p_{\infty}$. Then $f$ extends to a neighbourhood of $p$ and $f(p) \in \pa 
\cal E_{\theta}$. By shifting $p$ if necessary it can be assumed that 
$f$ is locally biholomorphic near $p$. Note that $f(p) \notin L$ as $\pa 
D$ is assumed to be of finite type near $p_{\infty}$. If $f(p) \in S$ then 
we may compose 
$f$ with $g$ as above and get a germ of a holomorphic mapping from a 
neighbourhood of $p \in \pa D$ into $\pa \mbf B^2$. As above this can be 
analytically continued to get a holomorphic mapping defined in an open 
neighbourhood of $p_{\infty}$ that takes $\pa D$ into $\pa \mbf B^2$. 
This leads to the conclusion that $E_{2m} \backsimeq \cal E_{\theta}$ 
which is false. The only other possibility is that $\pa D$ is weakly 
pseudoconvex near $p_{\infty}$. By \cite{Ber1} it follows that 
$D \backsimeq \ti D$ where $\ti D$ is as in (4.2). Therefore 
$\cal E_{\theta}$ must be pseudoconvex as well. However all points on 
$S \subset \pa \cal E_{\theta}$ are pseudoconcave points.
\end{proof}

\noindent The only possibility is that $D \backsimeq E_{\theta}$ for some 
$\theta > 0, \theta \not= 2$. Clearly $E_{\theta}$ is pseudoconvex and so 
must $D$ be. Note that $\pa E_{\theta}$ is spherical except along the 
circle $\{(e^{i \alpha}, 0)\}$. Using the plurisubharmonic function 
$\psi(z)$ in exactly the same way as in proposition 5.1, it can be seen 
that there are points $p \in \pa D$ near $p_{\infty}$ such that the 
cluster set of $p$ under the biholomorphism $f : D \ra E_{\theta}$ 
contains points of $\pa E_{\theta} \sm \{(e^{i  \alpha}, 0)\}$. Then $f$ 
will extend holomorphically to an open neighbourhood of $p$ and $f(p) \in 
\pa E_{\theta} \sm \{(e^{i \alpha}, 0)\}$. As above we may compose $f$ 
with $g$ to get a germ of a holomorphic mapping from a neighbourhood of $p 
\in \pa D$ into $\pa \mbf B^2$. By continuation this will give rise to a 
map from a neighbourhood of $p_{\infty}$ on $\pa D$ into $\pa \mbf B^2$. 
It follows from \cite{CPS} that $p_{\infty}$ must be weakly spherical and 
hence \cite{Ber1} shows that $D \backsimeq E_{2m}$ where $m \ge 2$ is an 
integer.

\end{document}